%% file: main.tex
\documentclass{style}
\input{macros}

\DeclareFontFamily{U}{mathc}{}
\DeclareFontShape{U}{mathc}{m}{it}%
{<->s*[1.03] mathc10}{}
\DeclareMathAlphabet{\mathcal}{U}{mathc}{m}{it}

\definecolor{dmitri}{HTML}{AF72B0}

\begin{document}

\author{Agustina Czenky}
\address{Department of Mathematics, University of Southern California, Los Angeles, CA.}
\email{czenky@usc.edu}

\author{David Jaklitsch}
\address{Department of Mathematics, University of Oslo,
 Moltke Moes vei 35, Niels Henrik Abels hus, 0851 Oslo}
\email{dajak@uio.no}

\author{Dmitri Nikshych}
\address{Department of Mathematics and Statistics,
University of New Hampshire,  Durham, NH 03824, USA}
\email{dmitri.nikshych@unh.edu}

\author{Julia Plavnik}
\address{Department of Mathematics, Indiana University, Bloomington, Indiana, 47405, USA, \& Department of Mathematics and Data Science, Vrije Universiteit Brussel, Pleinlaan 2, 1050 Brussels, Belgium}
\email{jplavnik@iu.edu}

\author{David Reutter}
\address{Universität Hamburg, Fachbereich Mathematik, 
Bundesstraße 55, 20146 Hamburg}
\email{david.reutter@uni-hamburg.de}

\author{Sean Sanford}
\address{School of Mathematics, The University of Edinburgh, Edinburgh, UK EH9 3FD}
\email{ssanford@ed.ac.uk}

\author{Harshit Yadav}
\address{Department of Mathematics, University of Alberta, Edmonton, AB, Canada T6G 2G1}
\email{hyadav3@ualberta.ca}

\onehalfspacing

\begin{abstract}
We develop pivotal and spherical versions of graded extension theory. We define the corresponding analogues of Brauer-Picard $2$-categorical groups and realize them as fixed points of natural $\zZ$ and $\zZ/2\zZ$ $2$-categorical actions. We classify pivotal graded extensions of a pivotal tensor category by monoidal $2$-functors into the pivotal Brauer-Picard $2$-categorical group. A similar statement is proven for spherical (unimodular) tensor categories. 
We also develop an obstruction theory for determining when pivotal and spherical structures can be extended.
\end{abstract}

\subjclass[2020]{16T05, 18M15, 18M20}
\title{{Pivotal Brauer-Picard groupoids and graded extensions}}
\maketitle

\setcounter{tocdepth}{1}
\numberwithin{equation}{section}
\tableofcontents


\section{Introduction}
A pivotal structure on a tensor category $\C$ is a monoidal natural isomorphism
$\fp\colon \id_{\C} \xRightarrow{~\sim~} (-)^{**}$ between the identity endofunctor and the double-dual functor.
Such structures and their variants were introduced by many authors, including
\cite{freyd1989braided,reshetikhin1990ribbon,barrett1999spherical}.
In the semisimple (fusion) setting, a pivotal structure yields left and right categorical traces, and the classical notion of sphericality of Barrett and Westbury \cite{barrett1999spherical} asks that these traces coincide.
This condition underlies state-sum constructions such as the Turaev--Viro invariant and provides the trace formalism used in topological field theory; see \cite{muger2003subfactors} for a review.

In the non-semisimple setting, categorical traces frequently vanish (for instance on projective objects), so trace-based formulations of sphericality do not capture the structure one needs in applications.
For this reason we use the definition of sphericality due to Douglas--Schommer-Pries--Snyder \cite{douglas2018dualizable}, which is tailored to finite tensor categories: for a unimodular pivotal tensor category, sphericality is expressed as a compatibility between the pivotal structure and the canonical trivialization of the quadruple dual coming from the Radford isomorphism.
In the fusion case, this notion agrees with the trace-based notion of sphericality \cite{douglas2018dualizable}.
This sphericality condition is an important input to recent constructions of non-semisimple $3$-manifold invariants and $(2+1)$-dimensional TQFTs, for instance in \cite{costantino2023non}. Moreover, the Drinfeld centers of such spherical categories are modular tensor categories \cite{shimizu2023ribbon}.

Graded extensions by finite groups provide a systematic way to build new tensor categories from old ones.
Given a finite group $G$, a $G$-graded extension of a tensor category $\C$ is a $G$-graded tensor category $\D=\bigoplus_{g\in G}\D_g$ whose trivial component is (monoidally) equivalent to $\C$.
Such extensions are an important tool for constructing new tensor categories and for organizing symmetry phenomena, including gauging in topological phases and orbifold constructions for vertex operator algebras.
Graded extensions are classified by monoidal $2$-functors into the Brauer-Picard $2$-group $\BrPic(\C)$ \cite{etingof2010fusion,davydov2021braided}.
The purpose of this article is to construct new spherical tensor categories via the $G$-extension procedure: starting from a pivotal or spherical $\C$, we study when a given $G$-graded extension $\D$ admits a compatible pivotal or spherical structure, and we develop obstructions and classification results for such structures.
In particular, the obstruction-theoretic viewpoint gives concrete criteria for when a graded extension admits no compatible pivotal structure, which may be useful in the search for fusion categories that are not pivotal.

\subsection*{Main results}

\subsubsection*{Classification of pivotal and spherical extensions}
We introduce pivotal and spherical analogues of the Brauer--Picard $2$-group, denoted $\PivBrPic(\C)$ and $\SphBrPic(\C)$.
Their role is to classify graded extensions equipped with compatible pivotal or spherical structures by the same mechanism as in the non-pivotal theory, but with the extra structure built in.
Conceptually, an object of $\PivBrPic(\C)$ is an invertible $\C$-bimodule category together with a trivialization of its relative Serre functor (called a \emph{pivotal structure}), which is a module category analogue of the double dual functor.
Similarly, an object of $\SphBrPic(\C)$ is such a pivotal bimodule category satisfying an additional \emph{sphericality} constraint encoded by a bimodule analogue of the Radford isomorphism (in the unimodular setting), matching the notion of sphericality for finite tensor categories in \cite{douglas2018dualizable}.

We organize these refinements using fixed-point constructions.
Using relative Serre functors and the pivotal structure on $\C$, we construct a canonical monoidal $B\underline{\zZ}$-action on $\BrPic(\C)$ and identify $\PivBrPic(\C)$ with the (homotopy) fixed points of this action.
In the unimodular setting, bimodule Radford isomorphisms similarly determine a monoidal $B\underline{\Ztwo}$-action on $\BrPic(\C)$, and the spherical refinement is again recovered as fixed points:
\[
\PivBrPic(\C)\simeq \BrPic(\C)^{B\underline{\zZ}}
\qquad\text{and}\qquad
\SphBrPic(\C)\simeq \BrPic(\C)^{B\underline{\Ztwo}}.
\]
These fixed-point descriptions are compatible with the ENO classification of $G$-graded extensions and yield refined classification statements.
For any finite group $G$, we obtain equivalences of $2$-groupoids
\[
\PivExt(G,\C)\simeq \MonFun(\uuG,\PivBrPic(\C))
\;\;\text{and}\;\;
\SphExt(G,\C)\simeq \MonFun(\uuG,\SphBrPic(\C)),
\]
see Theorem~\ref{thm:piv_ext} and Theorem~\ref{thm:sph_ext}.
Equivalently, once a graded extension is fixed by a monoidal $2$-functor $\uuG\to\BrPic(\C)$, pivotal (resp.\ spherical) structures on the extension are exactly monoidal lifts of that functor to $\PivBrPic(\C)$ (resp.\ $\SphBrPic(\C)$).

\subsubsection*{A cohomological description of the lifting problem}
Let $\C$ be pivotal and let $\mathsf{F}\colon \uuG\to \BrPic(\C)$ be the monoidal $2$-functor classifying a $G$-graded extension $\D$ of $\C$.
Endowing $\D$ with a pivotal structure extending that of $\C$ is equivalent to lifting $\mathsf{F}$ along the forgetful $2$-functor $\forg\colon \PivBrPic(\C)\to \BrPic(\C)$.
We make this lifting problem explicit by extracting two cohomological obstruction classes, together with a classification of the space of solutions when the obstructions vanish. We do this both algebraically and homotopically.

The first obstruction is already visible componentwise: it measures whether each homogeneous piece $\D_g$ can be equipped with a pivotal $\C$-bimodule structure.
It is a (twisted) $1$-cocycle with values in the group of invertible objects of the Drinfeld center,
\[
O_1(\mathsf{F})\in \widetilde{H}^1\!\bigl(G,\Inv(\Z(\C))\bigr),
\]
and it vanishes precisely when all components are pivotalizable in the required bimodule sense.
Assuming $O_1(\mathsf{F})=0$, one may choose pivotal trivializations on each component, but these choices need not be compatible with the multiplication functors $\D_g\boxtimes_\C \D_h\to \D_{gh}$.
The second obstruction:
\[
O_2(\mathsf{F})\in H^2\!\bigl(G,\kk^{\times}\bigr),
\]
detects exactly this coherence issue. This cohomology class is independent of auxiliary choices and vanishes if and only if the componentwise pivotal data can be chosen monoidally, and hence assembled into a pivotal structure on $\D$ extending that of $\C$.
When both obstructions vanish, the remaining ambiguity is the expected character twist: compatible pivotal structures form a torsor over $H^1(G,\kk^{\times})$. 

In the spherical (unimodular) setting, pivotalizable implies sphericalizable, so the first obstruction introduces no new phenomenon: it is again $O_1$, now viewed in $\widetilde{H}^1\!\bigl(G,\Inv(\Z(\C))_2\bigr)$.
The essential new subtlety is the second obstruction, given by the class of $O_2$ in
\[
H^2\!\bigl(G,(\kk^\times)_2\bigr)=
\begin{cases}
H^2(G,\mathbb{Z}/2\mathbb{Z}), & \operatorname{char}(\kk)\neq 2,\\
0, & \operatorname{char}(\kk)=2.
\end{cases}
\]
Here both the $2$-cocycles and the $1$-cochains determining coboundaries take values in $(\kk^\times)_2$, so the inclusion-induced map $H^2\!\bigl(G,(\kk^\times)_2\bigr)\to H^2\!\bigl(G,\kk^\times\bigr)$ need not be injective.
A nontrivial class of $O_2$ in this kernel means that the spherical structure on $\C$ extends to a pivotal structure on the extension, but not to a spherical one.

\subsubsection*{Sphericalization}
Section~\ref{sec:sphericalization} revisits the sphericalization construction of \cite[\S7.21]{etingof2015tensor} for unimodular finite tensor categories.
We define sphericalization for invertible bimodule categories and show that it extends to a monoidal $2$-functor
\[
(-)^{\sph}\colon \BrPic(\C)\longrightarrow \SphBrPic(\C^{\sph}).
\]
Moreover, sphericalization is compatible with graded extensions: sphericalizing a $G$-graded extension corresponds, under the above classification equivalences, to post-composition with $(-)^{\sph}$.
This gives a uniform procedure for producing spherical graded extensions after passing to the sphericalization of the base category.

\subsection*{Outlook.}
The same fixed-point philosophy should persist in the presence of braiding: one expects analogous $B\underline{\zZ}$- and $B\underline{\zZ/2\zZ}$-actions on the appropriate Picard $2$-groupoids of invertible module categories and braided module categories, so that balanced and ribbon structures on graded or crossed extensions can be organized as higher fixed points and analyzed via parallel lifting and obstruction theories.

\subsection*{Organization}
In Section~\ref{sec:preliminaries}, we recall background on finite tensor categories, module categories, relative Serre functors, and actions of $2$-groups on $2$-categories, and we review the classification of $G$-graded extensions from \cite{etingof2010fusion,davydov2021braided}.
Section~\ref{sec:pivotal-extensions} introduces pivotal $G$-graded extensions and the pivotal Brauer--Picard $2$-groupoid, realizes it as fixed points of a $B\underline{\zZ}$-action, and proves the classification of pivotal extensions.
Section~\ref{sec:spherical-extensions} treats the spherical case (for unimodular $\C$), defines $\SphBrPic(\C)$, realizes it as fixed points of a $B\underline{\Ztwo}$-action, and classifies spherical extensions.
In Section~\ref{sec:sphericalization}, we relate sphericalization to the Brauer--Picard formalism and to graded extensions.
Finally, Section~\ref{sec:obstructiontheory} develops the obstruction theory for extending pivotal structures, including both algebraic and homotopical descriptions and examples.


\subsection*{Acknowledgments}
A.C.\ was partially supported by the Simons Collaboration Grant No. 999367.
D.J.\ was supported by The Research Council of Norway - project 324944.
D.N.\ was supported  by  the  National  Science  Foundation  under  
Grant No.\ DMS-2302267 and  would like to thank the Isaac Newton Institute for Mathematical Sciences, Cambridge, for support and hospitality during the programme ``Quantum field theory with boundaries, impurities, and defects", where work on this paper was undertaken. This work was supported by EPSRC grant EP/Z000580/1.
J.P.\ was partially supported by US NSF Grant DMS2146392 and by Simons Foundation Award \#889000 as part of the Simons Collaboration on Global Categorical Symmetries.
D.R.\ was supported by the Deutsche Forschungsgemeinschaft under the Emmy
Noether program – 493608176, and under the Collaborative Research Center (SFB)
1624 “Higher structures, moduli spaces and integrability” – 506632645.
S.S.\ was partially supported by the
National Science Foundation under Grant No.\ DMS-2154389.
H.Y.\ was supported by NSERC discovery grant RGPIN 2024-05109.

This project began while the authors were participating in the workshop Higher categories and topological order at the American Institute of Mathematics, whose hospitality and support are gratefully acknowledged.

The authors thank Pavel Etingof, Thibault Décoppet, César Galindo and Noah Snyder for illuminating discussions.

\section{Preliminaries}
\label{sec:preliminaries}
Throughout the article we consider linear categories over an algebraically closed field $\kk$. We denote by $\Vect_{\kk}$ the category of finite dimensional $\kk$-vector spaces.


\subsection{Tensor categories}
We recall some definitions regarding tensor categories, and refer the reader to \cite{etingof2004tensor,etingof2015tensor}  for details. A $\kk$-linear abelian category is $\emph{locally finite}$ if every object is of finite length and all morphism spaces are finite-dimensional. We say a locally finite $\kk$-linear abelian category $\C$ is \emph{finite} if it has enough projectives and finitely many simple objects. 

A \emph{multi-tensor category} $\C$ is a locally finite $\kk$-linear abelian rigid monoidal category  such that its tensor product functor $\otimes$ is $\kk$-bilinear.
The unit object decomposes as $\unit=\oplus_{i\in I} \unit_i$ as a direct sum of simple objects, and $\C=\oplus_{i,j\in I} \C_{ij}$ where $\C_{ij}:=\unit_i\otimes\C\otimes\unit_j$.

A \emph{tensor category} is a multi-tensor category whose unit object $\mathbb \unit_{\C}$ is simple, i.e. $\End_{\C}(\mathbb \unit_{\C})\cong \kk$. Given a tensor category $\C$, we will denote by $\oC$ the category $\C$ with the opposite tensor product, i.e. $X \,{\otimes_{\overline{\C}}}\,Y \coloneqq Y \otimes X$. Following the conventions from \cite[Def.~2.10.2]{etingof2015tensor}, the left dual $X^*$ of an object $X\in\C$ comes equipped with evaluation and coevaluation morphisms
\begin{equation*}
{\rm ev}_X\colon X^* \otimes\, X\longrightarrow \unit \qquad{\rm and}\qquad
{\rm coev}_X\colon \unit\longrightarrow X\otimes X^* \,,
\end{equation*}
and the left dual ${}^* X$ of $X\in\C$ comes with evaluation and coevaluation morphisms
\begin{equation*}
\widetilde{{\rm ev}_X}\colon X \,\otimes {}^* X\longrightarrow \unit \qquad{\rm and}\qquad
\widetilde{{\rm coev}_X}\colon \unit\longrightarrow {}^* X\otimes X \,.
\end{equation*}

By a \emph{tensor functor}, we mean a $\kk$-linear, exact, faithful, strong monoidal functor.
A tensor functor $F\colon\C\longrightarrow\D$ between tensor categories preserves dualities, that is, we have natural isomorphisms $\xi^F_{X}\colon F(X^*)\xrightarrow{~\sim~} F(X)^*$ for all $X\in \C$.
By applying $\xi^F_X$ twice, we obtain a monoidal natural isomorphism \cite[Lem.~1.1]{MR2381536} 
\begin{equation}\label{eq:iso_F_double-dual}
 \zeta^F_{X}\colon F(X^{**})\xrightarrow{~\sim~}F(X^*)^*\xrightarrow{~\sim~} F(X)^{**}.
\end{equation}

A \emph{pivotal structure} on a tensor category $\C$ is a monoidal natural isomorphism $\fp\colon \id_{\C} \xRightarrow{~\sim~} (-)^{**}$. Given two pivotal tensor categories $(\C,\fp)$ and $(\D,\fq)$, a tensor functor $F\colon\C\longrightarrow\D$ is called \textit{pivotal} \cite{MR2381536} if it satisfies:
\begin{equation}\label{eq: F preserves pivotal structure}
    \fq_{F(X)} =  \zeta^F_{X} \circ F(\fp)\colon F(X) \rightarrow F(X)^{**}
\end{equation}
for every object $X\in\C$.

\subsubsection{Equivariantization}
Let $\C$ be a (multi-)tensor category, and denote by  $\operatorname{Aut}_{\otimes}(\C)$ the monoidal category of tensor auto-equivalences of $\C$. 
For a group $G$, let $\underline{G}$ denote the strict monoidal category with objects the elements of $G$, morphisms given by identity maps, and tensor product induced by the group law of $G$. 
\begin{definition} \cite[Def.\ 2.7.1]{etingof2015tensor}
	An \emph{action} of $G$ on $\C$ is a monoidal functor $T:\underline{G} \to \mathrm{Aut}_{\otimes}(\C)$.
\end{definition}
\begin{remark}
Group actions on linear categories (without a monoidal structure) are similarly defined, where monoidal auto-equivalences are replaced just by linear auto-equivalences.
\end{remark}

\begin{example}\label{action of Z_2 on C}
A tensor auto-equivalence $T\colon\C \to \C$ together with a monoidal natural isomorphism $J\colon T^2 \xrightarrow{~\sim~} \id_{\C}$ such that $JT = TJ \colon T^3 \to T$ determines an action of $\Ztwo$ on a tensor category $\C$.
\end{example}

\begin{definition}\cite[Def.\ 2.7.2]{etingof2015tensor}
Let $\C$ be a tensor category with an action of a finite group $G$. 
The \textit{$G$-equivariantization} $\C^G$ of $\C$ is the monoidal category of $G$-equivariant objects, i.e. pairs $(X,v)$ where $X$ is an object in $\C$ and  $v\coloneqq\{v_g\colon T_g(X) \to X \ | \ g \in G\}$ is a collection of isomorphisms satisfying an appropriate compatibility condition with the action. Morphisms  $f\colon (X,v) \to (Y,w)$ in $\C^G$ are maps  $f\colon X\to Y$ in $\C$ such that $f\circ v_g=w_g\circ f$, for all $g\in G$.
\end{definition}

From \cite[\S4.15]{etingof2015tensor}, we know that if $\C$ is a (multi-)tensor category then so is $\C^G$. The forgetful functor $\forg\colon\C^G\rightarrow\C$ is a tensor functor. Moreover, it admits a left and right adjoint functor which maps $X$ to $\Ind(X):=(\oplus_{g\in G} T_g(X),v)$ where $v$ is defined appropriately \cite[Lemma~4.6]{drinfeld2010braided}. 
\begin{lemma}\label{lem:equi-properties}
Let $\C$ be a finite (multi-)tensor category with a $G$-action. Then, $\C^G$ is a finite (multi-)tensor category.
\end{lemma}
\begin{proof}
As $\C$ is finite, it admits a projective generator $P$ (that is, $\Hom_\C(P,-):\C\rightarrow\Vect$ is exact and faithful). Then, $\Hom_{\C^G}(\Ind(P),-) \cong \Hom_{\C}(P,\forg(-))$ is exact and faithful. Thus, $\Ind(P)$ is projective generator of $\C^G$. Hence, $\C^G$ is finite, proving the  claim. 
\end{proof}


\subsection{Module categories}
A \emph{$($left$)$ module category} over a tensor category $\C$ is a $\kk$-linear abelian category $\M$ together with an exact functor $\tl \colon \C\times \M \to \M$, and an associator which satisfies the pentagon axiom. We will also refer to $\M$ as a \textit{left $\C$-module}, and use the notation  ${}_{\C}\M$ to indicate its left-module structure. Similarly, one can define a \textit{$($right$)$ module category}; we will use the notation $\N_{\C}$. We note that right $\C$-module categories are the same as left $\oC$-module categories.
Similarly, for finite tensor categories 
 $\C$ and $\D$, a $(\C,\D)$-\emph{bimodule category} is a 
(left) module category over the Deligne product $\C\boxtimes\overline{\D}$.

By MacLane's strictness theorem, we will assume that all module categories are strict (see \cite[Remark 7.2.4]{etingof2015tensor}). When $\C$ is finite, we ask that $\M$ is also finite as a $\kk$-linear category.

A (left) \textit{module functor} between (left) $\C$-module categories $\M$ and $\N$ is a functor $H \colon \M\longrightarrow \N$ together
with a collection of natural isomorphisms $H(X \tl M) \xrightarrow{~\sim~} X\tl H(M)$ for all $X\in \C$ and $M\in \M$ satisfying the evident compatibility condition. Functors of (right) $\C$-module  categories are defined analogously. 

Let $\C$ be a finite tensor category.
We call a left $\C$-module category $\M$ \textit{exact} if for any projective $P\in\C$ and any $M\in\M$, $P\tl M\in\M$ is projective. Exactness of right $\C$-module categories is defined analogously. We will denote the category of right exact $\C$-module functors by $\Rex_{\C}(\M,\N)$, and set $\C^*_{\M}:=\Rex_{\C}(\M,\M)$.
Moreover, $\Rex_{\C|\D}(\M,\N)$ will denote the category of right exact $(\C,\D)$-bimodule functors between $\M$ and $\N$. 
Let $\M$ be a right $\C$-module and $\N$ a left $\C$-module. The \textit{relative Deligne product} $\M\boxtimes_{\C}\N$ is an abelian category $\M\boxtimes_{\C}\N$ along with a functor $\mathrm{B}_{\M,\N}\colon\M\times\N \longrightarrow \M\boxtimes_{\C}\N$ universal among $\C$-balanced and right exact in each variable functors from $\M\times \N$ to abelian categories. See \cite{davydov2013picard}, \cite{douglas2019balanced}, \cite[\S 3.2]{davydov2021braided} for background on relative Deligne product of module categories. We will use the notation $M\boxtimes N$ for the image, in $\M\btC\N$, of $(M,N)$ under $\mathrm{B}_{\M,\N}$, and refer to such objects as simple tensors.

\begin{definition}\cite[Def.\ 4.1]{etingof2010fusion}
	A $(\C,\D)$-bimodule category $\M$ is called \textit{invertible} if there exists a $(\D,\C)$-bimodule category $\overline{\M}$ together with equivalences
	\begin{equation*}
		\M\boxtimes_{\D} \overline{\M} \simeq \C, \;\;\;\; \overline{\M}\boxtimes_{\C} \M \simeq \D.
	\end{equation*}
of $\C$-bimodule categories (resp. $\D$).    
\end{definition}
A fact that we will use frequently is that if $\M$ is an invertible $\C$-bimodule category, then $\M$ is exact as a left $\C$-module category \cite[Cor.~5.2]{davydov2021braided}.

\begin{lemma}\label{lem:multitensor_equivalences}
Let $\C=\oplus_{i,j\in I}\,\C_{i,j}$ be a multitensor category. Then for every $i,j,k \in I$ the tensor product induces an equivalence
\begin{equation}\label{eq:multitensor_equivalences}
    \C_{i,k}\boxtimes_{\C_k}\C_{k,j}\xrightarrow{~\sim~}\C_{i,j}
\end{equation}
of $(\C_i,\C_j)$-bimodule categories.
\end{lemma}
\begin{proof}
For each $k\in I$ consider the $(\C,\C_k)$-bimodule category $\M_k\,\coloneqq\bigoplus_{i\in I}\C_{i,k}$. Then, the regular left $\C$–module category $\C$ decomposes as
\[
   \C \;=\;\bigoplus_{k\in I}\M_k
   \qquad\text{(direct sum of $\C$–submodule categories).}
\]
Now, the category of $\C$–module endofunctors of the regular module is well known
to be tensor equivalent to the monoidal opposite of $\C$. 
Therefore, the direct-sum decomposition of $\C$ yields
\begin{equation}\label{eq:dual=sum}
   \overline{\C}\;\simeq\;\Fun_\C(\C,\C)\;=\;
      \Fun_\C\Bigl(\,\bigoplus_{k\in I}\M_k,\;
                         \bigoplus_{\ell\in I}\M_\ell\Bigr)
      \;\simeq\;
      \bigoplus_{k,\ell\in I}\Fun_\C(\M_k,\M_\ell).
\end{equation}
In particular, for the $k$-diagonal component, we read off $\Fun_\C(\M_k,\M_k)\simeq\overline{\C_{k}}$. Now, according to \cite[Prop.\,7.12.11]{etingof2015tensor} there is a canonical tensor equivalence $(\C^*_{\M_k})^*_{\M_k} \simeq \C$. Altogether, we obtain that
\begin{equation}\label{eq:endoMk}
   \Fun_{\overline{\C}_{k}}(\M_k,\M_k) \simeq  \C.
\end{equation}
On the other hand, from the definition of $\M_k$ we have that
\[
   \Fun_{\overline{\C}_{k}}(\M_k,\M_k)
      \;=\;\bigoplus_{i,j\in I}
      \Fun_{\overline{\C}_{k}}\!\Bigl(\;\C_{j,k},\,
        \C_{i,k}\Bigr)
      \;\simeq\;\bigoplus_{i,j\in I}      
         \C_{i,k}\boxtimes_{\C_k}\C_{k,j},
\]
where the last equivalence comes from \cite[Cor. 2.4.11]{douglas2018dualizable}.
Combining with \eqref{eq:endoMk}, we obtain an equivalence of
$\C$-bimodule categories
\begin{equation}\label{eq:block_dec_of_C}
   \C
   \;\simeq\;
   \bigoplus_{i,j\in I}
      \C_{i,k}\boxtimes_{\C_k}\C_{k,j}.
\end{equation}
The $(i,j)$-component on the left is $\C_{i,j}$ and on the right is $\C_{i,k}\boxtimes_{\C_k}\C_{k,j}$. Hence, the restriction of 
\eqref{eq:block_dec_of_C} to that component yields the desired equivalence
\eqref{eq:multitensor_equivalences}.
\end{proof}


\subsection{Internal Hom and relative Serre functors}
Given a left $\C$-module category $\M$, the action functor is exact and thus comes with a right adjoint $\uHom_{\M}^{\C}(M,M')$ for $M,M'\in\M$, i.e. there is a natural isomorphism
\begin{equation}\label{eq:left-internal-hom}
	\Hom_{\M}(X\tl M,M') \cong \Hom_{\C}(X, \uHom_{\M}^{\C}(M,M'))
\end{equation} 
which extends to a left exact functor $\uHom_{\M}^{\C}(-,-)\colon \M^\op\times\M\to\C$ called the \textit{internal Hom} of ${}_\C\M$. Internal Hom's for right module categories are similarly defined.

\begin{definition}\cite[Def.\ 4.22]{fuchs2020eilenberg}
\label{def:rel_Serre_functor}
Let $\C$ be a finite tensor category and $\M$ a left $\C$-module category. A \textit{(right) relative Serre functor} is an endofunctor $\dS^{\C}_{\M}\colon\M\rightarrow\M$ together with a natural isomorphism 
\begin{equation}\label{eq:phi_Serre}
		\phi_{M,N}\colon \uHom_\M^\C\left(N,\dS^{\C}_{\M}(M)\right)\longrightarrow \uHom_\M^\C\left(M,N\right)^*
	\end{equation}
for ${M,N\in\M}$. In a similar manner, a \textit{relative (left) Serre functor} $\odS_{\M}^{\C}$ comes with a natural isomorphism 
	\begin{equation}
		\overline{\phi}_{M,N}\colon \uHom_\M^\C\left(\odS^{\C}_{\M}(N),M\right)\longrightarrow {}^*\uHom_\M^\C\left(M,N\right)
	\end{equation}
for ${M,N\in\M}$.
\end{definition}
Relative Serre functors of $\M$ exist if and only if $\M$ is an exact left $\C$-module category \cite[Prop. 4.24]{fuchs2020eilenberg}. 
In fact, $\dS_{\M}^{\C}$ is an equivalence and $\odS_{\M}^{\C}$ serves as a quasi-inverse.
According to \cite[Lemma 4.23]{fuchs2020eilenberg} the relative Serre functor $\dS_{\M}^{\C}$ comes equipped with a twisted $\C$-module functor structure 
\begin{equation}\label{sCXM}
 X^{**} \tl \dS^{\C}_{\M}(M) \cong \dS^{\C}_{\M}(X\tl M)
\end{equation} 
for $X\in\C$ and $M\in\M$, or equivalently $\dS_{\M}^{\C}$ is a $\C$-module functor from $\M$ to $\M^{\HH}$.
Similarly, the relative Serre functors of an exact bimodule category ${}_\C\M_\D$ are endowed with the structure of twisted bimodule functors \cite[Prop.\ 2.9]{spherical2022}.
\begin{equation}\label{eq:twisted_serre_bimod}
 X^{**} \tl \dS^{\C}_{\M}(M) \tr Y^{**} \cong \dS^{\C}_{\M}(X\tl M\tr Y)\qquad {}^{**}X \tl \dS^{\overline{\D}}_{\M}(M) \tr {}^{**}Y \cong \dS^{\overline{\D}}_{\M}(X\tl M\tr Y)
\end{equation} 
for $X\in\C$, $Y\in\D$ and $M\in\M$. Furthermore, for every (bi)module functor $H \colon \M\to\N$ there is a natural isomorphism
\begin{equation}\label{eq:Serre_module_functor}
  \Lambda_H\colon \dS_\N^\C\circ H \xRightarrow{~\sim~}
  H^\rra\circ \dS_\M^\C
\end{equation}
of twisted (bi)module functors, where $H^\rra$ is the double right adjoint of $H$, see \cite[Thm.\,3.10]{shimizu2019relative} and \cite[Prop.\,2.11]{spherical2022}. 
In fact, relative Serre functors are unique up to unique natural isomorphisms.

\begin{lemma}\cite[Lemma 3.5]{shimizu2019relative}\label{lem:Serre_uniqueness}
Let $(\dS,\phi)$ and $(\dS',\phi')$ be relative Serre functors of a left $\C$-module category $\M$. Then there exists a unique natural isomorphism $\theta: \dS\xRightarrow{~\sim~} \dS'$ such that 
\begin{equation}\label{eq:Serre_uniqueness}
    \phi_{M,N} = \phi'_{M,N}\circ \uHom_{\M}^\C(N,\theta(M)).    
\end{equation} 
\end{lemma}

\begin{lemma}\label{lem:Serre-monoidal}
Let $F_1: \M_1 \to \N_1$ and $F_2: \M_2 \to \N_2$ be $\C$-bimodule equivalences. The natural isomorphism \eqref{eq:Serre_module_functor} is compatible with the relative Deligne product:
\[
\Lambda_{F_1 \btC F_2} \circ \mu_{\M_1, \M_2} = \mu_{\N_1, \N_2} \circ (\Lambda_{F_1} \btC \Lambda_{F_2}).
\]
\end{lemma}
\begin{proof}
Let $\overline{F_i}$ $(i=1,2)$ denote the quasi-inverse of $F_i$. A routine check shows that 
\[ \dS'_{\M_1\btC\M_2} := (\overline{F_1} \btC \overline{F_2}) \circ \dS_{\N_1 \btC \N_2}^{\C} \circ (F_1 \btC F_2).\]
is a relative Serre functor of $\M_1\btC\M_2$. Moreover, we have two natural isomorphisms from $\dS_{\M\btC\N}$ to $\dS'_{\M\btC\N}$.
The first one is obtained using the counit of the adjoint equivalences $F_i\dashv \overline{F_i}$ and the isomorphism $\Lambda_{F_1 \btC F_2}$.
Similarly, the second uses the counits and the following map: 
\begin{align*}
(F_1 \btC F_2) \circ \dS_{\M_1 \btC \M_2}^{\C} & \xrightarrow{\id \circ \mu_{\M_1, \M_2}^{-1}} (F_1 \btC F_2) \circ (\dS_{\M_1}^{\C} \btC \dS_{\M_2}^{\C})  \cong (F_1 \circ \dS_{\M_1}^{\C}) \btC (F_2 \circ \dS_{\M_2}^{\C}) \\
& \xrightarrow{\Lambda_{F_1} \btC \Lambda_{F_2}} (\dS_{\N_1}^{\C} \circ F_1) \btC (\dS_{\N_2}^{\C} \circ F_2) \cong (\dS_{\N_1}^{\C} \btC \dS_{\N_2}^{\C}) \circ (F_1 \btC F_2) \\
& \xrightarrow{\mu_{\N_1, \N_2} \circ \id} \dS_{\N_1 \btC \N_2}^{\C} \circ (F_1 \btC F_2).
\end{align*}
By Lemma~\ref{lem:Serre_uniqueness}, the two isomorphisms must be equal. This gives the identity:
\[
\Lambda_{F_1 \btC F_2} = \mu_{\N_1, \N_2} \circ (\Lambda_{F_1} \btC \Lambda_{F_2}) \circ \mu_{\M_1, \M_2}^{-1}\;.
\]
Rearranging this equation yields the desired result.
\end{proof}

Let $(\M,\tl,\tr)$ be a $(\C,\D)$-bimodule category and let $(\N,\tl',\tr')$ be a $(\D,\E)$-bimodule category.
Then the relative Deligne product $\M\boxtimes_{\D}\N$ is naturally a $(\C,\E)$-bimodule category, with actions
\begin{equation}
 X\Yright(M\boxtimes N)\Yleft'Z \coloneqq (X\tl M)\boxtimes (N\tr' Z)
\end{equation}
for $X\in\C$, $Z\in\E$, $M\in\M$, and $N\in\N$.
For objects of the form $M\boxtimes N$, the internal hom in the left $\C$-module category $(\M\btD \N,\Yright)$ is given by (cf.\ \cite[Proposition~4.15(3)]{schaumann2015pivotal})
\begin{equation}\label{eq:inner-hom-Deligne}
    \uHom^{\C}_{\M\btD\N}(M\boxtimes N, M' \boxtimes N') \cong \uHom_{\M}^{\C}(M \tr {}^*\uHom_{\N}^{\D}(N,N') , M'). 
\end{equation}
The following is $\D$-balanced functor
\[
F\colon \M\times \N \xrightarrow{\dS_{\M}^{\C} \times \dS_{\N}^{\D}} \M \times \N \xrightarrow{B_{\M,\N}} \M \btD \N,\quad
(M, N) \longmapsto \dS_{\M}^{\C}(M) \boxtimes \dS_{\N}^{\D}(N).
\]
Thus, by \cite[Theorem~3.3(4)]{douglas2019balanced}, it induces an exact endofunctor
\[
\dS_{\M}^{\C} \boxtimes_{\D} \dS_{\N}^{\D}\colon \M\btD \N \longrightarrow \M \btD \N,
\]
which is also a $(\C,\E)$-bimodule functor. 

\begin{lemma}\label{lem:inner-hom-Serre}
Suppose that $\M$ is an exact $(\C,\D)$-bimodule category and $\N$ is an exact left $\D$-module category such that $\M \btD \N$ is exact as a left $\C$-module category. 
Then $\dS_{\M}^{\C} \boxtimes_{\D} \dS_{\N}^{\D}$ is a (left) relative Serre functor of the $\C$-module category $(\M\btD\N,\Yright)$.
\end{lemma}
For instance, if $\M$ and $\N$ are invertible $\C$-bimodule categories, then so is $\M\btC\N$. In particular $\M\btC\N$ is exact as a left $\C$-module category. Thus, $\dS_\M\btC \dS_\N$ is a relative Serre functor for the left $\C$-action.

\begin{proof}
By definition \eqref{eq:phi_Serre}, we have to construct natural isomorphisms
\begin{equation}\label{eq:simple-tensor-Serre}
    \phi^\M\star\phi^\N \colon \uHom_{\M\btD \N}^{\C} \big(A, \big(\dS_{\M}^{\C} \btD \dS_{\N}^{\D} \big)(B) \big) \cong \uHom^{\C}_{\M\btD\N}\big(B, A\big)^*
\end{equation} 
for all $A,B\in \M\btD \N$. However, every object of $\M\btD\N$ is a finite colimit of simple tensors. Moreover, the internal Hom functor $\uHom^{\C}_{\M\btD\N}(-,-)$ is exact (because $\M\btD\N$ is exact as a left $\C$-module category) and so is $\dS_{\M}^{\C} \btD \dS_{\N}^{\D}$; thus their compositions preserve colimits. Therefore, it suffices to check \eqref{eq:simple-tensor-Serre} for simple tensors. 

On simple tensors, define $\phi^\M\star\phi^\N$ as the following composition of isomorphisms
\begin{equation}\label{phiMphiN}
\begin{aligned}
    \uHom_{\M\btD \N}^{\C} \big(M\boxtimes N, \big(\dS_{\M}^{\C} \btD \dS_{\N}^{\D} \big)(M'\boxtimes N' ) \big) 
    &  = \uHom_{\M\btD \N}^{\C} \big(M\boxtimes N, \dS_{\M}^{\C}(M') \boxtimes \dS_{\N}^{\D}(N')\big)\\
    {}^{\eqref{eq:inner-hom-Deligne}} 
    &  \cong \uHom_{\M}^{\C} \big( M \tr {}^*\uHom_{\N}^{\D}(N,\dS_{\N}^{\D}(N')) , \dS_{\M}^{\C}(M')  \big) \\
    {}^{{}^*\phi^\N} 
    & \cong \uHom_{\M}^{\C} \big(M \tr \uHom_{\N}^{\D}(N',N) , \dS_{\M}^{\C}(M')  \big) \\
    {}^{\phi^\M} 
    & \cong \uHom_{\M}^{\C} \big(M',M \tr \uHom_{\N}^{\D}(N',N)\big){}^* \\
    & \cong\uHom_{\M}^{\C}\big(M' \tr {}^*\uHom_{\N}^{\D}(N',N) , M\big) ^* \\
    {}^{\eqref{eq:inner-hom-Deligne}} 
    & \cong \uHom^{\C}_{\M\btD\N}\big(M'\boxtimes N', M\boxtimes N\big)^*\,.
\end{aligned}
\end{equation}
This isomorphism exhibits $(\dS_{\M}^{\C} \boxtimes_{\D} \dS_{\N}^{\D}, \phi^\M\star\phi^\N)$ as a relative Serre functor as desired.
\end{proof}

\begin{proposition}\label{prop:Serre_monoidal}
Let $\M$ be an exact $(\C,\D)$-bimodule category and $\N$ an exact left $\D$-module category. There is a natural isomorphism 
\begin{equation}\label{eq:Serre_monoidal}
    \mu_{\M,\N}\colon \dS_{\M}^{\C}\btD\dS_{\N}^{\D} \xRightarrow{~\sim~} \dS^{\C}_{\M\btD \N}
\end{equation}
of twisted $\C$-module functors, obeying that for any $M\boxtimes N$ and $M'\boxtimes N'$ in $\M\btD\N$, the following diagram commutes 
\begin{equation}\label{mu_MN}
    \begin{tikzcd}
        \uHom_{\M\btD \N} \big(M\boxtimes N, \dS^{\C}_{\M}(M') \btD\dS^{\D}_{\N}(N')\big)\ar[dd,swap,"\mu_{\M,\N}"]\ar[rrd,"\phi^\M\star\phi^\N"]&&\\
        &&\uHom_{\M\btD\N}(M'\boxtimes N',M\boxtimes N)^*\\
        \uHom_{\M\btD \N} \big(M\boxtimes N, \dS^{\C}_{\M\btD\N}(M'\boxtimes N')\big)\ar[rru,swap,"\phi^{\M\btD\N}"]&&
    \end{tikzcd}
\end{equation}
where $\phi^\M\star\phi^\N$ is given by the chain of isomorphisms in \eqref{phiMphiN}.
\end{proposition}
\begin{proof}
Lemma \ref{lem:inner-hom-Serre} states that $(\dS_{\M}^{\C}\btD\dS_{\N}^{\D}, \phi^{\M} \star \phi^{\N})$ is a relative Serre functor of $\M\boxtimes_{\D} \N$. Thus, the statement follows from the uniqueness of relative Serre functors \eqref{eq:Serre_uniqueness}.
\end{proof}


\subsection{Actions of $2$-groups on $2$-categories}
We refer the reader to \cite[\S2.2]{davydov2021braided} for background on monoidal $2$-categories. A \textit{$2$-groupoid} is a $2$-category in which all $1$-morphisms are equivalences and $2$-morphisms are invertible. A \textit{$2$-categorical group} (aka $3$-group) is a monoidal $2$-groupoid whose objects are invertible under the monoidal structure. 

\begin{example} Let $G$ be a group.
	\begin{enumerate}[(i)] 
    	\item We can form the $2$-categorical group $\uuG$: objects are $g\in G$ with tensor product given by group multiplication, and $1$- and $2$-morphisms are identities.
        
		\item If $G$ is abelian, we can also form the $2$-categorical group  $B\underline{G}$: As a $2$-groupoid, it has one object $*$ (with $*\boxtimes *=*$), a $1$-morphism for every $g\in G$, and only identity $2$-morphisms. Tensor product and composition of $1$-morphisms are induced by the group multiplication. 
        Equivalently, $B\underline{G}$ is the monoidal $2$-groupoid with one object and braided monoidal $1$-groupoid of endomorphisms given by $\underline{G}$ the braided monoidal $1$-groupoid with objects $g\in G$, only identity morphisms and trivial braiding. 
	\end{enumerate}
\end{example}

Given a $2$-category $\bA$, the monoidal $2$-category $\Aut(\bA)$ whose objects are autoequivalences of $\bA$, $1$-morphisms are pseudo-natural equivalences and $2$-morphisms are invertible modifications is a $2$-categorical group with composition as monoidal structure. Similarly, when $\bA$ is a monoidal $2$-category, we denote by $\MonAut(\bA)$ the $2$-category of monoidal autoequivalences of $\bA$.


\begin{definition} Let $\bG$ be a 2-categorical group and $\bA$ a $2$-category.
\begin{itemize}
    \item  An \emph{action} of $\bG$ on $\bA$ is a monoidal $2$-functor $\bG\rightarrow \Aut(\bA)$.
    \item If $\bA$ is a monoidal, a \emph{monoidal action} of $\bG$ on $\bA$ is a monoidal $2$-functor $\bG\rightarrow \MonAut(\bA)$.
\end{itemize}
\end{definition}

Throughout the paper, we will use the phrase `fixed point' in the homotopy coherent sense.
\begin{example}\label{ex: BZ and BZ2 actions}
Let $G$ be an abelian group. A $B\underline{G}$ action on a (monoidal) $2$-groupoid $\bA$ is equivalently a braided monoidal functor $\underline{G} \to \Omega \Aut(\bA)$ (resp. to $\Omega \MonAut(\bA)$) where this target is the braided monoidal $1$-groupoid of (monoidal) pseudonatural auto-equivalences of the identity functor on $\bA$ and (monoidal) invertible modifications between these, with braiding induced by the Eckmann-Hilton argument. 
Fixing a presentation of $G$, (up to equivalence) this data explicitly amounts to: 
\begin{itemize}
    \item For every generator $g$ of $G$ a (monoidal) pseudonatural equivalence $\eta^g \colon \id_{\bA} \Rightarrow \id_{\bA}$ such that its self-braiding
    $\eta^g \circ \eta^g \Rightarrow \eta^g\circ \eta^g$
    in $\Omega \Aut(\bA)$ is trivial. 
    Equivalently, for every $g\in G$ and every $a\in \bA$, the following $2$-isomorphism in $\bA$: 
\begin{equation}\label{eq:self_braiding}
    (\eta^g)_{\eta^g_a} \colon (\eta^g)_a \circ (\eta^g)_a \Rightarrow (\eta^g)_{a} \circ (\eta^g)_a
\end{equation}
    given by the component of the pseudo-natural equivalence $\eta^g$ at the $1$-morphism $\eta^g_a$ is trivial.
    \item For every relation $R = g_1 \cdots g_n$ an invertible modification $m^R\colon\id_{\id_{\bA}}\Rightarrow \eta^{g_1} \circ \cdots \circ \eta^{g_n}$. 
\end{itemize}

A $B \underline{\mathbb{Z}}$-action therefore amounts to the data of a (monoidal) pseudonatural equivalence $\eta: \id_{\bA} \Rightarrow \id_{\bA}$ with trivial self-braiding, and factoring it through a  $B\underline{\Ztwo}$-action amounts to a further (monoidal) invertible modification $m\colon \id_{\id_{\bA}} \Rightarrow \eta^2. $ 

\end{example}
\begin{definition}\label{def:fixed points} \cite{hessevalentino}
Given an action $F\colon\bG\longrightarrow\Aut(\bA)$ on a $2$-category $\bA$, its equivariantization or $2$-category of fixed points $\bA^{\bG}$ has as objects the following data:

\begin{itemize}
    \item an object $a\in \bA$;
    \item for every $g\in\bG$, a $1$-equivalence $\Theta_g\colon a\to f(a)$ in $\bA$;
    \item for every $1$-morphism $\gamma\colon g\to h$ in $\bG$, an invertible $2$-morphism $\Theta_\gamma\colon F_\gamma\circ \Theta_g \Rightarrow \Theta_h$ which is natural in $\gamma$;
    \item for every pair $(g,h)\in\bG\times\bG$ an invertible $2$-morphism $\Pi_{g,h}$
\begin{equation}
    \begin{tikzcd}
        a \ar[r,"\Theta_g"]\ar[d,swap,"\Theta_{gh}"]& f(a) \ar[d,"f(\Theta_h)"] \ar[ld,"\Pi_{g,h}",Rightarrow,swap]\\
        F_{gh}(a) \ar[r,swap,"\cong"]& f(F_h(a)) 
    \end{tikzcd}
\end{equation}
    \item an invertible $2$-morphism $\Theta_e\Rightarrow \iota_a \colon a\to F_e(a)$;
\end{itemize}
obeying multiple compatibility conditions as described in \cite{hessevalentino}.
\end{definition}
If $\bA$ is monoidal and $\bG$ acts by monoidal $2$-functors, then $\bA^{\bG}$ inherits a monoidal structure.

\begin{example}\label{ex: fixed point for BZ and BZ2 actions} Let $G$ be an abelian group and consider a $B\underline{G}$-action on a $2$-groupoid $\bA$ as unpacked in Example \ref{ex: BZ and BZ2 actions}.
Fixing a presentation of $G$, the data of a $B\underline{G}$-fixed point is, up to equivalence, 
    \begin{itemize}
        \item an object $a\in \bA$;
        \item for every generator $g$ in $G$ an invertible $2$-morphism $\Theta^g\colon\id_a\Rightarrow (\eta^g)_a$ in $\bA$;
        \item which for every relation $R=g_1 \cdots g_n$ obeys
        \[
\begin{tikzcd}[column sep=large]
	a & a & a & a
	\arrow[""{name=0, anchor=center, inner sep=0}, "{(\eta^{g_n})_a}"', curve={height=15pt}, from=1-1, to=1-2]
	\arrow[""{name=1, anchor=center, inner sep=0}, "{\id_a}", curve={height=-15pt}, from=1-1, to=1-2]
	\arrow[dotted, no head, from=1-2, to=1-3]
	\arrow[""{name=2, anchor=center, inner sep=0}, "{\id_a}", curve={height=-15pt}, from=1-3, to=1-4]
	\arrow[""{name=3, anchor=center, inner sep=0}, "{(\eta^{g_1})_a}"', curve={height=15pt}, from=1-3, to=1-4]
	\arrow["{\Theta^{g_n}}",shorten <=3pt, shorten >=3pt, Rightarrow, from=1, to=0]
	\arrow["{\Theta^{g_1}}", shorten <=3pt, shorten >=3pt, Rightarrow, from=2, to=3]
\end{tikzcd} = m^R_a\]
    \end{itemize}
A fixed point for a $B\underline{\mathbb{Z}}$-action (as unpacked at the end of Example~\ref{ex: BZ and BZ2 actions}) therefore amounts to a pair $(a, \Theta)$ of an object $a\in A$ and an invertible $2$-morphism $\Theta: \id_a \Rightarrow \eta_a$. This is a fixed point for a $B \underline{\Ztwo} $ action if furthermore $\eta_a(\Theta) \circ \Theta = m_a$.  See the proof of \cite[Theorem 4.1]{hessevalentino} for details. 

Following the same steps as in the proof of \cite[Theorem 4.1]{hessevalentino} one has that the data for a $1$-morphism of fixed points $(a, \Theta)\to (a',\Theta')$ in $\bA^{B \underline{\mathbb{Z}}}$ reduces to a 1-morphism $f\colon a\to a'$ in $\bA$ with the requirement that the following diagram commutes
\begin{equation}\label{eq:1-morph-equiv}    
\begin{tikzcd}[column sep=normal,row sep=scriptsize]
	&{\id_{a'}\circ f} & {\eta_{a'}\circ f}\\
    f &&\\
    &{f\circ \id_a} & {f\circ \eta_a} 
	\arrow["\sim", from=2-1, to=1-2]
	\arrow["\sim"', from=2-1, to=3-2]
	\arrow["\Theta"', from=3-2, to=3-3]
	\arrow["{\Theta'}", from=1-2, to=1-3]
	\arrow["{\eta_f \;\;.}", from=1-3, to=3-3]
\end{tikzcd}
\end{equation}
\end{example}


\subsection{Classification of $G$-graded extensions of tensor categories}

We recall in this section some of the results from \cite{etingof2010fusion,davydov2021braided} relating to the classification of $G$-graded extensions of tensor categories.

\begin{definition}
\label{def:Ext(G,C)}
Let $\C$ be a tensor category and $G$ be a finite group. The $2$-groupoid $\Ext(G,\C)$ of $G$-graded extensions of $\C$ is given by: 
	\begin{itemize}
		\item Objects are $G$-extensions, i.e. a $G$-graded tensor category $\D$ with a monoidal equivalence $\eqtriv\colon\D_{e}\xrightarrow{~\sim~} \C$. 
		\item 1-cells are grading preserving monoidal equivalences $F\colon\D\rightarrow\D'$ equipped with a monoidal natural isomorphism $\tau_F\colon \eqtriv^{\D'}\circ F_e \Rightarrow \eqtriv^{\D}$, where $F_e\colon\D_e\rightarrow \D'_e$ is the restriction of $F$ to $\D_e$, and

		\item $2$-cells are monoidal natural isomorphisms $\gamma\colon F\Rightarrow H$ which are compatible with the respective monoidal equivalences with $\C$. This means that for $F,H\colon\D \to \D'$, the following diagram commutes
		\begin{equation}\label{eq:comp_nat_triv_component}
			\begin{tikzcd}[column sep=normal]
				\eqtriv^{\D'}\circ F_e \ar[rr,"\gamma_e", Rightarrow]\ar[""{name=2}, dr,swap,Rightarrow,"\tau_F"]&& \eqtriv^{\D'}\circ H_e   \ar[""{name=1}, dl,Rightarrow,"\tau_H"] \\
				& \eqtriv^{\D}&
			\end{tikzcd}\; .
		\end{equation}
	\end{itemize} 
\end{definition}

\begin{remark}
The $2$-groupoid $\Ext(G,\C)$ defined in Definition \ref{def:Ext(G,C)} is equivalent to the $2$-groupoid whose objects are $G$-graded tensor categories $\D$ such that $\D_e=\C$, $1$-cells are grading preserving monoidal equivalences $F\colon\D\longrightarrow\D'$ with $F_e=\id_\C,$ and $2$-cells are monoidal natural isomorphisms $\gamma\colon F\Rightarrow H$.
\end{remark}
\begin{definition}\label{def:BrPic}
The \textit{Brauer-Picard $2$-categorical group} of a finite tensor category $\C$ is the $2$-categorical group $\BrPic(\C)$ whose objects are invertible $\C$-bimodule categories, $1$-cells are $\C$-bimodule equivalences, $2$-cells are bimodule natural isomorphisms and monoidal structure given by the relative Deligne product.
\end{definition}
  
\begin{theorem}[{\cite[Theorem 7.7]{etingof2010fusion}} and {\cite[Theorem 8.5]{davydov2021braided}}]
\label{thm:ext_classification}
Let $\C$ be a finite tensor category. There is an equivalence of $2$-groupoids
\begin{equation}\label{eq:ext_classification}
\mathrm{E}\colon \Ext(G,\C) \xrightarrow{~\simeq~} \MonFun\left(\uuG,\BrPic(\C)\right).
\end{equation}    
\end{theorem}

We include below a technical lemma on the relative Serre functors of $G$-graded tensor categories that will be used later on. 

\begin{lemma}\label{lem:Serre_homogeneous}
Let $\D$ be a $G$-graded finite tensor category. For every $g\in G$, the functor $(-)^{**}|_{\D_g}$ is endowed with the structure of a relative Serre functor of ${}_{\D_e}\D_g$. Then, there is a natural isomorphism
\begin{equation}
\label{eq:Serre_homogeneous_components}
    \Omega_g\colon\;\dS_{\D_g}^{\D_e} \xRightarrow{~\sim\,} (-)^{**}|_{\D_g} 
\end{equation}
of twisted $\D_e$-bimodule functors. Moreover, given a $G$-graded tensor equivalence $F:\D\to\D'$, the diagram
\begin{equation}\label{eq:Serre_comp_graded_functor}
    \begin{tikzcd}
        F\circ\dS_{\D_g}^{\D_e} \ar[d,"\eqref{eq:Serre_module_functor}",Rightarrow,swap]\ar[Rightarrow,rr,"\id\;\circ\;\eqref{eq:Serre_homogeneous_components}"]&& F\circ(-)^{**}|_{\D_g} \ar[d,"\eqref{eq:iso_F_double-dual}",Rightarrow]\\
        \dS_{\D_g'}^{\D_e'}\circ F \ar[Rightarrow,rr,"\eqref{eq:Serre_homogeneous_components}\;\circ\;\id",swap]&& (-)^{**}|_{\D_g'}\circ F
    \end{tikzcd}
\end{equation}
commutes for each $g\in G$.
\end{lemma}
\begin{proof}
Regard $\D=\oplus_{g\in G}\,\D_g$ as a $\D_e$-bimodule category, then from \cite[Proposition 2.5]{gjs2022} it follows that
\begin{equation}
    \uHom_\D^{\D_e}(X,Y)\cong \bigoplus_{g\in G}\,\uHom_{\D_g}^{\D_e}(X_g,Y_g)
\end{equation}
which means that for a homogeneous $X\in\D_g$
\begin{equation}
    \uHom_\D^{\D_e}(X,-)|_{\D_g}\cong \uHom_{\D_g}^{\D_e}(X,-)\;\colon \;\D_g\to\D_e\,.
\end{equation} 
Now, according to \cite[Proposition 4.2]{gjs2022}, the degree of the internal hom of two objects in the same homogeneous component is trivial. We thus have for $X,Y\in\D_g$ that
\begin{equation}
    \uHom_\D^{\D_e}(X,Y)\cong\uHom_\D^{\D}(X,Y)\cong Y\otimes X^*.
\end{equation}
Therefore, altogether, the relative Serre functor of ${}_{\D_e}\!\D_g$ is given by
\begin{equation}
    \dS_{\D_g}^{\D_e}(X)\cong \uHom_{\D_g}^{\D_e}(X,-)^\ra (\unit)\cong (-\otimes X^*)^\ra(\unit)\cong X^{**}
\end{equation}
which proves the desired result. Additionally, consider the following composition of natural isomorphisms
\begin{equation}
    \psi_g\colon \uHom_{\D_g}^{\D_e}(Y,X^{**})\cong X^{**}\otimes Y^*\cong (Y\otimes X^*)^*\cong \uHom_{\D_g}^{\D_e}(X,Y)^{*}
\end{equation}
explicitly endowing $(-)^{**}|_{\D_g}$ with the structure of a relative Serre functor for ${}_{\D_e}\D_g$. From \cite[Lemma~3.5]{shimizu2019relative} we obtain a unique natural isomorphism \eqref{eq:Serre_homogeneous_components} of relative Serre functors. We have, in particular, as a consequence that the diagram \eqref{eq:Serre_comp_graded_functor} commutes.
\end{proof}


\section{Pivotal extensions and the pivotal Brauer-Picard \texorpdfstring{$2$}{2}-groupoid}\label{sec:pivotal-extensions}
This section is organized as follows. In Section~\ref{subsec:PivG-ext}, we introduce the notion of pivotal $G$-graded extensions and set up the basic framework for studying extensions of pivotal tensor categories. In Section~\ref{subsec:pivotal-brauer-picard}, we define the pivotal Brauer-Picard $2$-categorical group and discuss its structure in the context of pivotal bimodule categories. In Section~\ref{subsec:homotopy-fixed-points}, we realize the pivotal Brauer-Picard $2$-categorical group as the $2$-groupoid of fixed points for a natural $B\underline{\mathbb{Z}}$-action. Finally, in Section~\ref{subsec:classification-pivotal-G}, we provide a classification of pivotal $G$-graded extensions in terms of monoidal functors into the pivotal Brauer-Picard $2$-categorical group.


\subsection{Pivotal $G$-graded extensions}\label{subsec:PivG-ext}

Let $\C$ be a pivotal tensor category with pivotal structure $\fp\colon \id_{\C} \xRightarrow{~\sim~} (-)^{**}$ and $G$ be a finite group.

\begin{definition}\label{def:pivotalGgradedextn}
A \emph{pivotal G-graded extension} of a pivotal tensor category $\C$ is a tuple $(\D,\eqtriv,\fq),$ where $(\D, \fq)$ is a pivotal tensor category and $(\D,\eqtriv\colon\D_e\xrightarrow{~\sim~}\C)$ is a $G$-graded extension of $\C$ where $\eqtriv$ is pivotal.
\end{definition}

\begin{definition}
\label{def: piv ext}
Let $\C$ be a pivotal tensor category. We define $\PivExt(G,\C)$ as the $2$-groupoid with
\begin{itemize}
	\item Objects being pivotal $G$-graded extensions $(\D,\eqtriv^{\D}, \fq)$ of $\C$,
	\item $1$-cells are grading preserving pivotal tensor equivalences $F\colon\D\rightarrow\D'$ equipped with a monoidal natural isomorphism $\tau_F\colon \eqtriv^{\D'}\circ F_e \Rightarrow \eqtriv^{\D}$, where $F_e\colon\D_e\rightarrow \D'_e$ is the restriction of $F$ to $\D_e$, 
	\item $2$-cells are monoidal natural isomorphisms obeying \eqref{eq:comp_nat_triv_component}.
\end{itemize} 
\end{definition}

The purpose of this section is to realize the $2$-groupoid $\PivExt(G,\C)$ of pivotal $G$-extensions of $\C$ as the $2$-groupoid of fixed points of an appropriate $B\underline{\zZ}$-action on $\Ext(G,\C)$ which depends on the choice of pivotal structure $\fp$. According to Example \ref{ex: BZ and BZ2 actions}, it is enough to define a pseudo-natural autoequivalence of the identity $2$-functor $\id_{\Ext(G,\C)}$ and show that this pseudo-natural transformation has trivial self-braiding.
The double-dual functors of the $G$-graded extensions of $\C$ assemble into such a pseudo-natural autoequivalence. Explicitly,
\begin{itemize}
    \item For every object $(\D,\eqtriv^{\D}) \in \Ext(G,\C),$ consider the  double-dual functor $(-)_{\D}^{**}:\D \to \D$ along with the natural isomorphism $\tau_{(-)^{**}_{\D}}\colon\eqtriv^{\D} \circ (-)^{**}_{\D_e}\Rightarrow \eqtriv^{\D}$ defined by the composition
    \begin{equation}\label{eq:triviso}
        \iota^{\D}\circ \ (-)_{\D_e}^{**} \xRightarrow{~\eqref{eq:iso_F_double-dual}~} (-)_{\C}^{**}\circ \iota^{\D}\xRightarrow{\fp^{-1}} \id_{\C} \circ \ \iota^{\D}.
    \end{equation}

    \item For every $1$-cell $F\colon\D\longrightarrow \D'$ in $\Ext(G,\C)$, consider the $2$-cell
        \begin{equation}\label{eq:PsNat}
       (-)^{**}_F\colon F \circ (-)_{\D}^{**}\xRightarrow{~\sim~} (-)_{\D'}^{**}\circ F
    \end{equation}
    given for $X\in \D$ by the monoidal natural isomorphism $\zeta^F_X\colon F(X^{**})\xrightarrow{~\sim~} F(X)^{**}$ from \eqref{eq:iso_F_double-dual}. 
\end{itemize}

\begin{proposition} \label{prop:BZ-action-Ext}
The collection defined above  assembles into a pseudo-natural equivalence
\begin{align*}
    (-)^{**}\colon \id_{\Ext(G,\C)}\xRightarrow{~\sim~} \id_{\Ext(G,\C)}\,.
\end{align*}
\end{proposition}
\begin{proof}
For every $\D$ in $\Ext(G,\C)$, the double-dual functor $(-)^{**}_{\D}:\D\rightarrow \D$ is grading preserving, 
and thus together with the isomorphism \eqref{eq:triviso} is a $1$-cell in $\Ext(G,\C)$. 
Consider now $1$-cells  $F,H\colon\D\to \D'$ in $\Ext(G,\C)$. To show that $(-)^{**}$ is natural for $2$-morphisms, we need to prove that the diagram
\[\begin{tikzcd}
	{F \circ (-)^{**}} && {(-)^{**}\circ F} \\
	{H \circ (-)^{**}} && {(-)^{**}\circ H}
	\arrow["{(-)^{**}_F}", Rightarrow, from=1-1, to=1-3]
	\arrow["\eta"', Rightarrow, from=1-1, to=2-1]
	\arrow["\eta", Rightarrow, from=1-3, to=2-3]
	\arrow["{(-)^{**}_H}"', Rightarrow, from=2-1, to=2-3]
\end{tikzcd}\]
commutes for any monoidal natural transformation $\eta\colon F\Rightarrow H$. This translates to $\eta_X^{**}\circ \zeta^{F}_{X}=\zeta^{H}_{X}\circ \eta_{X^{**}}$, which holds by \cite[Lemma 3.2]{shimizu2015pivotal}.
Now, to check pseudo-naturality, consider $1$-cells $F\colon \D' \to \D''$ and $H\colon\D\to \D'$ in $\Ext(G, \C)$. 
Notice that $(-)^{**}_F\circ (-)^{**}_H$ is the composition $F(H(X^{**}))\xrightarrow{F(\zeta^{H}_{X})} F((HX)^{**})\xrightarrow{\zeta^{F}_{HX}} F(H(X))^{**}$, for $X\in \D$. Thus $(-)^{**}_{F\circ H}=(-)^{**}_F\circ (-)^{**}_H$ follows by the fact that $\zeta^{F}_{H(X)}\circ F(\zeta^{H}_{X})=\zeta^{F\circ H}_{X}$, see \cite[Lemma 3.1]{shimizu2015pivotal}.

Next, to show that $(-)^{**}_F$ is a $2$-cell in $\Ext(G,\C)$, we need to prove that it is compatible with the monoidal equivalence  $\eqtriv^{\D}\colon\D_e\xrightarrow{~\sim~} \C$, i.e. that the diagram
\begin{equation*}
    \begin{tikzcd}[column sep=large,row sep=scriptsize]
        \eqtriv^{\D'}\circ F_e \circ (-)_{\D_e}^{**} \ar[r,"(-)_{F_e}^{**}", Rightarrow]\ar[""{name=2}, d,swap,Rightarrow,"\tau_{F\circ (-)_{\D}^{**}}"]& \eqtriv^{\D'}\circ (-)_{\D'_e}^{**} \circ F_e \ar[""{name=1}, d,Rightarrow,"\tau_{(-)_{\D'}^{**}\circ F}"] \\
      \eqtriv^{\D}  \ar[""{name=2},r,Rightarrow,swap,"\id"] 
      & \eqtriv^{\D}
    \end{tikzcd}
\end{equation*}
commutes. 
Replacing $\tau_{F\circ (-)_{\D}^{**}}=\tau_{(-)_{\D}^{**}}\circ \tau_F$, $\tau_{(-)_{\D'}^{**} \circ F}=\tau_F \circ \tau_{(-)_{\D'}^{**}},$ and $\tau_{(-)_{\D}^{**}}, \tau_{(-)_{\D'}^{**}}$ by their respective definitions \eqref{eq:triviso}, we obtain an equivalent diagram given by the outside composition of arrows in
\[\begin{tikzcd}[column sep=large,row sep=scriptsize]
	{\iota^{\D'}\circ F_e\circ(-)_{\D_e}^{**}} &&&& {\iota^{\D'}\circ(-)_{\D_e'}^{**}\circ F_e} \\
	{\iota^{\D}\circ(-)_{\D_e}^{**}} && {(-)_{\C}^{**}\circ\iota^{\D}} && {(-)_{\C}^{**}\circ\iota^{\D'}\circ F_e} \\
	{(-)_{\C}^{**}\circ\iota^{\D}} && {\iota^{\D}} && {\iota^{\D'}\circ F.}
	\arrow["{(-)_{F_e}^{**}}", Rightarrow, from=1-1, to=1-5]
	\arrow["{\tau_F}"', Rightarrow, from=1-1, to=2-1]
	\arrow["{(-)_{\iota^{\D'}\circ F_e}^{**}}", Rightarrow, from=1-1, to=2-5]
	\arrow["{(-)_{\iota}^{**}}", Rightarrow, from=1-5, to=2-5]
	\arrow["{(-)_{\iota^{\D}}^{**}}", Rightarrow, from=2-1, to=2-3]
	\arrow["{(-)_{\iota^{\D}}^{**}}"', Rightarrow, from=2-1, to=3-1]
	\arrow["{\tau_F}"', Rightarrow, from=2-5, to=2-3]
	\arrow["{\fp^{-1}}"', Rightarrow, from=3-1, to=3-3]
	\arrow["\fp", Rightarrow, from=3-3, to=2-3]
	\arrow["\fp"', Rightarrow, from=3-5, to=2-5]
	\arrow["{\tau_F}", Rightarrow, from=3-5, to=3-3]
\end{tikzcd}\]
Now, the bottom right square and top left triangle of this diagram commute by naturality of $\tau$. The bottom left square commutes trivially. The top right triangle commutes since $(-)^{**}$ respects composition of functors, which was checked in a preceding argument following \cite[Lemma 3.1]{shimizu2015pivotal}. Lastly, the top left square commutes by monoidality of $\tau$, see \cite[Lemma 3.1]{shimizu2015pivotal}.
\end{proof}

\begin{corollary}\label{cor: BZ action on Ext}
    A pivotal structure on a tensor category $\C$ determines a $B\underline{\zZ}$-action on $\Ext(G,\C)$.
\end{corollary}

\begin{proof}
This follows Proposition \ref{prop:BZ-action-Ext}. Indeed, as explained in Example \ref{ex: BZ and BZ2 actions} it only remains to be shown that for the pseudo-natural equivalence $(-)^{**}$ the self-braiding \eqref{eq:self_braiding} is trivial. More explicitly, that \eqref{eq:PsNat} for $F=(-)^{**}_\D$ is the identity for every extension $\D$. Now, $(-)^{**}_F$ is given by the composition of the tensor structure of $F$ and appropriate instances of evaluation and coevaluation morphisms. But the tensor structure of $(-)^{**}_\D$ is again a composition of evaluation and coevaluation morphisms. Altogether, the snake relations lead to the desired result.
\end{proof}

\begin{proposition}\label{prop: BZ fixed points}
Let $\C$ be a pivotal finite tensor category and $G$ a finite group.
The $2$-groupoid $\PivExt(G,\C)$ is $2$-equivalent to the $2$-groupoid of $B\underline{\zZ}$-fixed points $\Ext(G,\C)^{B\underline{\zZ}}$.
\end{proposition}

\begin{proof}
By Example \ref{ex: fixed point for BZ and BZ2 actions}, an object in $\Ext(G,\C)^{B\underline{\zZ}}$ is (up to equivalence) the data of a $G$-graded extension  $(\D,\eqtriv^{\D})$ of $\C$ together with monoidal natural isomorphism $\fq\colon  \id_{\D}\xRightarrow{~\sim~} (-)^{**}_{\D}$ obeying that
\begin{equation}
    \begin{tikzcd}[column sep=large]
        \eqtriv^{\D}\circ (-)^{**}_{\D_e} \ar[r,"\fq_{e}^{-1}", Rightarrow]\ar[""{name=2}, d,swap,Rightarrow,"\tau_{(-)_{\D}^{**}}"]& \eqtriv^{\D}\circ \id_{\D_e} \ar[""{name=1}, d,Rightarrow,"\id"] \\
      \eqtriv^{\D}  \ar[""{name=2},r,Rightarrow,swap,"\id"] 
      & \eqtriv^{\D}
    \end{tikzcd}
\end{equation}
commutes, where $\tau_{(-)_{\D}^{**}}$ is given by \eqref{eq:triviso}. Thus, we have that
\[\begin{tikzcd}
	{\eqtriv^{\D}\circ (-)^{**}_{\D_e}} & {\eqtriv^{\D}\circ \id_{\D_e}} \\
	{(-)_{\C}^{**}\circ\eqtriv^{\D}} & {\eqtriv^{\D}}
	\arrow["{\fq_{e}^{-1}}", Rightarrow, from=1-1, to=1-2]
	\arrow["{(-)_{\iota^{\D}}^{**}}"', Rightarrow, from=1-1, to=2-1]
	\arrow["\id", Rightarrow, from=1-2, to=2-2]
	\arrow["{\fp_{\C}^{-1}}"', Rightarrow, from=2-1, to=2-2]
\end{tikzcd}\]
commutes, which means that $\eqtriv^\D$ is pivotal. Hence, $(\D,\eqtriv^{\D}, \fq)$ is a pivotal $G$-graded extension of $(\C,\fp)$. Analogously, if we start with a pivotal $G$-graded extension $(\D,\eqtriv^{\D}, \fq)$ of $\C$, then $(\D,\eqtriv^{\D})$ with the invertible $2$-morphism $\fq$ gives a fixed point of the $B\underline{\zZ}$-action.
Hence, the data of an object in $\Ext(G,\C)^{B\underline{\zZ}}$ is the same as that of an object in $\PivExt(G,\C).$

On the other hand, as seen in Example \ref{ex: fixed point for BZ and BZ2 actions} the data of a 1-cell $(\D,\iota^{\D}, \fq)\to (\D',\iota^{\D'}, \fq')$ in $\Ext(G,\C)^{B\underline{\zZ}}$ is a 1-cell $(F,\tau_F)$ in $\Ext(G,\C),$ where $F\colon\D\to \D'$, such that the diagram 
\[\begin{tikzcd}
	F && {F\circ (-)^{**}} \\
	& {(-)^{**}\circ F}
	\arrow["\fq", from=1-1, to=1-3,Rightarrow]
	\arrow["{\fq'}"', from=1-1, to=2-2,Rightarrow]
	\arrow["{(-)_F^{**}}"', from=2-2, to=1-3,Rightarrow]
\end{tikzcd} \]
commutes, see \eqref{eq:1-morph-equiv}. That is, for $X\in \D$ we have that $\zeta_X^F\circ F(\fq_X)=\fq'_{F(X)},$ 
meaning that $F\colon(\D, \fq)\to (\D', \fq')$ is pivotal.
This shows the data of a 1-cell in $\Ext(G,\C)^{B\underline{\zZ}}$ is the same as the data of a 1-cell in $\PivExt(G,\C)$. It is clear that $2$-cells match as well, which finishes the proof.
\end{proof}


\subsection{The pivotal Brauer-Picard $2$-categorical group}\label{subsec:pivotal-brauer-picard}
Let $\C$ be a pivotal tensor category and $\M$ an exact $\C$-module category. The relative Serre functor of $\M$ becomes a left $\C$-module functor with module constraint given by the composition
\[ X\tl \dS^\C_{\M}(M) \xrightarrow{\fp_X \tl\id} X^{**}\tl \dS^\C_{\M}(M) \xrightarrow{\eqref{sCXM}} \dS^\C_{\M}(X\tl M). \]
Similarly, for a bimodule category ${}_\C\M_\D$ over pivotal finite tensor categories, the relative Serre functors $\dS_\M^\C$ and $\dS_\M^{\oD}$ become bimodule functors,
\begin{equation}\label{eq:Serre_bimodule_functor}
     X\tl \dS^\C_{\M}(M) \tr Y \xrightarrow{\fp_X \tl\id\tr\fq_Y} X^{**}\tl \dS^\C_{\M}(M)\tr Y^{**} \xrightarrow{\eqref{eq:twisted_serre_bimod}} \dS^\C_{\M}(X\tl M\tr Y). 
\end{equation}
Moreover, for a $\C$-bimodule equivalence $H:\M\rightarrow\N$, using $H^\rra = H$, the twisted bimodule equivalence \eqref{eq:Serre_module_functor} becomes a bimodule equivalence
\begin{equation}\label{eq:Serre_bimodule_equivalence}
\Lambda_H: \dS_\N^\C\circ H \xRightarrow{~\sim~} H\circ \dS_\M^\C.
\end{equation}

\begin{definition}
Let $\C$ and $\D$ be pivotal tensor categories.
\label{def: piv strc module} 
\begin{enumerate}[(i)]
    \item \cite[Def.\ 3.11]{shimizu2019relative} A \textit{pivotal structure} on an exact left $\C$-module category $\M$ is a $\C$-module natural isomorphism $\tfp\colon\id_{\M}{\xRightarrow{~\sim\,}} \dS^{\C}_{\M}$. We call $\M$ along with $\tfp$ a pivotal $\C$-module category.
    \item \cite[Def.\ 5.1]{spherical2022} A \emph{pivotal $(\C,\D)$-bimodule category} is an exact  $(\C,\D)$-bimodule category ${}_\C\M_\D$ together with a bimodule natural isomorphisms $\tfp\colon \id_\M{\xRightarrow{~\sim\,}}\dS_\M^\C$ and $\tfq:\id_{\M}\xRightarrow{~\sim} \dS_{\M}^{\oD}$.
    \item An \emph{invertible pivotal bimodule category} is an invertible bimodule category ${}_\C\M_\D$ together with a bimodule natural isomorphism $\tfp\colon \id_\M{\xRightarrow{~\sim\,}}\dS_\M^\C$.
\end{enumerate}
\end{definition}

\begin{remark}\label{rem:Serre-left-right}
Definition \ref{def: piv strc module} (iii) only considers the relative Serre functor $\dS^{\C}_{\M}$ corresponding to the left action and not the relative Serre functor $\dS^{\oD}_{\M}$.
The notion of a pivotal bimodule category defined in \cite[Def.\ 5.1]{spherical2022} requires an additional trivialization $\tfq \colon \id_\M{\xRightarrow{~\sim\,}}\dS_\M^{\oD}$.
However, when $\M$ is an invertible $(\C,\D)$-bimodule category, these two relative Serre functors are related: there is a natural isomorphism $\odS^{\C}_{\M}\cong \dS^{\oD}_{\M}$ of twisted bimodule functors, according to \cite[Cor.\ ~4.12]{spherical2022}. Thus, a trivialization of $\dS_{\M}^{\C}$ also yields a trivialization of $\dS_{\M}^{\oD}$, endowing $\M$ with the structure of a pivotal bimodule category in the sense of \cite[Def.\ 5.1]{spherical2022}.
\end{remark}

\begin{remark}
\label{rem:nonpivotalizable more general}
Let $\C$ be a fusion category with a fixed pivotal structure $\mathfrak{p}$. Let $G=U(\C)$  be the universal grading group of $\C$ \cite{gelaki2008nilpotent}, and assume that $\mathrm{char}(\kk)$ does not divide $|U(\C)|$. The pivotal structures on $\C$ are in bijection with elements $\phi\in \widehat{G} \cong \Aut_\otimes(\id_\C)$. Let us denote $\mathfrak{p}^\phi$ the corresponding pivotal structure, i.e.
\begin{equation}
\label{eqn: other pivotal structures}
(\mathfrak{p}^\phi)_X= \phi(\deg X) \mathfrak{p}_X \qquad \text{for any homogeneous $X\in \C$}.
\end{equation}

A semisimple $\C$-module category $\M$ has a compatible grading by a transitive $G$-set $G/H$. Suppose that $\M$ has a pivotal structure $\tilde\fp$ with respect to $\mathfrak{p}$. Then the pivotal structures on $\M$ with respect
to $\mathfrak{p}^\phi$ are in bijection with functions $f: G/H \to \kk^\times$ satisfying $\phi(g)f(xH)=f(gxH),\,
g,x\in G$. Such functions exist (i.e. $M$ is pivotalizable with respect to $\mathfrak{p}^\phi$) if and only
$\phi|_H=1$.

The above setup applies, in particular, when $\C$ is pseudo-unitary, in which case it possesses a canonical pivotal structure \cite{etingof2005fusion}. Moreover, by \cite[Proposition 5.8]{schaumann2013traces}, every $\C$-module category $\M$ then also admits a canonical pivotal structure.
This framework yields numerous examples of non-pivotalizable module categories. For example, if $\M$ arises from a fiber functor on $\C$, then $H = G$, and so $\M$ is non-pivotalizable with respect to $\mathfrak{p}^\phi$ unless $\phi = 1$.
\end{remark}

\begin{definition}\label{def:PivotalBimodFunctor}
Let $(\M,\tfp^{\M})$ and $(\N,\tfp^{\N})$ be pivotal 
$\C$-bimodule categories. A $\C$-bimodule equivalence $H\colon\M\xrightarrow{\;\sim\;}\N$ is called \textit{pivotal} if the following diagram commutes.
\begin{equation}\label{eq:PivBrPic-1-cell}
	\begin{tikzcd}
			& {H} &  \\
			{\dS^\C_{\N}\circ H} & & {H\circ \dS^\C_{\M}}.
			\arrow[Rightarrow,"{H\circ\tfp^{\M}}", from=1-2, to=2-3]
			\arrow[Rightarrow,"{\tfp^{\N}\circ H}"', from=1-2, to=2-1]
			\arrow[Rightarrow,"{\eqref{eq:Serre_bimodule_equivalence}}"', from=2-1, to=2-3]
	\end{tikzcd}
\end{equation} 
\end{definition}

\begin{lemma}
Composition of pivotal $\C$-module equivalences is pivotal. 
\end{lemma}
\begin{proof}
Let $F\colon\M\rightarrow\N$ and $H\colon\N\rightarrow\cL$ be pivotal left $\C$-module equivalences. The commutativity of the following diagram establishes that $H\circ F$ is pivotal:
\[\begin{tikzcd}[row sep=small]
	& {H\circ F} \\
	{} \\
	{\dS^\C_{\cL}\circ H\circ F} & {H\circ\dS^\C_{\N}\circ F} & {H\circ F\circ\dS^\C_{\M}}
	\arrow["{{\tfp^{\cL}\circ HF}}"', Rightarrow, from=1-2, to=3-1]
	\arrow["{{H\circ\tfp^{\M}\circ F}}"{description}, Rightarrow, from=1-2, to=3-2]
	\arrow["{{HF\circ\tfp^{\M}}}", Rightarrow, from=1-2, to=3-3]
	\arrow["{{\eqref{eq:Serre_bimodule_equivalence}}}"', Rightarrow, from=3-1, to=3-2]
	\arrow["{{\eqref{eq:Serre_bimodule_equivalence}}}"', Rightarrow, from=3-2, to=3-3]
\end{tikzcd}\]
The left square commutes by \eqref{eq:PivBrPic-1-cell} precomposed with $F$ and the right square by $H$ applied to \eqref{eq:PivBrPic-1-cell}.
\end{proof}
\begin{definition}\label{def:PivBrPic}
Let $\C$ be a pivotal finite tensor category. The \textit{pivotal Brauer-Picard $2$-groupoid of $\C$} is the $2$-groupoid $\PivBrPic(\C)$ whose objects are invertible pivotal $\C$-bimodule categories, $1$-cells are pivotal $\C$-bimodule equivalences and $2$-cells are bimodule natural isomorphisms.
\end{definition}
There is a canonical forgetful $2$-functor
\begin{equation}\label{eq:forget_BrPic}
    \forg\colon \PivBrPic(\C)\longrightarrow\BrPic(\C),\quad (\M,\tfp)\longmapsto \M
\end{equation}
that is faithful in $1$-morphisms and fully faithful in $2$-morphisms.

Given $\M, \N\in \PivBrPic(\C)$ their relative Deligne product $\M\btC \N$ inherits a pivotal structure $\fp^{\mathcal M \btC \mathcal N}$ given by the composition 
\begin{equation}\label{eq:pivotal_Deligne_pr}
\id_{\M\btC \N}=
\id_{\M}\btC \id_{\N} \xRightarrow{\;\tfp^{\M}\btC\, \tfp^{\N}\;}
   {\dS^\C_{\M} \btC \dS^\C_{\N}} 
   \xRightarrow{\;\mu_{\M,\N}\;}
   \dS^\C_{\M \btC \N}
\end{equation}
where $\mu_{\M,\N}$ is the natural isomorphism from \eqref{eq:Serre_monoidal}.

\begin{proposition}
\label{prop: endows}
The relative Deligne product endows $\PivBrPic(\C)$ with the structure of a $2$-categorical group.
Moreover, the forgetful $2$-functor $\forg$ becomes a monoidal $2$-functor.
\end{proposition}
\begin{proof}
For both claims, it suffices to show that $\PivBrPic(\C)$ is closed under the relative Deligne product. By definition every object in $\PivBrPic(\C)$ is invertible. For objects, this can be proven using the equivalence $\M\btC\N \simeq \Rex_{\C}(\M^{\op},\N)$ and invoking \cite[Proposition~5.4]{spherical2022}. An easy check shows that the pivotal structure assigned to $\M\btC \N$ in \cite{spherical2022} is same as the one described in \eqref{eq:pivotal_Deligne_pr}.

Lastly, we check that the product $\btC$ of pivotal module equivalences is pivotal. Take pivotal $\C$-bimodule categories $\M_1,\M_2,\N_1,\N_2\in\PivBrPic(\C)$ and pivotal $\C$-bimodule functors $F_1\colon\M_1\rightarrow\N_1$ and $F_2\colon\M_2\rightarrow\N_2$. Then $F_1\btC F_2\colon\M_1\btC\M_2\rightarrow\N_1\btC\N_2$ is pivotal because the following diagram commutes (we have suppressed $\btC$ in the diagram):
\[\begin{tikzcd}[column sep=large,row sep=scriptsize]
	& {F_1 F_2} \\
	\\
	{F_1 (\dS_{\N_2}\circ F_2)} & {F_1 (F_2\circ\dS_{\M_2})} & {(F_1\circ \dS_{\M_1})F_2} \\[2.5em]
	{(\dS_{\N_1}\circ F_1)(\dS_{\N_2}\circ F_2)} & {(F_1\circ \dS_{\M_1})(\dS_{\N_2}\circ F_2)} & {(F_1\circ \dS_{\M_1})(F_2\circ \dS_{\M_2})} \\
	{(\dS_{\N_1}\dS_{\N_2})\circ(F_1F_2)} && {(F_1F_2)\circ(\dS_{\M_1}\dS_{\M_2})} \\
	{\dS_{\N_1\N_2}\circ(F_1F_2)} && {(F_1F_2)\circ\dS_{\M_1\M_2}}
	\arrow["{F_1(\tfp_{\N_2}\circ F_2)}"', Rightarrow, from=1-2, to=3-1]
	\arrow["F_1(F_2\circ\tfp_{\M_2})"{description},Rightarrow, from=1-2, to=3-2]
	\arrow["{(F_1\circ \tfp_{M_2})F_2}", Rightarrow, from=1-2, to=3-3]
	\arrow["{F_1 \Lambda}", Rightarrow, from=3-1, to=3-2]
	\arrow["{(\tfp_{\N_1}\circ F_1) (\dS_{\N_2}\circ F_2)}"', Rightarrow, from=3-1, to=4-1]
	\arrow["{(F_1\circ \tfp_{\M_1})(\dS_{\N_2}\circ F_2)}"{description}, Rightarrow, from=3-1, to=4-2]
	\arrow["{(F_1\circ \tfp_{\M_1})(F_2\circ \dS_{\M_2})}"{description}, Rightarrow, from=3-2, to=4-3]
	\arrow["{(F_1\circ \dS_{\M_1})(F_2\circ\tfp_{\M_2})}", Rightarrow, from=3-3, to=4-3]
	\arrow["\Lambda"', Rightarrow, from=4-1, to=4-2]
	\arrow[equals, from=4-1, to=5-1]
	\arrow["\Lambda"', Rightarrow, from=4-2, to=4-3]
	\arrow[equals, from=4-3, to=5-3]
	\arrow["\mu"', Rightarrow, from=5-1, to=6-1]
	\arrow["\mu", Rightarrow, from=5-3, to=6-3]
	\arrow["\Lambda"', Rightarrow, from=6-1, to=6-3]
\end{tikzcd}\]
The two triangles commute by \eqref{eq:PivBrPic-1-cell} and the top two squares commute by level exchange. 
The bottom square commutes by Lemma~\ref{lem:Serre-monoidal}.
\end{proof}

Let $\pi_0(\BrPic(\C))$ (respectively, $\pi_0(\PivBrPic(\C))$) denote the underlying group of the isomorphism classes of objects of the $2$-categorical group $\BrPic(\C)$
(respectively, $\PivBrPic(\C)$).

For any $\C$-bimodule category $\M$, the relative Serre functor $\dS^\C_\M$ is a $\C$-bimodule autoequivalence of $\M$. When $\M$ is invertible, the latter is given by $Z_\M \tl -$, for an  invertible object $Z_\M$ of
$(\C\boxtimes \C^{\op})^*_\M \simeq \Z(\C)$ defined up to isomorphism.
Thus, we have a map
\begin{equation}
\label{eqn: hom S}
S\colon \pi_0(\BrPic(\C))\longrightarrow \Inv(\Z(\C)),\; \M \longmapsto Z_\M\,.
\end{equation}
Let 
\begin{equation}
\label{eqn: Aut-Pic iso}
\partial\colon \pi_0(\BrPic(\C))\longrightarrow \pi_0(\Aut^{br}(\Z(\C))),\; \M \longmapsto \partial_\M
\end{equation}
denote the canonical homomorphism, which is completely determined by the following isomorphism of $\C$-bimodule endofunctors of $\M$:
\begin{equation}
\label{eqn: partial}
\partial_\M(Z) \tl -  \cong -\tr Z, \qquad Z \in \Z(\C)\,.
\end{equation}
In particular, $\pi_0(\BrPic(\C))$ acts by automorphisms on the group $\Inv(\Z(\C))$.

\begin{proposition}
\label{Z is a 1-cocycle}
The function \eqref{eqn: hom S} is a $1$-cocycle, i.e it satisfies
\begin{equation}
Z_{\M \boxtimes_\C \N} \cong Z_\M \otimes \partial_\M(Z_\N)
\end{equation}
for all invertible $\C$-bimodule categories $\M,\, \N$. We have $Z_\M\cong \mathbf{1}$ if and only if
$\M$ admits a pivotal structure.
\end{proposition}
\begin{proof}
By Proposition~\ref{prop:Serre_monoidal}, $\dS^\C_{\M\boxtimes_\C \N} \cong \dS^\C_{\M} \boxtimes_\C \dS^\C_{\N}$. Using \eqref{eqn: partial}, we see that $\dS^\C_{\M}\boxtimes_\C \dS^\C_{\N}$ is given by
 \[
 (Z_\M \tl -) \boxtimes_\C (Z_\N \tl -) = (Z_\M \otimes \partial_\M(Z_\N)) \tl -\,,
 \]
while $\dS^\C_{\M\boxtimes_\C \N}$ is given by $Z_{\M\boxtimes_\C \N}$, so the statement follows.  The second is simply a restatement of the definition of a module pivotal structure.  
\end{proof}

\begin{remark}
Note that although $S$ is, in general,  not a group homomorphism, its kernel is a subgroup of $\pi_0(\BrPic(\C))$, namely the image of $\pi_0(\PivBrPic(\C))$
under the forgetful homomorphism.
\end{remark}

Let \(\C\) be a pseudo-unitary fusion category.  
By~\cite{etingof2005fusion}, it carries a canonical pivotal structure \(\fp\).  
Furthermore, as noted in Remark~\ref{rem:nonpivotalizable more general}, every \(\C\)-module category is pivotal with respect to \(\fp\).  
The pivotal structures \(\fp^\phi\) on \(\C\) are parameterized by \(\phi \in \widehat{U(\C)}\), as in~\eqref{eqn: other pivotal structures}.  A choice of $\phi$ gives a central object \(Z_\phi \in \Z(\C)^\times\) defined as the unit object \(\unit_\C\) equipped with the half-braiding  
\(X \otimes Z_\phi \to Z_\phi \otimes X\) given by \(\phi(\deg(X))\)  
for any \(X \in \C\) homogeneous with respect to the universal grading.  

Let \(\dS^{\C}_\M\) denote the relative Serre \(\C\)-bimodule endofunctor of \(\M\)  
with bimodule structure given by \eqref{eq:Serre_bimodule_functor} where we consider the pivotal structure \(\fp^\phi\) both on the left and right actions. Since, by Remark~\ref{rem:nonpivotalizable more general}, $\M$ is pivotalizable with respect to $\fp$, we have that $\dS^{\C}_\M$ is given by  
\(Z_\phi \triangleright - \triangleleft Z_\phi^{-1}\). In view of~\eqref{eqn: partial}, we have  
\[
\dS^{\C}_\M(M)
= Z_\phi \triangleright M \triangleleft Z_\phi^{-1}
= (Z_\phi \otimes \partial_\M(Z_\phi)^{-1}) \triangleright M,
\qquad M \in \M.
\]
This  proves the following result.

\begin{proposition}
\label{prop: fixes Z}
Let $\C$ be a pseudo-unitary fusion category.
An invertible $\C$-bimodule category $\M$ is pivotalizable with respect to $\fp^\phi$
if and only if $\partial_\M \in \Aut^{br}(\Z(\C))$ fixes $Z_\phi$.
\end{proposition}

\begin{corollary}
\label{cor: pi0(PivBrPic(C))}
The image of 
$\pi_0(\PivBrPic(\C,\, \fp^\phi))$ in  $\pi_0(\BrPic(\C))\cong \Aut^{br}(\Z(\C))$ is isomorphic to 
$\{\alpha \in \Aut^{br}(\Z(\C)) \mid \alpha(Z_\phi)\cong Z_\phi \}$.
\end{corollary}

\begin{remark}
Corollary~\ref{cor: pi0(PivBrPic(C))} points to an interpretation of the groups 
\(\pi_0(\PivBrPic(\C,\, \fp^\phi))\) 
as analogs of maximal parabolic subgroups of the orthogonal group. 
Indeed, when \(\C = \Rep((\mathbb{Z}/p\mathbb{Z})^n)\) for a prime \(p\), we have 
\(\pi_0(\BrPic(\C)) \cong \pi_0(\Aut^{br}(\Z(\C))) \cong O((\mathbb{Z}/p\mathbb{Z})^{2n})\). 
Furthermore, for a nontrivial linear character 
\(\phi: (\mathbb{Z}/p\mathbb{Z})^n \to \kk^\times\), 
the group \(\pi_0(\PivBrPic(\C,\, \fp^\phi))\) 
coincides with the stabilizer of the isotropic subspace generated by \(\phi\).
\end{remark}


\subsection{Realization of $\PivBrPic(\C)$ as fixed points}\label{subsec:homotopy-fixed-points}

Let $\C$ be a pivotal finite tensor category.
In this section, we define a $2$-categorical \textit{monoidal} $B\underline{\zZ}$-action on $\BrPic(\C)$ whose fixed points capture the pivotal Brauer-Picard $2$-categorical group $\PivBrPic(\C)$. To define such a monoidal action, it suffices to consider a monoidal pseudo-natural autoequivalence of the identity $2$-functor of $\BrPic(\C)$, as explained in Example \ref{ex: BZ and BZ2 actions}. These data comes from the relative Serre functors and the pivotal structure of $\C$. Explicitly:
\begin{itemize}
    \item Every object $\M\in\BrPic(\C)$ is in particular an exact $\C$-module category by \cite[Cor.~5.2]{davydov2021braided}, and thus we can consider the relative Serre functor
    \[\bS_{\M}\coloneqq \dS^{\C}_{\M}\colon \M\longrightarrow \M\]
    from Definition \ref{def:rel_Serre_functor} with $\C$-bimodule structure given by \eqref{eq:Serre_bimodule_functor}.
    \item For every $\C$-bimodule equivalence $H\colon \M\longrightarrow\N$, consider the equivalence 
    \[\bS_H\coloneqq\Lambda_H \colon \bS_{\N}\circ H \xRightarrow{~\sim~} H\circ \bS_{\M}
    \]
    of $\C$-bimodule functors from \eqref{eq:Serre_bimodule_equivalence}.
\end{itemize}

\begin{proposition}
\label{prop:RSpNE-is-Monoidal}
The relative Serre functors assemble into a monoidal pseudo-natural equivalence
\begin{align}\label{Serre_pseudo}
    \bS\colon \id_{\BrPic(\C)}\xRightarrow{~\sim~} \id_{\BrPic(\C)}
\end{align}
on the identity $2$-functor of $\BrPic(\C)$.
\end{proposition}

\begin{proof}
It follows from \cite[Theorem~3.10]{shimizu2019relative} that $\bS_{\id_{\M}}=\id_{\bS_{\M}}$ and $\bS_{F\circ H} = \bS_F \circ \bS_H$ for every $F\in\Rex_{\C|\C}(\M,\N)$ and $H\in\Rex_{\C|\C}(\mathcal{L},\M)$, and thus $\bS$ is a pseudo-natural equivalence. 
Now, we need to endow $\bS$ with the structure of a monoidal pseudo-natural equivalence \cite[Def.~2.13]{davydov2021braided}. To this end, consider for $\M,\N\in\BrPic(\C)$ the natural isomorphisms
\begin{equation}\label{eq:monoidal_Serre}
    \mu_{\M,\N}\colon \bS_{\M}\boxtimes_{\C} \bS_{\N} \xRightarrow{~\sim~} \bS_{\M\boxtimes_{\C} \N}
\end{equation}
coming from \eqref{eq:Serre_monoidal}. These must satisfy that
\begin{equation}\label{eq:monPNT-1}
    \begin{tikzcd}
        {\cL \M\N} &&&& \cL\M\N & {=} & \cL\M\N &&&& \cL\M\N
        \arrow[""{name=0, anchor=center, inner sep=0}, "{\bS_{\cL}\bS_{\M}\bS_{\N}}", curve={height=-40pt}, from=1-1, to=1-5]
        \arrow[""{name=1, anchor=center, inner sep=0}, "{\bS_{\cL\M\N}}"', curve={height=40pt}, from=1-1, to=1-5]
        \arrow[""{name=2, anchor=center, inner sep=0}, "{\bS_{\cL}\bS_{\M\N}}"{description}, from=1-1, to=1-5]
        \arrow[""{name=3, anchor=center, inner sep=0}, "{\bS_{\cL}\bS_{\M}\bS_{\N}}", curve={height=-40pt}, from=1-7, to=1-11]
        \arrow[""{name=4, anchor=center, inner sep=0}, "{\bS_{\cL\M\N}}"', curve={height=40pt}, from=1-7, to=1-11]
        \arrow[""{name=5, anchor=center, inner sep=0}, "{\bS_{\cL\M}\bS_{\N}}"{description}, from=1-7, to=1-11]
        \arrow["{\mu_{\cL\M,\N}}", shorten <=4pt, shorten >=4pt, Rightarrow, from=5, to=4]
        \arrow["{\mu_{\M,\N}}", shorten <=4pt, shorten >=4pt, Rightarrow, from=0, to=2]
        \arrow["{\mu_{\cL,\M}}", shorten <=4pt, shorten >=4pt, Rightarrow, from=3, to=5]
        \arrow["{\mu_{\cL,\M\N}}", shorten <=4pt, shorten >=4pt, Rightarrow, from=2, to=1]
    \end{tikzcd}
\end{equation}
for $\M,\N,\cL \in\BrPic(\C)$. Condition (\ref{eq:monPNT-1}) follows from the fact that both $\bS_{\cL}\bS_{\M}\bS_{\N}$ and $\bS_{\cL\M\N}$ are relative Serre functors of $\cL\M\N$ and relative Serre functors are unique up to unique isomorphism \cite[Lemma~3.5]{shimizu2019relative}. Indeed, if we check that on the one hand the isomorphism $\mu_{\cL,\M\N}\circ (\id_{\bS_\cL}\btC\mu_{\M,\N})$ on the left of \eqref{eq:monPNT-1} fulfills
\begin{equation}\label{phiL,MN}
    \phi^{\cL}\star(\phi^\M\star\phi^\N) = \phi^{\cL\btC\M\btC\N}\circ \uHom_{\cL\btC\M\btC\N} \bigg( L \boxtimes M\boxtimes N,\, \mu_{\cL,\M\N}\circ (\id_{\bS_\cL}\btC\mu_{\M,\N}) \bigg)\,,
\end{equation}
and that on the other hand $\mu_{\cL\M,\N}\circ (\mu_{\cL,\M}\btC\id_{\bS_\N})$ on the right of \eqref{eq:monPNT-1} obeys
\begin{equation}\label{phiLM,N}
    (\phi^{\cL}\star\phi^\M)\star\phi^\N = \phi^{\cL\btC\M\btC\N}\circ \uHom_{\cL\btC\M\btC\N} \bigg(L\boxtimes M\boxtimes N,\, \mu_{\cL\M,\N}\circ (\mu_{\cL,\M}\btC\id_{\bS_\N}) \bigg)
\end{equation}
then equation \eqref{eq:monPNT-1} would follow by the uniqueness in \cite[Lemma~3.5]{shimizu2019relative}. By a similar argument as in Lemma~\ref{lem:inner-hom-Serre}, it suffices to do this check on simple tensors.

From the definition of $\phi^{\cL}\star(\phi^\M\star\phi^\N)$ (see Lemma~\ref{lem:inner-hom-Serre}), equation \eqref{phiL,MN} becomes
{\small
\begin{equation}\label{monoidal_1}
    \begin{tikzcd}[column sep=normal, row sep=scriptsize]
        \uHom_{\cL\M\N} \big(LMN, \bS_{\cL}(L')\bS_{\M}(M')\bS_{\N}(N')\big) \ar[rr,"\eqref{eq:inner-hom-Deligne}"]\ar[d,"\id\btC\mu_{\M,\N}",swap]
        &&
        \uHom_{\cL} \big(L \tr {}^*\uHom_{\M\N}(MN,\bS_{\M}(M')\bS_{\N}(N')),\bS_{\cL}(L')\big)\ar[d,"\phi^\M\star\phi^\N"]\\
        \uHom_{\cL\M\N} \big(LMN, \bS_{\cL}(L')\bS_{\M\N}(M'N')\big)\ar[d,"\mu_{\cL,\M\N}",swap]&&
        \uHom_{\cL} \big(L \tr \uHom_{\M\N}(M'N',MN),\bS_{\cL}(L')\big)\ar[d,"\phi^\cL"]\\
        \uHom_{\cL\M\N} \big(LMN, \bS_{\cL\M\N}(L'M'N')\big)\ar[d,swap,"\phi^{\cL\btC\M\btC\N}"]
        &&\uHom_{\cL} \big(L',L \tr \uHom_{\M\N}(M'N',MN)\big)^*\ar[d,"\cong"]\\
        \uHom_{\cL\M\N} \big(L'M'N',LMN\big)^*&&\uHom_{\cL} \big(L'\tr {}^*\uHom_{\M\N}(M'N',MN),L\big)^*\ar[ll,"\eqref{eq:inner-hom-Deligne}"]
    \end{tikzcd}
\end{equation}}
\noindent Now equation \eqref{mu_MN} says that the isomorphisms $\mu_{\M,\N}$ and $\mu_{\cL,\M\N}$ fulfill
\begin{equation*}
    \phi^\M\star\phi^\N=\phi^{\M\btC\N}\circ \uHom_{\M\N}(MN,\mu_{\M,\N})\quad\text{ and }\quad \phi^\cL\star\phi^{\M\btC\N}=\phi^{\cL\btC\M\btC\N}\circ \uHom_{\cL\M\N}(MN,\mu_{\cL,\M\N})
\end{equation*}
respectively. This turns diagram \eqref{monoidal_1} into
{\small
\begin{equation*}
    \begin{tikzcd}[column sep=normal, row sep=scriptsize]
        \uHom_{\cL\M\N} \big(LMN, \bS_{\cL}(L')\bS_{\M}(M')\bS_{\N}(N')\big) \ar[rr,"\eqref{eq:inner-hom-Deligne}"]\ar[d,"\id\btC \,\mu_{\M,\N}",swap]
        &&
        \uHom_{\cL} \big(M \tr {}^*\uHom_{\M\N}(MN,\bS_{\M}(M')\bS_{\N}(N')),\bS_{\cL}(L')\big)\ar[d,"\id\btC \,\mu_{\M,\N}"]
        \\
        \uHom_{\cL\M\N} \big(LMN, \bS_{\cL}(L')\bS_{\M\N}(M'N')\big)\ar[rr,"\eqref{eq:inner-hom-Deligne}"]\ar[ddd,swap,"\phi^\cL\star\phi^{\M\btC\N}"]
        &&
        \uHom_{\cL} \big(L \tr {}^*\uHom_{\M\N}(MN,\bS_{\M\N}(M'N')),\bS_{\cL}(l')\big)
        \ar[d,"\;\phi^{\M\btC\N}"]
        \\
        &&
        \uHom_{\cL} \big(L \tr \uHom_{\M\N}(M'N',MN),\bS_{\cL}(L')\big)\ar[d,"\quad\phi^\cL"]
        \\
        &&
        \uHom_{\cL} \big(L',L \tr \uHom_{\M\N}(M'N',MN)\big)^*\ar[d,"\cong"]
        \\
        \uHom_{\cL\M\N} \big(L'M'N',LMN\big)^*&&\uHom_{\cL} \big(L'\tr {}^*\uHom_{\M\N}(M'N',MN),L\big)^*\ar[ll,"\eqref{eq:inner-hom-Deligne}"]
    \end{tikzcd}
\end{equation*}}
where the top square commutes due to naturality of \eqref{eq:inner-hom-Deligne} and the bottom square commutes due to the definition of $\phi^\cL\star\phi^{\M\btC\N}$, and thus proving that equation \eqref{phiL,MN} holds. A similar argument implies that \eqref{phiLM,N} holds, as well.
\end{proof}

\begin{corollary}\label{BZ_action_BrPicC}
Let $\C$ be a pivotal finite tensor category. Then the relative Serre pseudo-natural equivalence \eqref{Serre_pseudo} defines a monoidal $B\underline{\zZ}$-action on $\BrPic(\C)$. 
\end{corollary}
\begin{proof}
From Proposition \ref{prop:RSpNE-is-Monoidal} we consider the monoidal pseudo-natural equivalence $\bS$. As explained in Example \ref{ex: BZ and BZ2 actions}, we are left to show that the self-braiding \eqref{eq:self_braiding} is trivial for $\bS$. Explicitly, given $\M\in \BrPic(\C)$ we have to verify that the isomorphism $\Lambda_{\dS_\M}$ from \eqref{eq:Serre_bimodule_equivalence} is the identity. This is defined by means of the module Yoneda lemma applied to the composition \cite[Thm.\ 3.10]{shimizu2019relative}
\begin{equation}\label{eq:def_self_br_serre}
\begin{aligned}
    \uHom_\M(N,\dS_\M^2(M))\xrightarrow{\eqref{eq:phi_Serre}}
    \uHom_\M(\dS_\M(M),N)^*&\xrightarrow{\psi^*_\dS}\uHom_\M(M,\overline{\dS}_\M(N))^*\\
    &\xrightarrow{\eqref{eq:phi_Serre}}  \uHom_\M(\overline{\dS}_\M(N),\dS_\M(M))\xrightarrow{\psi_{\overline{\dS}}} \uHom_\M(N,\dS_\M^2(M))
\end{aligned}
\end{equation}
where $\psi_F$ for a module functor $F\colon\M\to\N$ is defined in \cite[Rem.\ 2.10]{shimizu2019relative} as the composition
\begin{equation*}
    \uHom_\N(F(M),N)\xrightarrow{\underline{F}^{\ra}}\uHom_\N(F^\ra \circ F(M),F^\ra(N))\xrightarrow{~\eta\,~~}\uHom_\N(M,F^\ra(N))
\end{equation*}
with $\eta$ being the unit of the adjunction $F\dashv F^\ra$. We show now that \eqref{eq:def_self_br_serre} is the identity. Consider the diagram
\begin{equation*}
\begin{tikzcd}
    \uHom_\M(N,\dS_\M^2(M)) \ar[r,"\eqref{eq:phi_Serre}"]\ar[d,"\varepsilon",swap]
    & \uHom_\M(\dS_\M(M),N)^*\ar[r,"\psi^*_\dS"]\ar[d,"\varepsilon",swap]
    &\uHom_\M(M,\overline{\dS}_\M(N))^* \ar[dl,"\underline{\dS}",swap]\ar[d,"\eqref{eq:phi_Serre}"]\\
    \uHom_\M(\dS_\M\overline{\dS}_\M(N),\dS_\M^2(M))\ar[r,"\eqref{eq:phi_Serre}"]
    \ar[dr,"\id", swap] 
    &\uHom_\M(\dS_\M(M),\dS_\M\overline{\dS}_\M(N))^*\ar[r,"\eqref{eq:phi_Serre}"] \ar[d,"\eqref{eq:phi_Serre}"]
    &\uHom_\M(\overline{\dS}_\M(N),\dS_\M(M))\ar[dl,"\underline{\dS}",swap]
    \ar[d,"\psi_{\overline{\dS}}"]\\
    &\uHom_\M(\dS_\M\overline{\dS}_\M(N),\dS_\M^2(M))\ar[r,"\eta",swap]&
    \uHom_\M(N,\dS_\M^2(M))
\end{tikzcd}
\end{equation*}
where the top-left square commutes due to naturality of \eqref{eq:phi_Serre} and the lower-left triangle since we are using an instance of \eqref{eq:phi_Serre} and its inverse. On the right side, we have four triangles. The two triangles involving \eqref{eq:phi_Serre} and $\underline{\dS}$ commute because of \cite[Lemma 3.18]{shimizu2019relative}. The remaining triangles commute from the definition of 
$\psi_\dS$ and its inverse \cite[Rem.\ 2.10]{shimizu2019relative}. Altogether, we have that \eqref{eq:def_self_br_serre} is $\eta\circ\varepsilon$ which gives the identity morphism, since they are the unit and counit witnessing the biadjunction of the equivalence $\dS_\M$ and its quasi-inverse.
\end{proof}
We recover the pivotal Brauer-Picard $2$-categorical group as fixed points for this action.
\begin{proposition}\label{prop:PivBrPic-is-BZ-equivariant}
Let $\C$ be a pivotal finite tensor category. The pivotal Brauer-Picard $2$-categorical group $\PivBrPic(\C)$ is monoidally $2$-equivalent to the $2$-categorical group of $B\underline{\zZ}$-fixed points of $\BrPic(\C)$.
\end{proposition} 
\begin{proof}
Using the definition of a fixed point (see Example~\ref{ex: fixed point for BZ and BZ2 actions}), an object of $\BrPic(\C)^{B\underline{\zZ}}$ is a pair $(\M,\tfp)$ where $\M\in\BrPic(\C)$ and $\tfp\colon\id_{\M} \xRightarrow{~\sim\,} \bS_\M=\dS_{\M}^{\C}$ is an invertible $2$-morphism in $\BrPic(\C)$, i.e. a $\C$-bimodule natural isomorphism. This is the datum of a pivotal structure of $\M$ and thus of an object in $\PivBrPic(\C)$. Next, $1$-cells between $(\M,\tfp^{\M})$ and $(\N,\tfp^{\N})$ are those $1$-cells $H\colon\M\rightarrow\N$ in $\BrPic(\C)$ (that is, $\C$-bimodule equivalences) which commute $\tfp^{\M}$ and $\tfp^{\N}$. These are precisely those $H$ which satisfy \eqref{eq:PivBrPic-1-cell}. Hence, the $1$-cells match.
At the level of $2$-cells, there is no additional condition to be fulfilled with respect to the fixed point data. Monoidality follows since the monoidal structures of both $\BrPic(\C)^{B\underline{\zZ}}$ and $\PivBrPic(\C)$ are given by the relative Deligne product of pivotal invertible bimodule categories.
\end{proof}


\subsection{The classification of pivotal $G$-graded extensions}\label{subsec:classification-pivotal-G}
\label{sec: pivotal graded extensions}

Recall that, according to Theorem \ref{thm:ext_classification}, there is an equivalence 
\begin{equation}
\mathrm{E}\colon \Ext(G,\C) \longrightarrow \MonFun\left(\uuG,\BrPic(\C)\right)
\end{equation}
of $2$-groupoids for a given finite tensor category $\C$. By Corollary \ref{cor: BZ action on Ext}, we have a $B\underline{\zZ}$-action on $\Ext(G,\C)$.
Also by Corollary ~\ref{BZ_action_BrPicC}, we get a monoidal $B\underline{\zZ}$-action on the $2$-categorical group $\BrPic(\C)$. This defines a $B\underline{\zZ}$-action on the $2$-groupoid $\MonFun\left(\uuG,\BrPic(\C)\right)$ determined by the pseudo-natural equivalence
\begin{equation}\label{MonFun_pseudo}
    \begin{tikzcd}
        {\MonFun\left(\uuG,\BrPic(\C)\right)} && {\MonFun\left(\uuG,\BrPic(\C)\right)}
        \arrow[""{name=0, anchor=center, inner sep=0}, "\id", curve={height=-20pt}, from=1-1, to=1-3]
        \arrow[""{name=1, anchor=center, inner sep=0}, "{\id}"', curve={height=20pt}, from=1-1, to=1-3]
        \arrow["\bS\circ-", shorten <=5pt, shorten >=5pt, Rightarrow, from=0, to=1]
    \end{tikzcd}
\end{equation}
whose components are explicitly given for $F\in \MonFun\left(\uuG,\BrPic(\C)\right)$ by the whiskering
\begin{equation}
    \begin{tikzcd}
        \uuG\ar[rr,"F"]&&{\BrPic(\C)} && {\BrPic(\C)}
        \arrow[""{name=0, anchor=center, inner sep=0}, "\id", curve={height=-15pt}, from=1-3, to=1-5]
        \arrow[""{name=1, anchor=center, inner sep=0}, "{\id}"', curve={height=15pt}, from=1-3, to=1-5]
        \arrow["\bS", shorten <=5pt, shorten >=5pt, Rightarrow, from=0, to=1]
    \end{tikzcd}
\end{equation}
and for a $1$-cell $\eta\colon F\Rightarrow H$ in $\MonFun\left(\uuG,\BrPic(\C)\right)$ by the horizontal composition
\begin{equation}
    \begin{tikzcd}
        \uuG && {\BrPic(\C)} && {\BrPic(\C)}
        \arrow[""{name=0, anchor=center, inner sep=0}, "F", curve={height=-20pt}, from=1-1, to=1-3]
        \arrow[""{name=1, anchor=center, inner sep=0}, "{H}"', curve={height=20pt}, from=1-1, to=1-3]
        \arrow["\eta", shorten <=5pt, shorten >=5pt, Rightarrow, from=0, to=1]
        \arrow[""{name=0, anchor=center, inner sep=0}, "\id", curve={height=-20pt}, from=1-3, to=1-5]
        \arrow[""{name=1, anchor=center, inner sep=0}, "{\id}"', curve={height=20pt}, from=1-3, to=1-5]
        \arrow["\bS", shorten <=5pt, shorten >=5pt, Rightarrow, from=0, to=1]
    \end{tikzcd}
\end{equation}
which can be interpreted as the modification $\id_\bS\circ\eta\colon\bS\circ F\xRightarrow{~\;\;}\bS\circ H$. We will next show that these two actions behave nicely under \eqref{eq:ext_classification}. 

\begin{proposition}\label{prop:phi-is-equivariant}
    The equivalence $\mathrm{E}\colon\Ext(G,\C) \xrightarrow{\;\simeq\;} \MonFun\left(\uuG,\BrPic(\C)\right)$ is $B\underline{\zZ}$-equivariant.
\end{proposition}
\begin{proof}
$B\underline{\zZ}$-equivariance of the $2$-functor $\mathrm{E}$ amounts to the existence of an invertible modification $\Omega$ between the pseudo-natural equivalences
\begin{equation}
    \begin{tikzcd}[column sep=2.5em]
       {\Ext(G,\C)}\ar[r,"\mathrm{E}"] &{\MonFun\left(\uuG,\BrPic(\C)\right)} && {\MonFun\left(\uuG,\BrPic(\C)\right)}
        \arrow[""{name=0, anchor=center, inner sep=0}, "\id", curve={height=-20pt}, from=1-2, to=1-4]
        \arrow[""{name=1, anchor=center, inner sep=0}, "{\id}"', curve={height=20pt}, from=1-2, to=1-4]
        \arrow["\bS\circ-", shorten <=5pt, shorten >=5pt, Rightarrow, from=0, to=1]
    \end{tikzcd}
\end{equation}
and
\begin{equation}
    \begin{tikzcd}
        {\Ext(G,\C)} && {\Ext(G,\C)} \ar[r,"\mathrm{E}"]& \MonFun\left(\uuG,\BrPic(\C)\right)
        \arrow[""{name=0, anchor=center, inner sep=0}, "\id", curve={height=-20pt}, from=1-1, to=1-3]
        \arrow[""{name=1, anchor=center, inner sep=0}, "{\id}"', curve={height=20pt}, from=1-1, to=1-3]
        \arrow["(-)^{**}", shorten <=5pt, shorten >=5pt, Rightarrow, from=0, to=1]
    \end{tikzcd}\,.
\end{equation}
For $\D\in \Ext(G,\C)$ we need an invertible $2$-cell in $\MonFun\left(\uuG,\BrPic(\C)\right)$ between $\bS\circ\mathrm{E}(\D)$ and $\mathrm{E}((-)^{**}_\D)$, i.e. a monoidal modification $\Omega_\D\colon\bS\circ\mathrm{E}(\D)\to\mathrm{E}((-)^{**}_\D)$. Define the components of $\Omega_\D$ for $g\in G$ by the natural isomorphism $\dS_{\D_g}^{\C} \xRightarrow{~\sim\,} (-)^{**}|_{\D_g}$ from equation \eqref{eq:Serre_homogeneous_components} (which is an isomorphism of module functors by using the pivotal structure of $\C$). The naturality on $1$-morphisms for $\Omega_\D$ is immediate since all $1$-morphisms in $\uuG$ are trivial. 
Now, we will consider the identification $\otimes_{g,h}\colon \D_{g}\boxtimes_{\C}\D_{h}\simeq \D_{gh}$ coming from the tensor product of $\D$. Then, the monoidality of $\Omega_\D$ is the condition that 
\begin{equation}\label{eq:mono_OmegaD}
{\scriptstyle
\begin{tikzcd}[column sep=normal, row sep=scriptsize]
{\D_g}\boxtimes_{\C}{\D_h}\ar[""{name=L},curve={height=20pt}, swap, from=1-1, to=3-1,"\bS_{\D_g}\boxtimes_{\C} \bS_{\D_h}"]
\ar[""{name=R},curve={height=-20pt}, from=1-1, to=3-1,"\bS_{\D_g\boxtimes_{\C} \D_h}"]
\arrow["\mu_{g,h}", shorten <=5pt, shorten >=5pt, Rightarrow, from=L, to=R]
\ar[rrr,"\otimes_{g,h}"]  &&&{\D_{gh}}
\ar[dd,"\bS_{\D_{\!gh}}",""{name=L2}, bend right=35pt,swap]\ar[dd,"(-)^{**}_{\D_{\!gh}}",""{name=R2}, bend left=35pt]
\arrow["\Omega_{g,h}", shorten <=5pt, shorten >=5pt, Rightarrow, from=L2, to=R2]    \\
\\
{\D_g}\boxtimes_{\C}{\D_h}\ar[rrr,"\otimes_{g,h}",swap]  
&&&{\D_{gh}}
\end{tikzcd}
\;=\;
\begin{tikzcd}[column sep=normal, row sep=normal]
{\D_g}\boxtimes_{\C}{\D_h}
\ar[""{name=L},curve={height=25pt}, swap, from=1-1, to=3-1,"\bS_{\D_g}\btC \bS_{\D_h}"]
\ar[""{name=R},curve={height=-25pt}, from=1-1, to=3-1,"(-)^{**}_{\D_g} \btC (-)^{**}_{\D_h}"]
\arrow["\Omega_g\btC\Omega_h\,", shorten <=5pt, shorten >=5pt, Rightarrow, from=L, to=R]
\ar[rrr,"\otimes_{g,h}"]  &&&{\D_{gh}}\ar[dd,"(-)^{**}_{\D_{\!gh}}"]\\\\
{\D_g}\boxtimes_{\C}{\D_h}\ar[rrr,"\otimes_{g,h}",swap]  &&&{\D_{gh}}
\end{tikzcd}  }  
\end{equation}
for every $g,h\in G$, where the $2$-cell filling the square diagram on the right hand side of \eqref{eq:mono_OmegaD} is the isomorphism $\nu_{g,h}\colon (X_g\otimes X_h)^{**}\cong X_g^{**}\otimes X_h ^{**}$ coming from the monoidal structure of the double-dual functor of the tensor category $\D$. This condition is the commutativity of the diagram
\begin{equation}
    \begin{tikzcd}[column sep=normal, row sep=small]
\otimes_{g,h}\circ \bS_{\D_g}\btC\bS_{\D_h}\ar[d,swap,"\mu_{g,h}"]\ar[rr,"(\Omega_\D)_g\otimes(\Omega_\D)_h"]  &&\otimes_{g,h}\circ(-)^{**}_{\D_g}\boxtimes_{\C}(-)^{**}_{\D_h}
    \ar[dd,"\nu_{g,h}"]\\
\otimes_{g,h}\circ\bS_{\D_{g}\btC\D_h}\ar[d,swap,"\bS_\otimes"]&&\\
    \bS_{\D_{\!gh}}\circ\otimes_{g,h}\ar[rr,"(\Omega_\D)_{gh}",swap] &&(-)^{**}_{\D_{\!gh}}\circ \otimes_{g,h}
    \end{tikzcd}\label{eq:mono_OmegaD2}
\end{equation}
By definition $\mu_{g,h}$ is the unique isomorphism obeying that
\begin{equation}\label{eq:property_mu}
\phi^g\star\phi^h=\phi^{\D_g\btC\D_h}\circ \uHom_{\D_{\!gh}}(X_1\otimes X_2,\mu_{g,h})   \,.
\end{equation}
Thus, the commutativity of the diagram \eqref{eq:mono_OmegaD2} can be derived by proving that the composition $\Theta\coloneqq \bS_\otimes^{-1}\circ(\Omega_\D)_{gh}^{-1}\circ\nu_{g,h}\circ  (\Omega_\D)_g\otimes(\Omega_\D)_h$ also fulfills \eqref{eq:property_mu}.
To show this consider the diagram
\begin{equation}\label{eq:diagramOmegaPhiPsi}
\begin{tikzcd}[column sep=8em,row sep=2em]
    \uHom_{\D_{\!gh}}(X_1\otimes X_2,\otimes_{g,h}\circ\bS_{\D_g\D_h}(Y_1\otimes Y_2))
    \ar[r,"\otimes_{g,h}\;\circ\;\phi^{\,\D_g\D_h}"]\ar[d,swap,"\bS_\otimes"]
    & \uHom_{\D_{\!gh}}(Y_1\otimes Y_2,X_1\otimes X_2)^*\\
    \uHom_{\D_{\!gh}}(X_1\otimes X_2,\bS_{\D_{\!gh}}(Y_1\otimes Y_2))\ar[d,"(\Omega_\D)_{gh}",swap] \ar[ur,"\phi^{{gh}}\,\circ\,\otimes_{g,h}"{description}]
    & \uHom_{\D_{\!gh}}(X_1\otimes X_2,\bS_g(Y_1)\otimes \bS_h(Y_2))\ar[d,"(\Omega_\D)_g\otimes(\Omega_\D)_h"]\ar[u," \otimes_{g,h}\;\circ\;\phi^g\star\phi^h",swap]\\
    \uHom_{\D_{\!gh}}(X_1\otimes X_2,(Y_1\otimes Y_2)_{gh}^{**}) 
    \ar[uur,"\psi^{{gh}}"{description}]
    & \uHom_{\D_{\!gh}}(X_1\otimes X_2,(Y_1)_{g}^{**}\otimes( Y_2)_{h}^{**})\ar[l,"\nu_{g,h}"]
\end{tikzcd}    
\end{equation}
where the upper triangle commutes since $\otimes_{g,h}$ is a bimodule equivalence and the middle triangle commutes since $(\Omega_\D)_{gh}$ is an isomorphism of relative Serre functors. Now, since $\nu_{g,h}$ is an isomorphism of relative Serre functors we have that
\begin{equation}\label{eq:property_nu}
\psi^g\star\psi^h=\psi^{{gh}}\circ \uHom_{\D_{\!gh}}(X_1\otimes X_2,\nu_{g,h})   \,.
\end{equation}
which reduces the lower-right triangle in the diagram \eqref{eq:diagramOmegaPhiPsi} to
\begin{equation}
\begin{tikzcd}[column sep=-2em,row sep=2em]
\uHom_{\D_{\!gh}}(Y_1\otimes Y_2,X_1\otimes X_2)^*&& \uHom_{\D_{\!gh}}(X_1\otimes X_2,\bS_g(Y_1)\otimes \bS_h(Y_2))
\ar[ll,"\otimes_{g,h}\;\circ\;\phi^g\star\phi^h",swap]
\ar[dl,"(\Omega_\D)_g\otimes(\Omega_\D)_h"]
\\
& \uHom_{\D_{\!gh}}(X_1\otimes X_2,(Y_1)_{g}^{**}\otimes( Y_2)_{h}^{**})\ar[ul,"\psi^g\star\psi^h"]
&
\end{tikzcd}
\end{equation}
Now, the diagram can be rewritten, by using the definition of $\phi^g\star\phi^h$ and $\psi^g\star\psi^h$ from \eqref{phiMphiN}, as follows
\begin{equation}
\begin{tikzcd}[column sep=0.1em,row sep=1.5em]
&(X_1\otimes X_2\otimes Y_2^*\otimes Y_1^*)^*
\ar[dl,"\psi_g",swap]\ar[dr,"\phi_g"]&
\\
Y_1^{**}\otimes (X_1\otimes X_2\otimes Y_2^*)^*
\ar[rr,"\Omega_g\otimes \id"]\ar[d,swap,"\id\otimes{}^*\psi_h"] && \bS_g(Y_1)\otimes (X_1\otimes X_2\otimes Y_2^*)^* \ar[d,"\id\otimes{}^*\phi_h"] 
\\
Y_1^{**}\otimes (X_1\otimes {}^{*}(Y_2^{**}\otimes X_2^*))^*\ar[rr,"\Omega_g\otimes \Omega_h"]\ar[d,"\cong",swap]
&& \bS_g(Y_1)\otimes (X_1\otimes {}^{*}(\bS_h(Y_2)\otimes X_2^*))^*\ar[d,"\cong"]
\\
Y_1^{**}\otimes Y_2^{**}\otimes (X_1\otimes X_2)^*\ar[rr,"\Omega_g\otimes \Omega_h",swap]&& \bS_g(Y_1)\otimes \bS_h(Y_2)\otimes (X_1\otimes X_2)^*
\end{tikzcd}
\end{equation}
where the lower square commutes due to naturality and the middle square and top triangle commute owing to Lemma \ref{lem:Serre_homogeneous}, which proves that $\Omega_\D$ is monoidal. Lastly, that $\Omega$ is a modification, i.e. that it is natural on $1$-morphisms follows from the commutativity of the diagram \eqref{eq:Serre_comp_graded_functor}.
\end{proof}


Now, by considering the $2$-groupoids of fixed points, we arrive to a pivotal version of the classification of extensions \eqref{eq:ext_classification} in terms of the pivotal Brauer-Picard $2$-categorical group.
\begin{theorem}\label{thm:piv_ext}
Let $\C$ be a pivotal finite tensor category and $G$ a finite group. The equivalence of $2$-groupoids \eqref{eq:ext_classification} lifts to an equivalence of $2$-groupoids 
\begin{equation}\label{eq:piv_ext}
    \PivExt(G,\C) \xrightarrow{\;\simeq\;} \MonFun\left(\uuG,\PivBrPic(\C)\right).
\end{equation}
\end{theorem}
\begin{proof}
Since, the action on the $2$-groupoid of monoidal $2$-functors is defined by post-composition with the action on $\BrPic(\C)$, we formally have that $\MonFun(\uuG, \BrPic(\C))^{B\underline{\zZ}} \simeq \MonFun(\uuG, \BrPic(\C)^{B\underline{\zZ}} )$. The assertion follows from Proposition \ref{prop: BZ fixed points}, \ref{prop:PivBrPic-is-BZ-equivariant} and \ref{prop:phi-is-equivariant}.
\end{proof}


\section{Spherical extensions and the spherical Brauer-Picard \texorpdfstring{$2$}{2}-groupoid}\label{sec:spherical-extensions}
This section is organized as follows. In Section~\ref{subsec:spherical-G-graded}, we introduce spherical $G$-graded extensions and establish the basic framework for their study. In Section~\ref{subsec:spherical-brauer-picard}, we define the spherical Brauer-Picard $2$-categorical group and explore its structure. In Section~\ref{subsec:sphBrPic-homotopy-fixed}, we realize the spherical Brauer-Picard $2$-categorical group as the $2$-groupoid of fixed points for a natural $B\underline{\Ztwo}$-action. Finally, in Section~\ref{subsec:classification-spherical}, we provide a classification of spherical extensions in terms of monoidal functors into the spherical Brauer-Picard $2$-categorical group.


\subsection{Spherical $G$-graded extensions}\label{subsec:spherical-G-graded}
We first recall some notions relevant for the description of sphericality of a pivotal tensor category in the non-semisimple setting.
Any finite $\kk$-linear category $\M$ is endowed with the structure of a $\Vect_{\kk}$-module category defined  for $V\in\Vect_{\kk}$ and $M\in\M$ by means of the isomorphism:
\[ \Hom_{\kk}\big(V,\Hom_{\M}(M,M')\big) \cong \Hom_{\M}(V\otimes_{\kk}M,M').\]
The (right exact) \textit{Nakayama functor} $\dN_{\M}$ of $\M$ is the endofunctor defined by the coend \cite[Def.\ 3.14]{fuchs2020eilenberg}
\[ \dN_{\M}(M) = \int^{M'\in\M} \Hom_{\M}(M,M')^* \otimes_{\kk} M' \;.\]
It comes equipped with a natural isomorphism
\begin{equation}\label{eq:Nak-F}
	\dN_{\N} \circ F \xRightarrow{~\sim~} F^{\rra} \circ \dN_{\M}.
\end{equation} 
for any right exact $\kk$-linear functor $F\colon\M\longrightarrow\N$, whose right adjoint is also right exact. According to \cite[Lemma 4.10]{fuchs2020eilenberg}, the Nakayama functor of a finite tensor category $\D$ can be described using \eqref{eq:Nak-F} as
\begin{equation}\label{eq:Nak_dual}
    \dN_{\D}\cong D_{\D}^{-1}\otimes (-)^{**}\cong {}^{**}(-) \otimes D_{\D}^{-1}
\end{equation}
where $D_{\D}\coloneqq \dN_{\D}(\unit)^*$ is called the \textit{distinguished invertible object} of $\D$. For every finite tensor category $\D$, we obtain from \eqref{eq:Nak_dual} a monoidal natural isomorphism
\begin{equation}\label{eq:Radford-C}
	\R^{\D}\colon D_{\D} \otimes-\otimes D_{\D}^{-1}\xRightarrow{~\sim~}  (-)^{****}
\end{equation}
called the \textit{Radford isomorphism}. These definitions of distinguished invertible object and Radford isomorphism coincide with the original definitions in \cite{etingof2004analogue} as it was shown in \cite[Appendix A]{shimizu2023ribbon}.

A finite tensor category $\D$ is called \textit{unimodular} \cite{etingof2004analogue}
if $D_{\D}$ is isomorphic to the monoidal unit $\unit$. In that case, any isomorphism $\textbf{u}_\D\colon D_{\D}\xrightarrow{\;\sim\;} \unit$ of the distinguished invertible object provides an identification $ D_{\D} \otimes-\otimes D_{\D}^{-1}\cong \id_\D$. This does not depend on the choice of $\textbf{u}_\D$ since $\unit$ is simple and $\Hom_\D(\unit,D_\D)$ is one-dimensional. Hence, the Radford isomorphism \eqref{eq:Radford-C} turns into a monoidal trivialization
\begin{align}\label{eq: Radford for unimodular}
	\mathcal R^{\D} \colon \id_{\D} \xRightarrow{~\sim~}  (-)^{****}\,,
\end{align}
in the case that $\D$ is unimodular.
\begin{definition}{\cite[Def.\ 3.5.2.]{douglas2018dualizable}} \label{def:spherical}
A unimodular pivotal tensor category $\D$ is called \emph{spherical}
\footnote{When $\C$ is a fusion category, this definition of sphericality is equivalent to more well known definition in terms of trace, due to \cite{barrett1999spherical} (see \cite[Theorem 7.3]{etingof2004analogue} or \cite[Prop.~3.5.4]{douglas2018dualizable} for a proof). But these notions differ in the nonsemisimple setting and neither one implies the other.} if
\begin{equation}\label{eq:R=p^2}
	\begin{tikzcd}[column sep=normal, row sep=scriptsize]
			\id_{\D} \ar[rr,"\mathcal R^{\D}", Rightarrow]\ar[dr,swap,"\fp", Rightarrow]&& (-)^{****}  \\
			&(-)^{**}  \ar[ru,swap,"\fp^{**}", Rightarrow]&  
	\end{tikzcd}
\end{equation}
commutes, where $\fp\colon \id_{\D} \xRightarrow{~\sim~} (-)^{**}$ is the pivotal structure of $\D$.
\end{definition}
To discuss sphericality, then we need that graded extensions behave well together with unimodularity as shown in the following lemma.
\begin{lemma}\label{lemma: unimodularity of extension}
Let $\D$ be a $G$-extension of a finite tensor category $\C$. Then $\C$ is unimodular if and only if $\D$ is unimodular.
\end{lemma}
\begin{proof}
Any $G$-graded extension $\D$ of $\C$ can be seen as a $\C$-module category. According to Lemma \ref{lem:Serre_homogeneous} the relative Serre functor of ${}_\C\D$ is given the double-dual functor. Thus, by \cite[Thm.\ 4.26]{fuchs2020eilenberg}, we have that
\begin{equation}
    D_\D^{-1}=\dN_\D(\unit)\cong D_\C^{-1}\tl \unit^{**}\cong D_\C^{-1}\,,
\end{equation}
which gives the desired result. This Lemma also follows from \cite[Thm.\ 6.1]{etingof2004analogue}.
\end{proof}
\begin{definition} Let $\C$ be a spherical (unimodular) tensor category.
\begin{enumerate}[(i)]
\item A \emph{spherical G-graded extension} of $\C$ is a pivotal $G$-extension $(\D,\eqtriv^{\D}, \tilde \fp)$ of $\C$  such that $(\D, \tilde \fp)$ is spherical.
\item The $2$-groupoid $\SphExt(G,\C)$ of spherical extensions of $\C$ is defined as the full sub $2$-groupoid of $\PivExt(G,\C)$ with objects being spherical $G$-graded extensions of $\C$.
\end{enumerate}
\end{definition}

For the rest of this section, we fix a spherical (unimodular) finite tensor category $\C$. In the same vein of Section \ref{subsec:PivG-ext}, the purpose of this section is to realize the $2$-groupoid $\SphExt(G,\C)$ of spherical extensions of $\C$ as the $2$-groupoid of fixed points for a natural $B\underline{\Ztwo}$-action on $\Ext(G,\C)$. In Proposition \ref{prop:BZ-action-Ext}, we proved that the double-dual functors of graded extensions, form a pseudo-natural autoequivalence of the identity $2$-functor $\id_{\Ext(G,\C)}$. It follows that the fourth power of the dual functors of graded extensions assemble into a pseudo-natural equivalence
\begin{align*}
    (-)^{****}\colon \id_{\Ext(G,\C)}\xRightarrow{~\sim~} \id_{\Ext(G,\C)}\,,
\end{align*}
as well.
According to Example \ref{ex: BZ and BZ2 actions}, to define a $B\underline{\Ztwo}$-action, it remains to define an invertible modification between $\id_{\id_\Ext(G,\C)}$ and $(-)^{****}$. The Radford isomorphisms of the $G$-graded extensions of $\C$ form such an invertible modification, as we show in the next proposition.

\begin{proposition}\label{prop: Rad-inv-modification}
The Radford isomorphisms
\eqref{eq: Radford for unimodular} assemble into an invertible modification 
\begin{equation}\label{eq:monPNT}
\R\colon  \id_{\id_\Ext(G,\C)}\xRightarrow{~\sim~}(-)^{****}\,.
\end{equation}
\end{proposition}

\begin{proof}
We first check naturality for $1$-morphisms: we need to prove that the diagram
	 \begin{equation*}
		\begin{tikzcd}[column sep=large]
			F\circ \id_{\D} \ar[r,"\id * \mathcal R^{\D}", Rightarrow]\ar[""{name=2}, d,swap,Rightarrow,"\id"]& F\circ (-)_{\D}^{****} \ar[""{name=1}, d,Rightarrow,"(-)_F^{****}"] \\
			\id_{\D'} \circ F \ar[""{name=2},r,Rightarrow,swap,"\mathcal R^{\D'} *\id"] 
			& (-)_{\D'}^{****} \circ F
		\end{tikzcd}
	\end{equation*}
commutes for every 1-cell $F:\D \xrightarrow{~\sim~} \D'$ in $\Ext(G, \C)$, i.e. that 
\begin{equation*}
	\begin{tikzcd}[column sep=large]
			F(X) \ar[r,"F(\mathcal R^{\D}_X)", rightarrow]
			\ar[""{name=2}, dr,swap,"\mathcal R^{\D'}_{F(X)}", rightarrow]& F(X^{****}) \ar[""{name=2},d,"\zeta^F_X\circ \zeta^F_X", rightarrow]  \\
			&F(X)^{****}  
	\end{tikzcd}
\end{equation*}
commutes for every $X\in \D$, where $\zeta^F_X$ is
the duality isomorphism for $F$ from \eqref{eq:iso_F_double-dual}. This is nothing else than the statement in \cite[Theorem 4.4]{shimizu2023ribbon}, once we consider that $\C$ is unimodular and that $F$ is an equivalence.

It remains to be shown that for $(\D, \eqtriv^{\D})\in \Ext(G, \C)$, the Radford isomorphism $\R^{\D}\colon \id_{\D} \Rightarrow (-)_{\D}^{****}$ is a $2$-cell in $\Ext(G,\C)$, i.e. that it obeys \eqref{eq:comp_nat_triv_component}. 
The following diagram 
\begin{equation}
    \begin{tikzcd}\label{eq:R-mod}
	{\iota^{\D}\circ(-)_{\D_e}^{****}} && {(-)_{\C}^{**} \circ \iota^{\D}} && {\id_{\C}\circ \iota^{\D}\circ (-)_{\D_e}^{**}} \\
	&&&& {\id_{\C}\circ (-)_{\C}^{**}\circ \iota^{\D}} \\
	{(-)_{\C}^{****} \circ \iota^{\D}} &&&& {\id_{\C}\circ \iota^{\D}}
	\arrow["{(-)_{\iota^{\D}}^{**}}", Rightarrow, from=1-1, to=1-3]
	\arrow["{(-)_{\iota^{\D}}^{**}\circ (-)_{\iota^{\D}}^{**}}"', Rightarrow, from=1-1, to=3-1]
	\arrow["{\fp^{-1}}", Rightarrow, from=1-3, to=1-5]
	\arrow["{(-)_{\iota^{\D}}^{**}}", Rightarrow, from=1-5, to=2-5]
	\arrow["{\fp^{-1}}", Rightarrow, from=2-5, to=3-5]
	\arrow["{\fp^{-1}}", Rightarrow, from=3-1, to=2-5]
	\arrow["{(\R^{\C})^{-1}}"', Rightarrow, from=3-1, to=3-5]
\end{tikzcd}
\end{equation}
commutes, since the bottom triangle holds by sphericality of $\C$, and the top half commutes trivially by level exchange. Now, the natural isomorphism $\tau_{(-)_{\D}^{****}}$ equals, by  definition, the composition of the top and right arrows in \eqref{eq:R-mod}. Hence, the condition \eqref{eq:comp_nat_triv_component} applied to $\R^{\D}$ translates into
    \begin{equation*}
		\begin{tikzcd}[column sep=large, row sep=scriptsize]
			\iota^{\D}(X) \ar[r,"\iota^{\D}(\mathcal R^{\D_e}_X)", rightarrow]
			\ar[""{name=2}, dr,swap,"(\R^{\C}_{\iota^{\D}})_X", rightarrow]& \iota^{\D}(X^{****}) \ar[""{name=2},d,"\zeta^{\iota^{\D}}_X\circ \zeta^{\iota^{\D}}_X", rightarrow]  \\
			&\iota^{\D}(X)^{****}  
		\end{tikzcd}
	\end{equation*}
for $X\in \D$, which also commutes due to \cite[Theorem 4.4]{shimizu2023ribbon}. 
\end{proof}

\begin{corollary}
Let $\C$ be a spherical (unimodular) finite tensor category. The data consisting of
    \begin{itemize}
        \item the invertible pseudo-natural transformation $(-)^{**}\colon\id_{\Ext(G,\C)}\xRightarrow{~\sim~} \id_{\Ext(G,\C)},$ and
        \item the invertible modification $\mathcal R\colon \id_{\id_\Ext(G,\C)}\xRightarrow{~\sim~} (-)^{****}$ defined in \eqref{eq:monPNT}, induced by the Radford isomorphisms,
    \end{itemize} 
define a $B\underline{\zZ/2\zZ}$ action on $\Ext(G,\C)$.
\end{corollary}
\begin{proof}
This follows from Example \ref{ex: BZ and BZ2 actions} and Proposition \ref{prop: Rad-inv-modification}.
\end{proof}

\begin{proposition}\label{prop:SphExt-Z2}
Let $\C$ be a spherical (unimodular) finite tensor category and $G$ a finite group.
The $2$-groupoid $\SphExt(G, \C)$ is equivalent to the $2$-groupoid of $B\underline{\zZ/2\zZ}$-fixed points $\Ext(G,\C)^{B\underline{\zZ/2\zZ}}$. 
\end{proposition}
\begin{proof}
According to Example \ref{ex: fixed point for BZ and BZ2 actions}, a fixed point for the $B\underline{\zZ/2\zZ}$ action is equivalent to a choice of an object $(\D,\eqtriv^{\D})\in \Ext(G,\C)$ and an invertible $2$-morphism $\fp_{\D}\colon \id_{\D} \Rightarrow (-)^{**}_{\D}$ such that $\fp_{\D}^{**}\circ \fp_{\D}=\R^{\D}$, i.e. a spherical structure on $\D$.
As in the proof of Proposition \ref{prop: BZ fixed points}, $\fp_{\D}$ fulfilling \eqref{eq:comp_nat_triv_component} means $\eqtriv^{\D}$ is pivotal. Hence, the data of an object in $\Ext(G,\C)^{B\underline{{\Ztwo}}}$ is the same as that of an object in $\SphExt(G,\C)$. The same argument from Proposition \ref{prop: BZ fixed points} shows that $1$- and $2$-cells in $\Ext(G,\C)^{B\underline{\zZ/2\zZ}}$ agree with those of $\SphExt(G,\C)$.
\end{proof}


\subsection{The spherical Brauer-Picard $2$-categorical group}
Given a $(\C,\D)$-bimodule category $\M$, the isomorphisms \eqref{eq:Nak-F} associated to the right and left actions on $\M$ endow the Nakayama functor $\dN_{\M}$ with the structure of a twisted $(\C,\D)$-bimodule functor \cite[Thm~4.5]{fuchs2020eilenberg}
\begin{equation}\label{eq:Nakayama_bimodule_functor}
     {}^{**}X\tl \dN_{\M}(M)\tr Y^{**} \cong \dN_{\M}(X\tl M\tr Y)\,. 
\end{equation}
The relative Serre functors and Nakayama functor of $\M$ are related by natural isomorphisms 
\begin{equation}\label{eq:Serre-Nak-relation}
	D_{\C}\tl \dN_{\M} \cong \dS_{\M}^{\C} \qquad  \text{ and } \qquad D_{\D}\tl \dN_{\M} \cong \dS_{\M}^{\overline{\D}}
\end{equation} 
of twisted bimodule functors \cite[Thm~4.26]{fuchs2020eilenberg}, where the twisted $\C$-module structure (resp. $\D$) of $D_{\C}\tl \dN_{\M}$ involves the Radford isomorphism of $\C$ (resp. $\D$).
The isomorphisms \eqref{eq:Serre-Nak-relation} lead to an analogue of the Radford isomorphism for a exact bimodule categories \cite[Thm.\ 4.14]{spherical2022}. In particular, for an invertible $\C$-bimodule category $\M$, there is an isomorphism
\begin{equation}\label{eq:Radford-M}
	\R^{\M}\colon D_{\C}\tl -\tr D_{\C}^{-1} \xRightarrow{~\sim\,} \dS_{\M}^{\C}\circ \dS_{\M}^{\C}
\end{equation}
of twisted $\C$-bimodule functors called the \textit{bimodule Radford isomorphism of $\M$} \cite[Cor.\ 4.16]{spherical2022}.
\label{subsec:spherical-brauer-picard}
In the case that $\C$ is a spherical (unimodular) finite tensor category, any trivialization $\textbf{u}_\C\colon D_{\C}\xrightarrow{\;\sim\;} \unit$ of the distinguished invertible object turns the Radford isomorphism \eqref{eq:Radford-M} of $\M$ into a natural isomorphism 
\begin{align}\label{eq:Rad for modules over sph}
    \R^{\M}\colon \id_{\M}\xRightarrow{~\sim\,} \dS_{\M}^{\C}\circ \dS_{\M}^{\C}
\end{align}
of $\C$-bimodule functors. Once again, this does not depend on the choice of 
$\textbf{u}_\C$ since $\unit$ is simple.

\begin{definition}\label{def:spherical_module}\cite[Def.\ 5.20]{spherical2022}
Let $\C$ be a spherical (unimodular) finite tensor category. An invertible pivotal $\C$-bimodule category $\M$ is called \textit{spherical} if the diagram
\begin{equation}\label{eq:R_M=p_M^2}
    \begin{tikzcd}[column sep=large]
        \id_{\M} \ar[rr,"\R^{\M}", Rightarrow]\ar[dr,swap,"\tfp", Rightarrow]&& \dS_{\M}^{\C}\circ \dS_{\M}^{\C}  \\
        &\dS_{\M}^{\C}  \ar[ru,swap,"\tfp", Rightarrow]  &
    \end{tikzcd}
\end{equation}
commutes, where $\tfp\colon \id_\M\xRightarrow{~\sim\,}\dS_\M^\C$ is the pivotal structure of $\M$.
\end{definition}

\begin{remark}
\label{rem: 2 spherical structures}
Pivotal module categories do not need to be spherical.
Let $\C$ be spherical and $\M$ an indecomposable exact left $\C$-module category admitting a pivotal structure $\tfp$. Then any other pivotal structure is a scalar multiple of $\tfp$. 
From \eqref{eq:R_M=p_M^2}, $(\R^{\M})^{-1} (\tfp_{\M}\cdot\id)\tfp_{\M}$ is an automorphism of $\id_{\M}$. As $\M$ is indecomposable, this must be a scalar multiple (say $c\in\kk^\times$) of identity. Hence, $c\,(\tfp_{\M}\circ\id)\tfp_{\M} = \R^{\M}$. Consequently, $\pm \sqrt{c}\,\tfp_{\M}$ will be a spherical structure on $\M$. However, any other choice of scalar multiple of $\tfp_{\M}$ will lead to a pivotal module category that is not spherical.
\end{remark}

\begin{definition}\label{def:SphBrPic}
Let $\C$ be a spherical (unimodular) finite tensor category. The \textit{spherical Brauer-Picard $2$-groupoid of $\C$} is the full sub $2$-groupoid of $\PivBrPic(\C)$ whose objects are spherical invertible $\C$-bimodule categories, and denoted by $\SphBrPic(\C)$.
\end{definition}
To show that $\SphBrPic(\C)$ inherits the monoidal structure of $\PivBrPic(\C)$, we need to prove the compatibility between the bimodule Radford isomorphisms and the relative Deligne product.
\begin{proposition}[Monoidality of bimodule Radford isomorphisms]\label{prop:radford is monoidal}
Let $\C$ be a unimodular finite tensor category. Given $\M,\N\in \BrPic(\C)$, then
\begin{equation*}
\begin{tikzcd}
    &&{\dS_{\M\boxtimes_{\C}\mathcal N}^{\C}\circ \dS_{\M\boxtimes_{\C}\mathcal N}^{\C}} \ar[dd,"{\mu_{\M, \mathcal N}\circ \mu_{\M, \mathcal N}}",Rightarrow]\\
    \id_{\M\boxtimes_{\C}\N}
    \ar[rru,"{\mathcal R^{\M\boxtimes_{\C}\N}}",Rightarrow]
    \ar[rrd,"{\mathcal R^{\M}\boxtimes_{\C} \mathcal R^{\mathcal N}}",swap,Rightarrow]
    &&\\
    &&\left(\dS_{\M}^{\C} \circ \dS_{\M}^{\C}\right) \boxtimes_{\C}\left(\dS_{\mathcal N}^{\C}\circ \dS_{\mathcal N}^{\C}\right)=\left(\dS_{\M}^{\C}\boxtimes_{\C}\dS_{\mathcal N}^{\C}\right) \circ \left( \dS_{\M}^{\C}\boxtimes_{\C}\dS_{\mathcal N}^{\C}\right)
\end{tikzcd}
\end{equation*}
commutes, where $\mathcal{R}$ denotes the bimodule Radford isomorphism \eqref{eq:Rad for modules over sph}.
\end{proposition}
\begin{proof}
We define an auxiliary finite multitensor category $\mathbb{B}$, as follows
\begin{equation*}
\mathbb{B}=
\begin{bmatrix}
\C & \M & \M\boxtimes_{\C}\N\\
\overline{\M} & \C & \N\\
\overline{\N}\boxtimes_{\C}\overline{\M} & \overline{\N} & \C
\end{bmatrix}	
\end{equation*}
where the entries of the matrix are the components of $\mathbb{B}$ and whose tensor product is defined by matrix multiplication. Rigidity follows from \cite[Thm.\ 4.2]{spherical2022}. Moreover, the distinguished invertible object of $\mathbb{B}$ is given by the matrix with only non-zero entries being $D_\C$ in the diagonal. Thus, a trivialization of $D_\C$ gives a trivialization of $D_\mathbb{B}$ and hence $\mathbb{B}$ is unimodular. Now, the Nakayama functor of $\mathbb{B}$ decomposes as
\begin{equation}
    \dN_{\mathbb{B}}\cong \bigoplus_{i,j=1}^3\dN_{\mathbb{B}_{i,j}}\,,
\end{equation}
and a similar argument to the one used in Lemma \ref{lem:Serre_homogeneous} shows that the double-duals of $\mathbb{B}$ obey for any object $A\in \mathbb{B}_{i,j}$ that $A^{**}\cong \dS_{\mathbb{B}_{i,j}}(A)$. It follows that the Radford isomorphism of $\mathbb{B}$ is related to those of its components:
\begin{equation}\label{eq:Radford_decomp}
\begin{tikzcd}
    {}^{**}(-)\cong{}^{**}(-)\otimes D_\C^{-1}\arrow[""{name=0, anchor=center, inner sep=0}, "{{\mathcal{R}^\mathbb{B}}}", curve={height=-30pt}, from=1-1, to=1-3,Rightarrow]\ar[d,equal]\ar[r,"\eqref{eq:Serre-Nak-relation}",Rightarrow]
    &\dN_{\mathbb{B}}\ar[d,equal]\ar[r,"\eqref{eq:Serre-Nak-relation}",Rightarrow]&
    D_\C\otimes (-)^{**}\cong (-)^{**}\ar[d,equal]\\
    \bigoplus_{i,j=1}^3 \overline{\dS}_{\mathbb{B}_{i,j}}\cong\bigoplus_{i,j=1}^3 \overline{\dS}_{\mathbb{B}_{i,j}}\otimes D_\C^{-1} \arrow[""{name=1, anchor=center, inner sep=0}, "{{\oplus_{i,j=1}^3\mathcal{R}^{\mathbb{B}_{i,j}}}}", curve={height=30pt}, from=2-1, to=2-3,Rightarrow,swap]\ar[r,"\eqref{eq:Serre-Nak-relation}",swap,Rightarrow]   &\bigoplus_{i,j=1}^3\dN_{\mathbb{B}_{i,j}}\ar[r,"\eqref{eq:Serre-Nak-relation}",swap,Rightarrow]&
    \bigoplus_{i,j=1}^3 D_\C\otimes \dS_{\mathbb{B}_{i,j}}
    \cong \bigoplus_{i,j=1}^3 \dS_{\mathbb{B}_{i,j}}
\end{tikzcd}
\end{equation}
Now, since $\mathcal{R}^\mathbb{B}$ is monoidal, in particular, we obtain for $M\in\M$ and $N\in\N$ that
\begin{equation*}
\begin{tikzcd}[column sep=normal, row sep=small]
    && (M\otimes^\mathbb{B}N)^{****} \ar[dd,"{\cong}",Rightarrow]\\
    M\otimes^\mathbb{B}N
    \ar[rru,"{\mathcal R^{\mathbb{B}}_{M\otimes^\mathbb{B}N}}",Rightarrow]
    \ar[rrd,"\mathcal R^{\mathbb{B}}_M\,\otimes^{\mathbb{B}}\, \mathcal R^{\mathbb{B}}_N",swap,Rightarrow]
    &&\\
    &&M^{****}\otimes^\mathbb{B}N^{****}
\end{tikzcd}
\end{equation*}
which translates into the desired result once we consider \eqref{eq:Radford_decomp}.
\end{proof}

\begin{lemma}\label{lem:spherical-Deligne}
Let $\C$ be spherical and $\M,\N$ be spherical $\C$-bimodule categories. Then the pivotal structure \eqref{eq:pivotal_Deligne_pr} on $\M\btC\N$ is spherical.
\end{lemma}
\begin{proof}
Sphericality of the pivotal structure on $\M\btC\N$ follows from the commutativity of the diagram below:
{\small
\[\begin{tikzcd}[column sep=normal, row sep=normal]
	{\id_{\M}\boxtimes_{\C}\id_{\mathcal N} = \id_{\M\boxtimes_{\C}\N}} && {\dS_{\M\boxtimes_{\C}\mathcal N}^{\C}\circ \dS_{\M\boxtimes_{\C}\mathcal N}^{\C}} \\
	& {\left(\dS_{\M}^{\C} \circ \dS_{\M}^{\C}\right) \boxtimes_{\C}\left(\dS_{\mathcal N}^{\C}\circ \dS_{\mathcal N}^{\C}\right)=\left(\dS_{\M}^{\C}\boxtimes_{\C}\dS_{\mathcal N}^{\C}\right) \circ \left( \dS_{\M}^{\C}\boxtimes_{\C}\dS_{\mathcal N}^{\C}\right)} \\
	\\
	{\dS_{\M}^{\C}\boxtimes_{\C}\dS_{\mathcal N}^{\C}} & {\dS_{\M\boxtimes_{\C}\mathcal N}^{\C}} & {\dS_{\M\boxtimes_{\C}\mathcal N}^{\C} \circ\left( \dS_{\M}^{\C}\boxtimes_{\C}\dS_{\mathcal N}^{\C}\right)}
	\arrow["{{\mathcal R^{\M\boxtimes_{\C}\N}}}", Rightarrow, from=1-1, to=1-3]
	\arrow["{{\mathcal R^{\M}\boxtimes_{\C} \mathcal R^{\mathcal N}}}"{pos=0.7}, Rightarrow, from=1-1, to=2-2]
	\arrow["{{\tfp_{\M}\boxtimes_\C\tfp_{\mathcal N}}}"', Rightarrow, from=1-1, to=4-1]
	\arrow["{{\mu_{\M, \mathcal N}\circ \mu_{\M, \mathcal N}}}", Rightarrow, from=2-2, to=1-3]
	\arrow["{{(\id \circ \tfp_{\M})\boxtimes_{\C} (\id \circ \tfp_{\mathcal N})= \id \circ (\tfp_{\M}\boxtimes_{\C}\tfp_{\mathcal N})}}"{description}, Rightarrow, from=4-1, to=2-2]
	\arrow["{{\mu_{\M, \mathcal N}}}"', Rightarrow, from=4-1, to=4-2]
	\arrow["{{\id \circ (\tfp_{\M}\boxtimes_\C\tfp_{\mathcal N})}}"', Rightarrow, from=4-2, to=4-3]
	\arrow["{{\id \circ \mu_{\M, \mathcal N}}}"', Rightarrow, from=4-3, to=1-3]
\end{tikzcd}\]}
Here, the top triangle commutes by Proposition \ref{prop:radford is monoidal}, the left triangle by sphericality of the pivotal structures on $\M$ and $\N$, and the remaining rectangle by level exchange. 
\end{proof}
 
\begin{proposition}
The monoidal structure on $\PivBrPic(\C)$ induces a monoidal structure on $\SphBrPic(\C)$.
\end{proposition}
\begin{proof}
By Lemma \ref{lem:spherical-Deligne}, we have that $\SphBrPic(\C)$ is closed under the relative Deligne product. Since $\SphBrPic(\C)$ is a full sub $2$-groupoid of $\PivBrPic(\C)$ the result follows.
\end{proof}


\subsection{Realization of $\SphBrPic(\C)$ as fixed points}\label{subsec:sphBrPic-homotopy-fixed}
Let $\C$ be a spherical (unimodular) finite tensor category. In this section, we define a $2$-categorical $B\underline{\zZ/2\zZ}$-action on $\BrPic(\C)$. 
According to Proposition \ref{prop:RSpNE-is-Monoidal}, the relative Serre functors assemble into a pseudo-natural autoequivalence of the identity $2$-functor $\id_{\BrPic(\C)}$. Hence, by Example \ref{ex: BZ and BZ2 actions}, it is enough to define an invertible monoidal modification between $\id$ and $\bS^2$. The Radford isomorphisms of invertible module categories give rise to such modification as we show next.
\begin{proposition}\label{radford is inv-modification}
The bimodule Radford isomorphisms \eqref{eq:Rad for modules over sph} form an invertible monoidal modification 
\begin{equation}\label{eq:module-Radford-modif}
\R\colon  \id_{\id_{\BrPic(\C)}}\xRightarrow{~\sim~}\bS^2\,.
\end{equation}

\end{proposition}
\begin{proof}
To show that $\R$ is compatible with $1$-morphisms, consider a $\C$-module equivalence $H\colon \M\longrightarrow\N$. The required condition 
\begin{equation}\label{eq: radford:mod}
    \begin{tikzcd}[column sep=3.5em, row sep=2em]
        H\circ\id_\M \ar[Rightarrow,r,"\id\circ \R^\M"] \ar[Rightarrow,d,"\id",swap]&H\circ \bS_\M\circ \bS_\M\ar[Rightarrow,d,"\bS_H\circ \bS_H"]\\
        \id_\N\circ H \ar[Rightarrow,r,"\R^\N\circ\id",swap] &\bS_\N\circ \bS_\N \circ H 
    \end{tikzcd}
\end{equation}
is equivalent to the commutativity of
\begin{equation}
\begin{tikzcd}[column sep=4em, row sep=normal]
    H\circ\overline{\bS}_\M
    \ar[Rightarrow,r,"\eqref{eq:Serre-Nak-relation}",swap]
    \ar[Rightarrow,d,"\overline{\bS}_H",swap]&H\circ \dN_\M\ar[d,"\eqref{eq:Nak-F}",Rightarrow]&H\circ  \bS_\M\ar[Rightarrow,d,"\bS_H"]\ar[Rightarrow,l,"\eqref{eq:Serre-Nak-relation}"]\\
\overline{\bS}_\M\circ H    \ar[Rightarrow,r,"\eqref{eq:Serre-Nak-relation}"]
&\dN_\N\circ H&\bS_\N\circ H \ar[Rightarrow,l,"\eqref{eq:Serre-Nak-relation}",swap]
\arrow[""{name=0, anchor=center, inner sep=0}, "\id\circ \R^\M", curve={height=-20pt}, Rightarrow, from=1-1, to=1-3]
\arrow[""{name=0, anchor=center, inner sep=0}, "\R^\N\circ\id"', curve={height=20pt}, Rightarrow, from=2-1, to=2-3]
    \end{tikzcd}
\end{equation}
where the triangles commute from the definition of the bimodule Radford isomorphism, and the rectangles commute since \eqref{eq:Serre-Nak-relation} is an isomorphism of twisted bimodule functors. Following \cite[Definition 2.15]{davydov2021braided} that the modification  $\mathcal{R}$ is monoidal corresponds to the condition
\[\begin{tikzcd}
	{\M\btC \N } && {\M\btC \N } &&&& {\M\btC \N } &&& {\M\btC \N } \\
	&&&& {=} \\
	{\M\btC \N } && {\M\btC \N } &&&& {\M\btC \N } &&& {\M\btC \N }
	\arrow["\id", from=1-1, to=1-3]
	\arrow["{\id_{\M}\boxtimes\id_{\N}}"', from=1-1, to=3-1]
	\arrow[""{name=0, anchor=center, inner sep=0}, "{\dS_{\M\btC\N}^{\C}\circ \dS_{\M\btC\N}^{\C}}"{pos=0.3}, shift left=5, curve={height=-18pt}, from=1-3, to=3-3]
	\arrow[""{name=1, anchor=center, inner sep=0}, "{\id_{\M\btC\N}}"', shift right=5, curve={height=18pt}, from=1-3, to=3-3]
	\arrow[""{name=2, anchor=center, inner sep=0}, "\id", from=1-7, to=1-10]
	\arrow[""{name=3, anchor=center, inner sep=0}, "{\id_{\M}\boxtimes\id_{\N}}"'{pos=0.8}, shift right=5, curve={height=24pt}, from=1-7, to=3-7]
	\arrow[""{name=4, anchor=center, inner sep=0}, "{(\dS_{\M}^{\C} \circ \dS_{\M}^{\C})\btC (\dS_{\N}^{\C} \circ \dS_{\N}^{\C})}"{description, pos=0.8}, shift left=5, curve={height=-24pt}, from=1-7, to=3-7]
	\arrow["{\dS_{\M\btC\N}^{\C}}", from=1-10, to=3-10]
	\arrow["\id", from=3-1, to=3-3]
	\arrow[""{name=5, anchor=center, inner sep=0}, "\id"', from=3-7, to=3-10]
	\arrow["{\R^{\M\btC\N}}", shorten <=11pt, shorten >=11pt, Rightarrow, from=1, to=0]
	\arrow["{\R^{\M}\btC\R^{\N}}", shorten <=14pt, shorten >=14pt, Rightarrow, from=3, to=4]
	\arrow["{\mu_{\M,\N}\circ \mu_{\M,\N}}"'{pos=0.2}, shift left=5, curve={height=-30pt}, shorten <=11pt, shorten >=11pt, Rightarrow, from=2, to=5]
\end{tikzcd}\]
which holds as shown in Proposition \ref{prop:radford is monoidal}.
\end{proof}

\begin{corollary}\label{BZ2_action_BrPicC}
Let $\C$ be a spherical (unimodular) finite tensor category. The data consisting of
	\begin{itemize}
		\item the monoidal pseudo-natural equivalence $\bS\colon \id_{\BrPic(\C)}\xRightarrow{~\sim~} \id_{\BrPic(\C)}$, and
		\item the invertible monoidal modification $\R\colon  \id_{\id_{\BrPic(\C)}}\xRightarrow{~\sim~}\bS^2$,
	\end{itemize}
define a monoidal $B\underline{\zZ/2\zZ}$-action on $\BrPic(\C)$.
\end{corollary}
\begin{proof}
This follows from Example \ref{ex: BZ and BZ2 actions}, Propositions \ref{prop:RSpNE-is-Monoidal}, \ref{prop:radford is monoidal} and \ref{radford is inv-modification}. 
\end{proof}

We recover the spherical Brauer-Picard $2$-categorical group as fixed points for this action. 

\begin{proposition}\label{prop:SphBrPic-Z2}
Let $\C$ be a spherical (unimodular) finite tensor category. The $2$-categorical group $\SphBrPic(\C)$ is monoidally $2$-equivalent to the $2$-categorical group of $B\underline{\zZ/2\zZ}$-fixed points of $\BrPic(\C)$.
\end{proposition}
\begin{proof}
Following Example~\ref{ex: fixed point for BZ and BZ2 actions}, an object of $\BrPic(\C)^{B\underline{\zZ/2\zZ}}$ is a pair $(\M,\tfp)$ where $\M\in\BrPic(\C)$ and $\tfp\colon\id_{\M} \Rightarrow \dS_{\M}^{\C}$ is an invertible $2$-morphism in $\BrPic(\C)$ such that $\bS(\tfp)\circ\tfp=\mathcal R_{\M}$, that is $(\M,\tfp)$ is an invertible spherical $\C$-bimodule category. The same argument in the proof of Proposition \ref{prop:PivBrPic-is-BZ-equivariant} shows that $1$-cells and $2$-cells agree. Lastly, the monoidal structure inherited by $\BrPic(\C)^{B\underline{\zZ/2\zZ}}$ is given by the relative Deligne product, the same as in $\SphBrPic(\C)$.
\end{proof}


\subsection{Classification of spherical extensions}\label{subsec:classification-spherical}

In this subsection, we prove a spherical version of the classification of extensions \eqref{eq:ext_classification} in terms of the spherical Brauer-Picard $2$-categorical group.
\begin{proposition}\label{prop:phi-is-Bz_2-equivariant}
The equivalence $\mathrm{E}\colon\Ext(G,\C) \xrightarrow{\;\simeq\;} \MonFun\left(\uuG,\BrPic(\C)\right)$ is $B\underline{\zZ/2\zZ}$-equivariant.
\end{proposition}
\begin{proof}
The statement is proven following the proof in Proposition \ref{prop:phi-is-equivariant} mutatis mutandis. The key detail to verify is that the modification $\Omega$ commutes with the corresponding Radford modifications defining the $B\underline{\zZ/2\zZ}$-actions. This amounts to check that, given a $G$-extension $\D$, the bimodule Radford isomorphism of a graded component $\D_g$ is given, up to $\Omega_g$, by the Radford isomorphism of $\D$ restricted to $\D_g$, explicitly that 
\begin{equation}\label{first diag}
    \begin{tikzcd}[column sep=3ex, row sep=3ex]
    {\id_{\D_g} } && {\dS_{\D_g}^{\D_e} \circ \dS_{\D_g}^{\D_e}} \\
    & {(-)^{****}_{\D_g}} &
    \arrow["{{\mathcal{R}^{\D_g}}}", Rightarrow, from=1-1, to=1-3]
    \arrow["{{\mathcal{R}^{\D}|_{\D_g}}}", Rightarrow, from=1-1, to=2-2,swap]
    \arrow["{{\Omega_g \circ \Omega_g }}", Rightarrow, from=1-3, to=2-2]
    \end{tikzcd}
\end{equation}
commutes. To verify this, we start by considering the diagram 
\begin{equation}\label{prelim diag}
\begin{tikzcd}
\overline{\dS}_{\D_g}^{\D_e} && {\dN_{\D_g}} && {\dS_{\D_g}^{\D_e}} \\
{{}^{**}(-)_{\D_g}} && {\dN_{\D}|_{\D_g}} && {(-)^{**}_{\D_g}}
\arrow[""{name=0, anchor=center, inner sep=0}, "\overline{\dS}_{\D_g}^{\D_e}(\mathcal{R}^{\D_g})", curve={height=-30pt}, Rightarrow, from=1-1, to=1-5]
\arrow["\overline{\dS}_{\D_g}^{\D_e}(\Omega_g)"', Rightarrow, from=1-1, to=2-1]
\arrow[""{name=1, anchor=center, inner sep=0}, "\eqref{eq:Serre-Nak-relation}"', Rightarrow, from=1-3, to=1-1]
\arrow[""{name=2, anchor=center, inner sep=0}, "\eqref{eq:Serre-Nak-relation}", Rightarrow, from=1-3, to=1-5]
\arrow["\eqref{eq:Nak-F}", Rightarrow, from=1-3, to=2-3]
\arrow["\Omega_g", Rightarrow, from=1-5, to=2-5]
\arrow[""{name=3, anchor=center, inner sep=0}, "{{{}^{**}(\mathcal R^{\D}|_{\D_g})}}"', curve={height=30pt}, Rightarrow, from=2-1, to=2-5]
\arrow[""{name=4, anchor=center, inner sep=0}, "\eqref{eq:Serre-Nak-relation}"', Rightarrow, from=2-3, to=2-1]
\arrow[""{name=5, anchor=center, inner sep=0}, "\eqref{eq:Serre-Nak-relation}", Rightarrow, from=2-3, to=2-5]
\end{tikzcd}
\end{equation}
where the natural isomorphism $\dN_{\D_g}\cong \dN_{\D}|_{\D_g}$ comes from applying \eqref{eq:Nak-F} to the inclusion functor $\D_g\to \D$. The top and bottom triangles commute by definition, and the middle squares since $\Omega$ is the unique natural isomorphism realizing double-dual functors as relative Serre functors. Applying $\overline{\dS}_{\D_g}^{\D_e}$ to \eqref{first diag}, we obtain the equivalent diagram
\[\begin{tikzcd}[row sep=small,column sep=small]
	\overline{\dS}_{\D_g}^{\D_e} &&  && {\dS_{\D_g}^{\D_e}} \\
	& {{}^{**}(-)_{\D_g}} && {(-)_{\D_g}^{**}} \\
	\\
	\\
	&& \overline{\dS}_{\D_g}^{\D_e} \circ (-)^{****}_{\D_g}
	\arrow[""{name=0, anchor=center, inner sep=0}, "\overline{\dS}_{\D_g}^{\D_e}(\mathcal{R}^{\D_g})", Rightarrow, from=1-1, to=1-5]
	\arrow["\overline{\dS}_{\D_g}^{\D_e}(\Omega_g)"{pos=1}, Rightarrow, from=1-1, to=2-2]
	\arrow[""{name=0, anchor=center, inner sep=0}, "\overline{\dS}_{\D_g}^{\D_e}({\mathcal R^{\D}|_{\D_g}})"', shift right=5, curve={height=6pt}, shorten >=10pt, Rightarrow, from=1-1, to=5-3]
	\arrow["\Omega_g"', Rightarrow, from=1-5, to=2-4]
	\arrow[""{name=1, anchor=center, inner sep=0}, "{{}^{**}(\mathcal R^{\D}|_{\D_g})}",  Rightarrow, from=2-2, to=2-4]
	\arrow["\overline{\dS}_{\D_g}^{\D_e}(\Omega_g)\circ \Omega_g", shift left=6, Rightarrow, from=1-5, to=5-3, curve={height=-13pt}]
	\arrow["\overline{\dS}_{\D_g}^{\D_e}(\Omega_g)\circ \id_{(-)^{**}}"{description}, Rightarrow, from=5-3, to=2-4]
\end{tikzcd}\]
where the top square commutes by \eqref{prelim diag}, the left bottom square by naturality of $\overline{\dS}_{\D_g}^{\D_e}(\Omega_g)$, and the right triangle commutes trivially.
\end{proof}

\begin{theorem} \label{thm:sph_ext}
Let $\C$ be a spherical (unimodular) finite tensor category. There is an equivalence of $2$-groupoids 
\begin{equation}\label{eq:sph_ext_class}
   \SphExt(G,\C) \xrightarrow{\;\simeq\;} \MonFun\left(\uuG,\SphBrPic(\C)\right)\;.
\end{equation}
\end{theorem}
\begin{proof}
By Proposition~\ref{prop:SphExt-Z2}, the fixed points under the $B\underline{\zZ/2\zZ}$-action on left hand side of \eqref{eq:ext_classification} yields $\SphExt(G,\C)$. While on the right hand side we have $\MonFun\left(\uuG,\BrPic(\C)\right)^{B\underline{\zZ/2\zZ}} \simeq \MonFun(\uuG,\BrPic(\C)^{B\underline{\zZ/2\zZ}})$. The statement follows from Proposition~\ref{prop:SphBrPic-Z2} and \ref{prop:phi-is-Bz_2-equivariant}.
\end{proof}

\begin{proposition}
The equivalence of $2$-groupoids \eqref{eq:piv_ext} factorizes as follows
\begin{equation*}
    \begin{tikzcd}[row sep=normal]
    \SphExt(G,\C)  \ar[r, hookrightarrow] \ar[d, "\eqref{eq:sph_ext_class}",swap] &\PivExt(G, \C) \ar[d, "\eqref{eq:piv_ext}"] \\
    \MonFun\left(\uuG, \SphBrPic(\C)\right)\ar[r, hookrightarrow]&\MonFun\left(\uuG, \PivBrPic(\C)\right)
    \end{tikzcd}
\end{equation*}
\end{proposition}

\begin{proof}
Let $(\D, \fp_{\D})\in \SphExt(G,\C)$, via $\mathrm{E}$ we get an induced monoidal $2$-functor $\uuG\to \PivBrPic(\C)$ which maps $g$ to $\D_g$ together with pivotal $\C$-module structure given by the composition
\[ \id_{\D_g}\xRightarrow{\fp_{\D}|_{\D_g}} (-)_{\D_g}^{**} \xLeftarrow{\;\Omega_g\;}   \dS_{\D_g}^{\D_e} \] 
where $\Omega$ is the natural isomorphism from \eqref{eq:Serre_homogeneous_components}.
We need to show that this pivotal module structure is spherical, i.e. the outside of diagram 
\begin{equation}
    \begin{tikzcd}[column sep=small, row sep=small]
&& {\id_{\D_g} } &&&& {\dS_{\D_g}^{\D_e} \circ \dS_{\D_g}^{\D_e}} \\
\\
{(-)^{**}_{\D_g}} &&&& {(-)^{****}_{\D_g}} &&& {\dS_{\D_g}^{\D_e} \circ (-)^{**}_{\D_g}} \\
\\
&&&& {\dS_{\D_g}^{\D_e}}
\arrow["{{\mathcal{R}^{\D_g}}}", Rightarrow, from=1-3, to=1-7]
\arrow["{{\fp_{\D}|_{\D_g}}}"', Rightarrow, from=1-3, to=3-1]
\arrow["{{\mathcal{R}^{\D}|_{\D_g}}}", Rightarrow, from=1-3, to=3-5]
\arrow["{{\id \circ \fp_{\D}|_{\D_g}}}", Rightarrow, from=3-1, to=3-5]
\arrow["{{\Omega_g}}", Rightarrow, from=5-5, to=3-1]
\arrow["{{\Omega_g \circ \Omega_g }}", Rightarrow, from=1-7, to=3-5]
\arrow["{{\Omega_g \circ \id}}", Rightarrow, from=3-8, to=3-5]
\arrow["{{\id \circ \Omega_g}}", Rightarrow, from=1-7, to=3-8]
\arrow["{{\id \circ \fp_{\D}|_{\D_g}}}"', Rightarrow, from=5-5, to=3-8]
    \end{tikzcd}
\end{equation}
commutes. Indeed, the top left triangle commutes by sphericality of $\D,$ and the top right and bottom triangles commute trivially. The middle top triangle commutes by \eqref{first diag}.
\end{proof}


\section{Sphericalization of unimodular finite tensor categories and graded extensions}
\label{sec:sphericalization}
There is a general construction assigning a spherical tensor category to any tensor category \cite[\S 7.21]{etingof2015tensor}. The goal of this section is to relate this construction and the classification of extensions from Theorem \ref{thm:sph_ext}. In Section~\ref{subsec:sph_ftc}, we discuss this \textit{sphericalization} procedure for unimodular tensor categories. In Section~\ref{subsec:sph_modules}, we introduce an analogue to sphericalization for bimodule categories and obtain a monoidal $2$-functor between the Brauer Picard and spherical Brauer Picard $2$-categorical groups. Finally, in Section~\ref{subsec:sph_extensions} we show that the sphericalization construction commutes with the equivalence \eqref{eq:sph_ext_class} that classifies spherical extensions.


\subsection{Sphericalization of a unimodular finite tensor category}
\label{subsec:sph_ftc}
Let $\C$ be a unimodular finite tensor category (see \S\ref{subsec:spherical-G-graded}). The double-dual functor $(-)^{**}_\C$ together with the Radford isomorphism $\R^{\C}\colon \id_{\C}\xRightarrow{~\sim~} (-)_{\C}^{****}$ define a $\Ztwo$-action on $\C$ as described in Example \ref{action of Z_2 on C}. Here, the condition $\R^\C_{X^{**}}=(\R^\C_{X})^{**}$ follows from \cite[Theorem 4.4]{shimizu2023ribbon} and that \eqref{eq:iso_F_double-dual} applied to $(-)^{**}_\C$ is trivial.

\begin{definition}\cite[\S 7.21]{etingof2015tensor}
The {\em sphericalization} of $\C$ is the equivariantization $\C^\sph \coloneqq \C^{\Ztwo}$.
Explicitly, 
\begin{itemize}
    \item Objects of $\C^\sph$ are pairs $(X,\, f)$, where $X$ is an object of $\C$ and $f\colon X \xrightarrow{~\sim~} X^{**}$ is an isomorphism such that $f^{**}\circ f=\R^{\C}_X$, and 
    \item Hom spaces are given by $\Hom_{\C^\sph}((X,f),\, (X',f')) =\{h\in \Hom_\C(X,\, X') \mid f'\circ h =h^{**}\circ f \}$.
\end{itemize}
\end{definition}
\noindent The sphericalization construction comes with a forgetful tensor functor
\begin{align*}
    \forg\colon\C^\sph\longrightarrow\C,\;(X,f)\longmapsto X
\end{align*}
with identity morphisms as a tensor structure.

\begin{remark}
By Lemma~\ref{lem:equi-properties}, the sphericalization of a unimodular finite tensor category is again finite.
On the other hand, every tensor autoequivalence of a tensor category $\C$ defines a $\zZ$-action on it. In particular, the equivariantization of the action coming from the double-dual functor gives a tensor category $\C^\piv$. This tensor category is naturally endowed with a pivotal structure; it is however, not necessarily finite, even in the case $\C$ is a finite tensor category \cite{shimizu2015pivotal}.
\end{remark}

\begin{proposition}
\label{prop:equiv_unimodular}
Let $G$ be a finite group acting on a finite tensor category $\C$ by tensor autoequivalences. If $\C$ is unimodular, then so is $\C^G$.
\end{proposition}
\begin{proof}
Let $D$ denote the distinguished object of $\C^G$. We first show that $D\in \Rep(G)=\Vect^G \subset \C^G$. 
Let $P,\, \tilde{P}$ denote the projective covers of 
$\unit_\C,\, \unit_{\C^G}$, respectively. We know that  $D^*\hookrightarrow\tilde P$ is the socle.  Applying the forgetful functor $\forg\colon \C^G \longrightarrow \C$, we find that $\forg(D^*)\hookrightarrow\forg(\tilde P)\cong P^{\oplus n}$.  Since $D^*$ is invertible, $\forg(D^*)$ must also be invertible.  In particular, $\forg(D^*)$ must be a simple subobject of $P^{\oplus n}$, but then unimodularity of $\C$ implies that $\forg(D^*)\cong\unit$, which is equivalent to $D\in \Rep(G)$. This fact also follows from \cite[Ex.\ 3.16 and Thm.\ 3.19]{jakyad2025tensorfrobenius}.

Next, observe that it suffices to prove the statement in the case where $G$ is a simple group. Indeed, if 
$\{1\}=G_0\subset G_1 \subset \cdots \subset G_n=G$ is a composition series of $G$, then $\C^G$ is obtained from $\C$
by a sequence of equivariantizations with respect to the actions of the composition factors $G_1/G_0,\, G_2/G_1, \dots, G_n/G_{n-1}$.
If $G$ is simple and non-Abelian, then it has no nontrivial linear characters, and hence $D=\unit_{\C_G}$. 
Thus we may assume that $G=\mathbb{Z}/p\mathbb{Z}$ for some prime $p$. If $p=\mathrm{char}(\kk)$, then the previous argument applies. 
This leaves the case where $\kk G$ is semisimple, so that $\Rep(G)$ has $p$ non-isomorphic invertible objects (linear characters) $\chi: G \to \kk^\times$. 

Let $\Ind\colon \C \longrightarrow \C^G$ be the induction functor, i.e. the left adjoint to  $\forg$. Since
\[
\Hom_{\C^G}(\Ind(P),\, -) \cong  \Hom_{\C}(P,\, \forg(-)),
\]
and both functors $\forg$ and $\Hom_{\C}(P,\, -)$ are exact, the composition $\Hom_{\C^G}(\Ind(P),\, -)$ is exact as well. Hence $\Ind(P)$ is projective.

The projective cover of $\chi$ in $\C^G$ is $\tilde{P}_\chi = \tilde{P}\otimes \chi$. We have $\Ind(P)\cong \Ind(P)\otimes \chi$, and hence
it contains each $\tilde{P}_\chi$ as a direct summand. Therefore,
\[
\Ind(P)\cong \bigoplus_{\chi \in \widehat{G}}\, \tilde{P}\otimes \chi.
\]
Applying the forgetful functor to both sides and comparing indecomposable projective summands, we conclude that
$\forg(\tilde{P}) \cong P$. This means that  object $P$  has $p$ equivariant structures, one for each $\chi\in \widehat{G}$, and so $\tilde{P}_\chi^* \cong \tilde{P}_{\chi^{-1}}$.
It follows from \cite[\S~6.1]{etingof2015tensor} that $\C^G$ is unimodular.
\end{proof}

\begin{corollary}
\label{prop:sph_unimodular}
If $\C$ is unimodular, then so is $\C^\sph$.    
\end{corollary}
\begin{proof}
The category  $\C^\sph$ is a $\Ztwo$-equivariantization of $\C$, so the results
follows from Proposition~\ref{prop:equiv_unimodular}.
\end{proof}

\begin{proposition}\label{prop:sphericalization}
Let $\C$ be a unimodular finite tensor category. The finite tensor category $\C^\sph$ is endowed with a canonical pivotal structure $\fp$ given by 
$$\fp_{(X,f)}\coloneqq f\colon (X,f)\xrightarrow{~\sim~} (X^{**},\, f^{**})$$
that is spherical.
\end{proposition}
\begin{proof}
$\C^\sph$ and $\C$ are unimodular by Corollary \ref{prop:sph_unimodular}. Therefore, from \cite[Thm.\ 3.19]{jakyad2025tensorfrobenius} and \cite[Ex.\ 3.16]{jakyad2025tensorfrobenius}, it follows that the forgetful functor preserves the Radford isomorphism, i.e. $\forg(\R^{\C^\sph}_{(X,f)})=\R_X^\C$. By construction we have that $\R^{\C}_X=f^{**}\circ f=\forg(f^{**}\circ f)$. Since $\forg$ is faithful it follows that $\R^{\C^\sph}_{(X,f)}=f^{**}\circ f$ and thus the pivotal structure $\fp$ on $\C^\sph$ is spherical (see also \cite[Corollary 7.6]{etingof2004analogue} for the semisimple case). 
\end{proof}
The sphericalization of tensor categories extends to a $2$-functor for graded extensions.
\begin{proposition}Let $G$ be a finite group and $\C$ a unimodular finite tensor category. The assignment
    \begin{equation}
        (-)^\sph\colon \Ext(G,\C) \longrightarrow \SphExt(G,\C^\sph),\quad \D\longmapsto\D^\sph 
    \end{equation}
is a well-defined $2$-functor between $2$-groupoids.
\end{proposition}
\begin{proof}
According to Lemma \ref{lemma: unimodularity of extension}, any extension $\D$ is unimodular and thus a valid input for the sphericalization construction. At the level of $1$-cells, given an equivalence of extensions $F\colon \D\longrightarrow\D'$, the assigned $1$-cell is
\begin{equation*}
    F^\sph\colon \D^\sph\longrightarrow \D'^\sph, \quad (X,f)\longmapsto (F(X), \tilde{F}(f))
\end{equation*}
where $\tilde F(f)$ is given by the composition $F(X)\xrightarrow{F(f)} F(X^{**})\xrightarrow{~\eqref{eq:iso_F_double-dual}~}F(X)^{**}$. This means that the tensor functor $F^\sph$ is pivotal. 
\end{proof}


\subsection{Sphericalization of module categories}
\label{subsec:sph_modules}
In this section, we extend the sphericalization construction to $\C$-module categories, and show that the resulting pivotal $\C$-module is actually spherical.

Let $\C$ be a unimodular finite tensor category and $\M$ an invertible $\C$-bimodule category. The relative Serre functor 
$\dS_\M^\C$ plays the role of the double-dual functor for $\M$, and together with the bimodule Radford isomorphism $\R^{\M}\colon \id_{\M} \xRightarrow{~\sim~} \dS_\M^\C\circ \dS_\M^\C$ from \eqref{eq:Rad for modules over sph} define a $\Ztwo$-action on $\M$. Note that, since $\C$ is not necessarily pivotal, $\dS_\M^\C$ is only a twisted $\C$-bimodule equivalence, see \eqref{eq:twisted_serre_bimod}. However, $\dS_\M^\C\circ \dS_\M^\C$ is a $\C$-bimodule autoequivalence of $\M$, once we untwist the its bimodule structure with the Radford isomorphism $\R^{\C}$. 

\begin{definition}
Let $\C$ be a unimodular finite tensor category and $\M$ an invertible $\C$-bimodule category.
The {\em sphericalization} $\M^\sph$ of $\M$ is the equivariantization $\M^\sph \coloneqq \M^{\Ztwo}$.
Explicitly,
\begin{itemize}
    \item Objects of $\M^\sph$ are pairs $(M,\, s)$, where $M$ is an object of $\M$ and $s\colon M \xrightarrow{~\sim~} \dS_\M^\C(M)$ is an isomorphism such that $\dS_\M^\C(s) \circ s= \R^{\M}_M$, and 
    \item Hom spaces are given by $\Hom_{\M^\sph}((M,s),(N,t)) = \{p\in \Hom_\C(M,N) \mid t\circ p =\dS_\M^\C(p)\circ s \}$.
\end{itemize}
\end{definition}

Note that sphericalization of $\C$-module categories is a special case of the process of equivariantization of $\C$-module categories studied in \cite[\S3.5]{galindo2012module}.

\begin{lemma}
\label{lem:Mpiv-Cpiv-module}
The sphericalization $\M^\sph$ is endowed with the structure of a $\C^\sph$-bimodule category via 
\begin{equation*}
(X,f) \tl (M,s) \tr (Y,h) \coloneqq (X \tl M\tr Y, q)\;,
\end{equation*}
for $(X,f), (Y,h)\in\C^{\sph}$ and $(M,s)\in\M^{\sph}$, where $q$ is defined by the composition
\begin{equation*}   
q\colon X \tl M \tr Y  \xrightarrow{f \tl s\tr h} X^{**} \tl  \dS_\M^\C(M)\tr Y^{**} \xrightarrow{\;{\eqref{eq:twisted_serre_bimod}}\;~} \dS_\M^\C (X \tl M \tr Y)\;.
\end{equation*}
\end{lemma}
\begin{proof}
The associativity of the bimodule action is shown by a routine check involving the condition fulfilled by \eqref{eq:twisted_serre_bimod} as the twisted bimodule structure of $\dS_\M^\C$.
\end{proof}

The procedure of sphericalization of an invertible $\C$-bimodule category $\M$ can be alternatively be understood in terms of the associated finite multitensor category
\begin{equation*}
\mathbb{M}=
\begin{bmatrix}
\C & \M \\
\overline{\M} & \C
\end{bmatrix}	
\end{equation*}
where rigidity follows from \cite[Thm.\ 4.2]{spherical2022}. According to \cite[Prop.\ 4.11]{spherical2022}, we have that the double-duals in $\mathbb{M}$ for objects $M{\in}\M$ are given by the relative Serre functor $\dS^\C_\M(M)$. Additionally, the Radford isomorphism of $\mathbb{M}$ restricts on $\M$ to the bimodule Radford isomorphism $\mathcal{R}^\M$, by a similar argument to the one in the proof of Proposition \ref{prop:radford is monoidal}. Hence, the sphericalization of $\mathbb{M}$ is the finite multitensor category
\begin{equation*}
\mathbb{M}^\sph=
\begin{bmatrix}
\C^\sph & \M^\sph \\
\overline{\M}^\sph & \C^\sph
\end{bmatrix}
\end{equation*}
which in particular, shows that $\M^\sph$ is invertible. Moreover, from the description of double-duals in the equivariantization, we have that the relative Serre functor of $\M^\sph$ is given by
\begin{equation}\label{eq:Serre_sphericalization}
    \dS_{\M^\sph}^{\C^\sph}\left(M,\, s\right)\cong\left(\dS_\M^\C(M),\dS_\M^\C(s)\right)
\end{equation}
for an object $(M,s)\in\M^{\sph}$.

\begin{proposition}
\label{prop: Mpiv is a Cpiv module}
Let $\C$ be a unimodular finite tensor category and $\M$ an invertible $\C$-bimodule category. The sphericalization $\M^\sph$ is endowed with the structure of an invertible pivotal $\C^\sph$-bimodule category given by
    \begin{equation}\label{Mpiv pivotal structure}
    \tfp_{(M,\, s)}\coloneqq \dS^\C_\M \,\colon\; (M,s)\xrightarrow{~\sim\,}\dS_{\M^\sph}^{\C^\sph}(M,\, s)\;,
    \end{equation}
that is spherical in the sense of Definition \ref{def:spherical_module}.
\end{proposition} 
\begin{proof}
That $\tfp_{(M,\, s)}$ defines a pivotal structure follows directly from the description of the relative Serre functor \eqref{eq:Serre_sphericalization}. Now, from the proof of Proposition \ref{prop:sphericalization}, we know that the component of the bimodule Radford isomorphism of $\M^\sph$ associated to an object $(M,s)$ agrees with $\R^\M_M$. Since $\dS_\M^\C(s) \circ s= \R^{\M}_M= \R^{\M^\sph}_{(M,s)}$ by construction, we conclude bimodule sphericality.
\end{proof}
\begin{remark}
The sphericalization procedure can also be applied to (left) $\C$-module categories. In that situation, we additionally need the $\C$-module category $\M$ to be unimodular in the sense of \cite{yadav2023unimodular}, i.e. such that the dual tensor category $\C^*_\M$ is unimodular. Then, we obtain a module Radford isomorphism of the form \eqref{eq:Rad for modules over sph}, which allows to define the $\zZ_2$-action.
\end{remark}
The sphericalization construction for bimodule categories extends to an appropriate $2$-functor, as well.
\begin{proposition}
Let $\C$ be a unimodular finite tensor category. The assignment 
\begin{equation*}
 (-)^\sph\colon \BrPic(\C) \longrightarrow  \SphBrPic(\C^\sph),\quad \M\longmapsto\M^\sph   
\end{equation*}
is a well-defined monoidal $2$-functor.   
\end{proposition}

\begin{proof}
Proposition \ref{prop: Mpiv is a Cpiv module} ensures that $(-)^\sph$ is well-defined on objects. Given a $1$-cell $H\colon\M \longrightarrow \N$ in $\BrPic(\C)$, define 
\begin{equation*}
        H^{\sph}\colon \M^{\sph}\longrightarrow \N^{\sph}, \quad (M,s) \longmapsto (H(M), \tilde H(s))
\end{equation*}
where $\tilde H(s)$ is given by the composition $H(M)\xrightarrow{H(s)} H\circ\dS_{\M}^{\C}(M)\xrightarrow{~\eqref{eq:Serre_bimodule_equivalence}~}\dS_{\N}^{\C}\circ H(M)$. 
That $(H(M), \tilde H(s))$ belongs to $\N^{\sph}$ follows from commutativity of the diagram below
\begin{equation*}
\begin{tikzcd}[column sep=normal]
	{H(M)} && {H\circ \dS_\M^\C (M)} && {\dS_\N^\C\circ H(M)} \\
	&& {H\circ\dS_\M^\C\circ \dS_\M^\C (M)} && {\dS_\N^\C\circ H\circ \dS_\M^\C (M)} \\
	&&&& {\dS_\N^\C\circ\dS_\N^\C\circ H(M)}
	\arrow["{H(s)}", from=1-1, to=1-3]
	\arrow["{H(\R^\M_M)}"'{pos=0.7}, shift left, from=1-1, to=2-3]
	\arrow["{\R_{H(M)}^\N}"', shift right, curve={height=35pt}, from=1-1, to=3-5]
	\arrow["{\eqref{eq:Serre_bimodule_equivalence}}", from=1-3, to=1-5]
	\arrow["{H\circ\dS_\M^\C(s)}", from=1-3, to=2-3]
	\arrow["{\dS_\N^\C\circ H(s)}", from=1-5, to=2-5]
	\arrow["{\eqref{eq:Serre_bimodule_equivalence}}", from=2-3, to=2-5]
	\arrow["{\eqref{eq:Serre_bimodule_equivalence}^2}"'{pos=0.3}, from=2-3, to=3-5]
	\arrow["{\eqref{eq:Serre_bimodule_equivalence}}", from=2-5, to=3-5]
\end{tikzcd}
\end{equation*}
where, the top left triangle commutes by definition of $(M,s)$, and the bottom right one commutes trivially. The top square commutes by naturality of $\eqref{eq:Serre_bimodule_equivalence}$, and the bottom left region by commutativity of \eqref{eq: radford:mod}. That $H^{\sph}$ is an equivalence follows from $H$ being an equivalence. Lastly, pivotality of $H^{\sph}$ follows from the definition of $\tilde{H}(s)$. 

The data of a monoidal structure on a $2$-functor is defined in \cite[Def.\ 2.10]{davydov2021braided}. Given $\M,\N \in \BrPic(\C)$, consider the auxiliary finite multitensor category $\mathbb{B}$ from Proposition \ref{prop:radford is monoidal}.
The equivariantization of $\mathbb{B}$ is the multitensor category
\begin{equation*}
\mathbb{B}^\sph=
\begin{bmatrix}
\C^\sph & \M^\sph & (\M\boxtimes_{\C}\N)^\sph\\
\overline{\M}^\sph & \C^\sph & \N^\sph\\
(\overline{\N}\boxtimes_{\C}\overline{\M})^\sph & \overline{\N}^\sph & \C^\sph
\end{bmatrix}	
\end{equation*}
By Lemma \ref{lem:multitensor_equivalences} we obtain an equivalence
\begin{equation}
\begin{aligned}
\Psi_{\M,\N}\colon\M^\sph\boxtimes_{\C^\sph}\N^\sph&\longrightarrow(\M\boxtimes_{\C}\N)^\sph\\
    \left((m,s)\boxtimes (n,t)\right)&\longmapsto(m\boxtimes n, \mu_{\M,\N}\circ s\boxtimes t)
\end{aligned}
\end{equation}
of $\C^\sph$-bimodule categories.
Now given $\M,\N,\cL\in \BrPic(\C)$, we can similarly consider the finite multitensor category
\begin{equation*}
\mathbb{T}=
\begin{bmatrix}
\C & \M & \M\boxtimes_{\C}\N & \M\boxtimes_{\C}\N\boxtimes_{\C}\cL\\
\overline{\M} & \C & \N& \N\boxtimes_{\C}\cL\\
\overline{\N}\boxtimes_{\C}\overline{\M} & \overline{\N} & \C&\cL\\
\overline{\cL}\boxtimes_{\C}\overline{\N}\boxtimes_{\C}\overline{\M}&\overline{\cL}\boxtimes_{\C}\overline{\N}&\overline{\cL}&\C
\end{bmatrix}	
\end{equation*}
The induced associators on $\mathbb{T}^\sph$ provide bimodule natural isomorphisms
\begin{equation*}
    \alpha_{\M,\N,\cL}\colon \Psi_{\M,\N\cL}\circ \id_\M \boxtimes_{\C^\sph}\Psi_{\N,\cL}\xRightarrow{~\sim~\;} \Psi_{\M\N,\cL}\circ  \Psi_{\M,\N}\boxtimes_{\C^\sph}\id_\cL
\end{equation*}
That these fulfill the required conditions for a monoidal structure on $(-)^\sph$ follows from the pentagon axioms that they obey in the monoidal category $\mathbb{T}^\sph$. We can obtain all associators simultaneously by considering a multitensor category of the form of $\mathbb{T}$, but involving all of the (finitely many) invertible $\C$-bimodules.
\end{proof}


\subsection{Sphericalization and the classification of extensions}
\label{subsec:sph_extensions}
Recall that for any finite tensor category $\C$ there is an equivalence \eqref{eq:ext_classification}
\begin{equation}
	 \mathrm{E}\colon\Ext(G,\C) \longrightarrow \MonFun\left(\uuG,\BrPic(\C)\right)
\end{equation}
of $2$-groupoids as established in \cite[Theorem~7.7]{etingof2010fusion}. 
Given a monoidal $2$-functor $\mathsf{F}\colon \uuG \longrightarrow \BrPic(\C),$ we denote by $\D_{\mathsf{F}}$ the corresponding $G$-graded extension of $\C$. Reciprocally, given an extension $\D\in \Ext(G,\C),$ we denote by $\mathrm{E}_{\D}\colon \uuG \longrightarrow \BrPic(\C)$ the corresponding monoidal $2$-functor. In this section, we show the relation between the sphericalization procedure and graded extensions.

\begin{proposition}\label{prop:piv-graded-ext}
Let $\C$ be a unimodular finite tensor category and $\mathsf{F}\colon \uuG\to\BrPic(\C)$ a monoidal $2$-functor. There is a canonical equivalence
\begin{equation}\label{piv-graded-ext}
\left(\D_\mathsf{F}\right)^\sph \xrightarrow{~\sim~} \D_{(-)^\sph \circ \mathsf{F}}
\end{equation}
of spherical $G$-graded extensions of $\C^\sph$.
\end{proposition}
\begin{proof}
For an homogeneous object $(X, f)$ in the sphericalization $(\D_\mathsf{F})^\sph$, with $X{\in}\mathsf{F}(g)$,
consider the isomorphism
\begin{equation*}
    X\xrightarrow{~f~} X^{**} \xrightarrow{~\Omega_g^{-1}~} \dS_{\mathsf{F}(g)}^{\C}(X)
\end{equation*}
where $\Omega_g\colon \dS_{\mathsf{F}(g)}^{\C}\xRightarrow{~\sim~} (-)^{**}|_{\mathsf{F}(g)}$ is the natural isomorphism from Lemma \ref{lem:Serre_homogeneous}. Then, the pair $(X, \Omega_g^{-1}\circ f)$ belongs to $\D_{(-)^\sph \circ \mathsf{F}}$. Indeed, the condition
\begin{equation*}
\begin{tikzcd}[column sep=normal]
	{X} && {X^{**}} && {\dS_{\mathsf{F}(g)}^{\C}(X)} \\
	&& {X^{****}} && {\dS_{\mathsf{F}(g)}^{\C}(X^{**})} \\
	&&&& {\dS_{\mathsf{F}(g)}^{\C}\circ\dS_{\mathsf{F}(g)}^{\C}\left(X\right)}
	\arrow["{f}", from=1-1, to=1-3]
	\arrow["{\R_{X}^{\D_\mathsf{F}}|_{\mathsf{F}(g)}}"{description}, from=1-1, to=2-3]
	\arrow["{\R_{X}^{\mathsf{F}(g)}}"', shift right, curve={height=35pt}, from=1-1, to=3-5]
	\arrow["{\eqref{eq:Serre_homogeneous_components}}", from=1-3, to=1-5]
	\arrow["{f^{**}}", from=1-3, to=2-3]
	\arrow["{\dS_{\mathsf{F}(g)}^\C(f)}", from=1-5, to=2-5]
	\arrow["{\eqref{eq:Serre_homogeneous_components}}", from=2-3, to=2-5]
	\arrow["{\eqref{eq:Serre_homogeneous_components}^2}"'{pos=0.3}, from=2-3, to=3-5]
	\arrow["{\eqref{eq:Serre_homogeneous_components}}", from=2-5, to=3-5]
\end{tikzcd}
\end{equation*}
is fulfilled since the top left triangle commutes by definition of $f$, the top right square by naturality of $\Omega$, and the remaining left region by \eqref{first diag}. To define the functor \eqref{piv-graded-ext} on morphisms, consider a morphism $\lambda\colon (X, f)\to   (X', f')$ in $(\D_F)^\sph$. From the naturality of $\Omega$, it follows that
$\dS_{\mathsf{F}(g)}^{\C}(\lambda)\circ\Omega_g^{-1}\circ f=\Omega_g^{-1}\circ f\circ \lambda$ and thus $ \lambda\colon (X, \Omega_g^{-1}\circ f)\longrightarrow  (X', \Omega_g^{-1}\circ f')$ is a morphism in $\D_{(-)^\sph \circ \mathsf{F}}$.
A routine check shows that the assignment given by
\begin{align*}
	  \left(X\;,\; s\colon X\to \dS_{\mathsf{F}(g)}^{\C}(X)\right)\longmapsto   \left(X\;,\; X\xrightarrow{s} \dS_{\mathsf{F}(g)}^{\C}(X)\xrightarrow{\Omega_g} X^{**}\right)
\end{align*}
provides a quasi-inverse for \eqref{piv-graded-ext}.
To endow \eqref{piv-graded-ext} with a monoidal structure, consider homogeneous objects $(X_g, f_g)$ and $(X_h, f_h)$ in $(\D_\mathsf{F})^\sph$. Then, their tensor product gets assigned the object
\begin{align*}
  \left(X_g\otimes X_h\;,\;  X_g\otimes X_h \xrightarrow{f_g\otimes f_h}X_g^{**}\otimes X_h^{**}\cong (X_g\otimes X_h)^{**}\xrightarrow{\Omega_{gh}^{-1}}\dS_{\mathsf{F}(g)}^{\C}(X_g\otimes X_h) \right)\,.
\end{align*}
On the other hand, the tensor product of their images under \eqref{piv-graded-ext} is given by
\begin{align*}
    \left(X_g\otimes X_h\;,\;  X_g\otimes X_h \xrightarrow{f_g\otimes f_h}X_g^{**}\otimes X_h^{**}\xrightarrow{\Omega_{g}^{-1}\otimes\Omega_{h}^{-1}}\dS_{\mathsf{F}(g)}^{\C}(X_g)\otimes\dS_{\mathsf{F}(g)}^{\C}(X_h)\cong\dS_{\mathsf{F}(g)}^{\C}(X_g\otimes X_h) \right)\,,
\end{align*}
which coincide according to \eqref{eq:mono_OmegaD2}, and thus we can consider the trivial monoidal structure on \eqref{piv-graded-ext}.
Lastly, recall that the pivotal structure of an object $(X_g, f)$ in $(\D_\mathsf{F})^\sph$ is given by $f$, while the pivotal structure of $ (X_g, s) $ in $\D_{(-)^\sph \circ \mathsf{F}}$ is given by $  \Omega_g\circ s$, see Section \ref{sec: pivotal graded extensions}. Then the pivotal structure of $(X_g, \Omega_g^{-1}\circ f)$ is  given by $\Omega_g\circ\Omega_g^{-1}\circ f=f$ and thus \eqref{piv-graded-ext} is pivotal in the sense of \eqref{eq: F preserves pivotal structure}.
\end{proof}

\begin{proposition}
Let $\C$ be a unimodular finite tensor category. There exists a pseudo-natural equivalence

\begin{equation}\label{eq:pseudo-nat piv ext}
\begin{tikzcd}[row sep=large]
\Ext(G,\C) \ar[rr,"\mathrm{E}"]\ar[d,"(-)^{\sph}",swap] &&\MonFun\left(\uuG,\BrPic(\C)\right) \ar[d," (-)^\sph\circ -"]\ar[dll,Rightarrow,"\simeq"]\\
\SphExt(G,\C^\sph) \ar[rr,"\eqref{eq:sph_ext_class}"'] && \MonFun\left(\uuG,\SphBrPic(\C^\sph)\right) 
\end{tikzcd}
\end{equation}
\end{proposition}
\begin{proof}
Let $\tilde{\mathrm{E}}$ denote the composition of the $2$-functors $\mathrm{E}$, $(-)^\sph \circ -$ and  the inverse of \eqref{eq:sph_ext_class}. We define the desired pseudo-natural equivalence $(-)^{\rm sph} \xRightarrow{~\sim~} \tilde{\mathrm{E}}$ as follows: for an extension $\D\in \Ext(G, \C)$, consider the equivalence $\D^\sph \cong (\D_{\mathrm{E}_{\D}})^\sph  \xRightarrow{\eqref{piv-graded-ext}}\D_{(-)^\sph \circ \mathrm{E}_{\D}}$ as described in Proposition \ref{prop:piv-graded-ext}. Now, given a $1$-cell $F\colon\D_1\to \D_2$ in $\Ext(G, \C)$, we need to define a $2$-cell
      \begin{equation}\label{eq:PsNat diag}
  	\begin{tikzcd}[column sep=large, row sep=normal]
  		\D_1^{\sph} \ar[r,"F^{\sph}"]\ar[""{name=0}, d,swap,"\eqref{piv-graded-ext}"]& \D_2^{\sph} \ar[""{name=1}, d,"\eqref{piv-graded-ext}"] \\
  		\D_{(-)^\sph \circ \mathrm{E}_{\D_1}}  \ar[r,swap,"\tilde{\mathrm{E}}(F)"]& \D_{(-)^\sph \circ \mathrm{E}_{\D_2}}
  		\arrow["", shorten <=14pt, shorten >=14pt, Rightarrow, from=0, to=1]
  	\end{tikzcd}
  \end{equation}
obeying the required pseudo-naturality condition. To this end, note that the composition of the left and bottom functors assigns to an object $(X,f)$ in $\D_1^{\sph}$ of homogeneous degree $g\in G$, the following value
  \begin{align*}
  	(X, f)\longmapsto \left(X, \Omega_g^{-1}\circ f\right)\longmapsto \left(F(X), \tilde F(\Omega_g^{-1}\circ f)\right),
  \end{align*}                  
where $\tilde F(\Omega_g^{-1}\circ f)$ is given by the composition $F_g(X)\xrightarrow{F(\Omega_g^{-1}\circ f)}F_g\circ \dS_{(\D_1)_g}^{\C}(X)\xrightarrow{~\eqref{eq:Serre_bimodule_equivalence}~}\dS_{(\D_2)_g}^{\C}\circ F_g(X)$. On the other hand, the composition of the top and right functors yields
  \begin{align*}
  	(X, f)\longmapsto \left(F(X), \tilde F(f)\right)\longmapsto \left(F(X), \Omega_g^{-1}\circ\tilde F(f)\right),
  \end{align*}                  
where $\tilde{F}(f)$ is given by the composition $F(X)\xrightarrow{F(f)} F(X^{**})\xrightarrow{~\eqref{eq:iso_F_double-dual}~} F(X)^{**}$. It follows from Lemma \ref{lem:Serre_homogeneous}, that these two values agree and thus we can define \eqref{eq:PsNat diag} as the identity $2$-cell, thereby obtaining the desired pseudo-natural equivalence, which finishes the proof.
\end{proof}


\section{Obstruction theory}
\label{sec:obstructiontheory}

In this section, we develop an obstruction theory for pivotal and spherical extensions.
The general pattern we find is that there are two obstructions $O_1$ and $O_2$.
The first obstruction $O_1$ comes from the fact that some bimodule categories are not pivotalizable.
The second obstruction $O_2$ checks whether or not a given choice of bimodule pivotal structures is monoidal on the whole extension.


\subsection{An algebraic description of obstructions}\label{sec:algebraicobstruction}

Let $\C$ be a tensor category with a fixed pivotal structure $\mathfrak{p}$.
Recall the 1-cocycle
\[
S\colon \pi_0(\BrPic(\C))\longrightarrow \Inv(\Z(\C)),\; \M \longmapsto Z_\M
\]
defined in \eqref{eqn: hom S}. By Proposition~\ref{Z is a 1-cocycle}, pivotal invertible $\C$-bimodule categories are precisely elements of the kernel of $S$.

Let $G$ be a finite group. Consider a $G$-graded extension
\[
\D =\bigoplus_{g\in G}\, \D_g,\qquad \D_e=\C
\]
corresponding to a monoidal $2$-functor $\mathsf{F}\colon \uuG \longrightarrow \BrPic(\C),\;  g\longmapsto \D_g$.

For $\D$ to have a pivotal structure  extending that of $\C$, it is necessary (but, in general,  not sufficient) for it to have a structure of a pivotal $\C$-bimodule category. In particular, its homogeneous components $\D_g,\, g\in G$ must be invertible pivotal $\C$-bimodule categories (with respect to $\mathfrak{p}$).
In order to build such a structure, we would begin by trying to choose $\C$-bimodule pivotal structures:
\begin{equation}
\label{eqn pivotal structures on Cg}
\tilde{\mathfrak{p}}_g\colon \id_{\D_g}{\xRightarrow{~\sim\,}} \dS^{\C}_{\D_g},\; g\in G\,.
\end{equation}

The following composition gives an obstruction to being able to pick \eqref{eqn pivotal structures on Cg}:
\begin{equation}
\label{eqn: obstruction O1}
    O_1\coloneqq S\circ F\colon G \longrightarrow \Z(\C)^\times: g \mapsto Z_{\D_g}.
\end{equation}
where $S$ is the map from \eqref{eqn: hom S}.

Let $Z \mapsto  Z^g,\, Z\in \Z(\C),\, g\in G,$ denote the restriction of the action $\partial\circ F\colon G \to \Aut^{br}(\Z(\C))$ to $\Z(\C)^\times$, where $\partial$ is the map from \eqref{eqn: partial}.
Note that $Z_{\D_g}= S(\D_g)$, whereas $Z^g$ is shorthand for $\partial_{\D_g}(Z)$.

\begin{proposition}\label{prop:1-cocycle}
The map \eqref{eqn: obstruction O1} satisfies $O_1(gh)=O_1(g) O_1(h)^g$, i.e.
$O_1$ is a 1-cocycle on $G$.
\end{proposition}
\begin{proof}
This is an immediate consequence of Proposition~\ref{Z is a 1-cocycle}.
\end{proof}

\begin{corollary}\label{rem: O1Alg}
One can choose bimodule pivotal structures (with respect to $\mathfrak{p}$) on the homogeneous components \eqref{eqn pivotal structures on Cg} if and only if the obstruction $O_1$  vanishes as a function, i.e. $O_1(g)=\unit$  for all $g\in G$, (equivalently if it is trivial as an element $O_1 \in \widetilde{H}^1(G, \Inv(\Z(\C)))$ in reduced group cohomology\footnote{Recall that for a group $G$ acting on an abelian group $M$ the map $\widetilde{H}^k(G, M) \to H^k(G,M)$ from reduced to unreduced group cohomology is an isomorphism for $k\geq 2$. For $k=1$, the former is the group of `twisted' homomorphisms $G\to M$, i.e. functions satisfying the $1$-cocycle condition from Proposition~\ref{prop:1-cocycle} while the latter is the quotient thereof by the 1-coboundaries of the form $g \mapsto m^g m^{-1}$ for an $ m \in M$.}, also see \S~\ref{sec:homotopicalobstruction}).  
\end{corollary}
Let us denote $\dS_g:=\dS^{\C}_{\D_g}$ and $\id_g :=\id_{\D_g}$ for all $g\in G$.
The direct sum of module pivotal structures \eqref{eqn pivotal structures on Cg} would then be a natural isomorphism  
\begin{equation}
\label{eqn: direct sum p}    
\mathfrak{p}^\D=\bigoplus_{g\in G}\, \tilde{\mathfrak{p}}_g: \id_\D {\xRightarrow{~\sim\,}} \bigoplus_{g\in G}\, \dS_g {\xRightarrow{~\eqref{eq:Serre_homogeneous_components}\,}}(-)^{**} . 
\end{equation}

Next, we determine an obstruction for the natural isomorphism \eqref{eqn: direct sum p} to be a pivotal structure on $\D$, i.e. for $\mathfrak{p}^\D$ to be monoidal.  
The following composition of natural isomorphisms:
\begin{equation}
\label{eqn: O2alg}  
\id_{gh} \xRightarrow{\,\sim\,} \id_{g} \boxtimes_\C \id_{h}
\xRightarrow{ \tilde{\mathfrak{p}}_g \boxtimes_\C \tilde{\mathfrak{p}}_h }
\dS_g \boxtimes_\C \dS_h \xRightarrow{\,\sim\,}  \dS_{gh}
\xRightarrow{ \tilde{\mathfrak{p}}_{gh}^{-1}}  \id_{gh}, \qquad g,h\in G,
\end{equation}
where the middle isomorphism  is from Proposition~\ref{prop:Serre_monoidal},
determines a function
\[
O_2: G\times G \to \kk^\times\,.
\]

\begin{proposition}
The function $O_2$ is a $2$-cocycle whose cohomology class in $H^2(G, \kk^\times)$ is independent of the choice of bimodule pivotal structures 
$\tilde{\mathfrak{p}}_g$, $g\in G$, from
\eqref{eqn pivotal structures on Cg}. This class 
is trivial if and only if the isomorphisms $\tilde{\mathfrak{p}}_g$ can be chosen
so that $\mathfrak{p}^\D$ defined in \eqref{eqn: direct sum p} is an isomorphism of tensor functors, i.e. a pivotal structure on $\D$.
\end{proposition}

\begin{proof}
The $2$-cocycle condition is straightforward: both  
$O_2(f,g) \circ O_2(fg,h)$ and $O_2(g,h)\circ O_2(f,gh)$ 
coincide with the following composite:
\[
\id_{fgh} 
\xRightarrow{\,\tilde{\mathfrak{p}}_{fgh}^{-1}\,} 
\dS_{fgh} \cong 
\dS_{f} \boxtimes_{\C} \dS_{g} \boxtimes_{\C} \dS_{h} 
\xRightarrow{\,\tilde{\mathfrak{p}}_f \boxtimes_\C \tilde{\mathfrak{p}}_g \boxtimes_\C \tilde{\mathfrak{p}}_h\,} 
\id_{f} \boxtimes_\C \id_{g} \boxtimes_\C \id_{h}  
\xRightarrow{\,\sim\,} 
\id_{fgh}, \qquad f,g,h\in G.
\]
If $O_2(g,h) = c(g,h)c(g)^{-1}c(h)^{-1}$ for some function $c\colon G \to \kk^\times$,
then from \eqref{eqn: O2alg} we see that $\tilde{\fp}'_g := c(g)\tilde{\fp}_g$ satisfies
$\tilde{\fp}'_g \boxtimes_\C \tilde{\fp}'_h = \tilde{\fp}'_{gh}$ for all $g,h\in G$,
i.e. $\fp^\D$ is a pivotal structure on $\D$. 
Since each $\tilde{\fp}_g$ is determined up to a nonzero scalar, the converse also holds.
\end{proof}

\begin{remark}
Pivotal structures on $\D$ 
extending the given pivotal structure on $\C$ are parameterized by a torsor over $H^1(G,\, \kk^\times)$  (note that it is isomorphic to the group of tensor automorphisms of $\id_\D$ trivial on $\id_\C$).  
\end{remark}

Now suppose that $\mathcal C$ is a spherical (unimodular) finite tensor category.
Recall from Remark \ref{rem: 2 spherical structures} that although pivotal bimodule categories need not be spherical, their pivotal structures can nevertheless be modified to produce spherical structures.
In other words, an invertible bimodule is pivotalizable if, and only if, it is sphericalizable.
Thus we find that the function $O_1$ is precisely the first obstruction to spherical structures as well.

\begin{proposition}\label{prop:O1Alg-Spherical}
Assuming $\mathcal C$ is a spherical (unimodular) finite tensor category, one can choose bimodule spherical structures (with respect to $\mathfrak{p}$) on the homogeneous components \eqref{eqn pivotal structures on Cg} if and only if the obstruction $O_1$  vanishes as a function, i.e. $O_1(g)=\unit$  for all $g\in G$, (equivalently if it is trivial as an element $O_1 \in \widetilde{H}^1(G, \Inv(\Z(\C))_2)$.)
\end{proposition}

The only difference between Corollary \ref{rem: O1Alg} and Proposition \ref{prop:O1Alg-Spherical} is that the image of $O_1$ lies in the 2-torsion subgroup $\Inv(\mathcal Z(\mathcal C))_2 \subseteq \Inv(\mathcal{Z}(\mathcal C))$.
This follows from the existence of the Radford isomorphism (\ref{eq:Rad for modules over sph}).

The second obstruction to sphericality is, once again, the same function as before.

\begin{proposition}
    Suppose that $\mathcal C$ is a spherical (unimodular) finite tensor category, and the first obstruction $O_1$ for the extension vanishes.
    Upon choosing spherical bimodule structures on each of the graded components $\,\mathcal C_g$, the resulting function $O_2$ is a $2$-cocycle whose cohomology class in $$H^2(G, (\kk^\times)_2) = \left\{ \begin{array}{cc} H^2(G, \mathbb{Z}/2\mathbb{Z}) & \mathrm{char}(\kk)\neq 2\\ 0 & \mathrm{char}(\kk)=2\end{array}\right.$$ is independent of the choice of bimodule spherical structures $\tilde{\mathfrak{p}}_g$, $g\in G$, from
    \eqref{eqn pivotal structures on Cg}. This class is trivial in $H^2(G, (\kk^{\times})_2)$ if and only if the isomorphisms $\tilde{\mathfrak{p}}_g$ can be chosen
    so that $\mathfrak{p}^\D$ defined in \eqref{eqn: direct sum p} is a spherical structure on $\D$.
\end{proposition}

\begin{remark}\label{rem:O2Alg-Spherical vs Pivotal}
    A key difference here is that both the 2-cocycles and the 1-cochains that determine the coboundaries can only take values in $(\kk^\times)_2$.
    This implies that the map $H^2(G,(\kk^\times)_2)\to H^2(G,\kk^\times)$, induced by inclusion, need not be injective.
    If the class of $O_2$ is a nontrivial element in this kernel, then this means that the spherical structure on $\mathcal C$ can be extended to a pivotal structure, but not to a spherical one.
\end{remark}

\begin{corollary}
    If $\mathcal D$ is a pivotal $G$-graded unimodular finite tensor category over an algebraically closed field $\kk$ of characteristic $2$, then $\mathcal D$ is spherical if and only if $\mathcal C:= \mathcal D_e$ is spherical. In $\mathrm{char}(\kk)\neq 2$, this is obstructed by a class in $\mathrm{ker}\left(H^2(G, \mathbb{Z}/2\mathbb{Z}) \to H^2(G, \kk^{\times}) \right) = \mathrm{Im}\left(H^1(G, \kk^{\times}) \to H^2(G, \mathbb{Z}/2\mathbb{Z})\right).$
\end{corollary}


\subsection{A homotopical perspective} \label{sec:homotopicalobstruction}

We interpret our algebraic obstruction theory from \S \ref{sec:algebraicobstruction}  in homotopical terms. 

Let $\C$ be a tensor category with a fixed pivotal structure, $G$ a group, and a given $G$-graded extension classified by a monoidal $2$-functor ${\mathsf{F}}\colon \uuG \to \BrPic(\C)$. Then, by Theorem~\ref{thm:piv_ext}, pivotal structures on the $G$-graded extension, compatible with the given pivotal structure on $\C$, are classified by monoidal lifts:
\[\begin{tikzcd}
 & \PivBrPic(\C) \arrow[d, "\forg"] \\
\uuG \arrow[r, "{\mathsf{F}}"'] \arrow[ru, dashed, "\widetilde{\mathsf{F}}"] & \BrPic(\C)   
\end{tikzcd}\]

\begin{corollary}\label{cor:trivializationofelement}
For $\C$ a tensor category with a fixed pivotal structure, $G$ a group and a given $G$-graded extension classified by a monoidal functor ${\mathsf{F}}\colon\uuG \to \BrPic(\C)$, pivotal structures on the graded extension are equivalent to trivializations of the element 
\begin{equation}\label{eq:whiskering}
\left(\begin{tikzcd}[scale=0.75, ampersand replacement=\&]
	{\underline{\underline{G}}} \& {\bf{BrPic}(\mathcal C)} \&\& {\bf{BrPic}(\mathcal C)}
	\arrow["{\mathsf{F}}", from=1-1, to=1-2]
	\arrow[""{name=0, anchor=center, inner sep=0}, "{\mathrm{id}}"', curve={height=16pt}, from=1-2, to=1-4]
	\arrow[""{name=1, anchor=center, inner sep=0}, "{\mathrm{id}}", curve={height=-16pt}, from=1-2, to=1-4]
	\arrow["{\mathbf{S}}", shorten <=5pt, shorten >=5pt, Rightarrow, from=1, to=0]
\end{tikzcd} \right) \in \Omega_{\mathsf{F}}\MonFun(\uuG, \BrPic(\C)).\end{equation}
Here, $\bS$ denotes the Serre pseudo-natural equivalence constructed in Proposition \ref{prop:RSpNE-is-Monoidal}.
\end{corollary}
\begin{proof}
    By Corollary~\ref{prop:PivBrPic-is-BZ-equivariant}, the monoidal $2$-groupoid $\PivBrPic(\C) = \BrPic(\C)^{B\underline{\mathbb{Z}}}$ is the monoidal $2$-groupoid of homotopy fixed points for an action of the $2$-categorical group $B\underline{\mathbb{Z}}$ constructed in Proposition~\ref{BZ_action_BrPicC}. 
Therefore, compatible pivotal structures on the graded extension are also equivalent to $B\underline{\mathbb{Z}}$-fixed point data on the element  ${\mathsf{F}}\in \MonFun(\underline{G}, \BrPic(\C)) $ with $B\underline{\mathbb{Z}}$-action on $\MonFun(\underline{G}, \BrPic(\C))$ determined by that on $\BrPic(\C)$. As explained in Example~\ref{ex: fixed point for BZ and BZ2 actions}, given an element $y \in Y$ in a $2$-groupoid $Y$ with a $B\underline{\mathbb{Z}}$-action, a fixed point structure on $y$ amounts to a choice of trivialization of the element $\lambda_y \in \Omega_y Y$ determined by the action.
\end{proof}

Thus, the homotopy class of the composite~\eqref{eq:whiskering}
in the group $\pi_1 (\MonFun(\uuG, \BrPic(\C)); {\mathsf{F}})$ is an (obvious) complete obstruction to the existence of a pivotal structure.

Translated into homotopy theory, the $2$-groupoid $\MonFun(\underline{G}, \BrPic(\C))$ is the space of pointed maps $\Map_*(B G, B\BrPic(\C))$ where $BG$ and $B\BrPic(\C)$ denote the respective classifying spaces.

\begin{proposition}\label{prop:obstruction}
    Let $X$ be a pointed connected $3$-groupoid and $F\in \Map_*(BG, X)$. Then, there is a long exact sequence of group homomorphisms:
    $$0 \to \widetilde{H}^2(BG, \pi_3 X)\to \pi_1(\Map_*(BG,X); F) \to \widetilde{H}^1(BG, \pi_2 X) \to \widetilde{H}^3(BG, \pi_3 X) .$$
Here,  $\widetilde{H}$ denotes reduced group cohomology\footnote{
We remind the reader of the following homotopical description of reduced and unreduced group cohomology: Write $K(A,n)$ for an Eilenberg-MacLane space of an abelian group $A$ and recall that \[\pi_0 \Map(BG, K(A,n)) = H^n(BG, A)\hspace{1cm} \pi_0 \Map_*(BG, K(A,n)) = \widetilde{H}^n(BG, A)\] computes unreduced and reduced cohomology with trivial coefficients, respectively. More generally, for a $G$-action on $A$ with induced $G$-action on $K(A,n)$ with homotopy fixed point space $K(A,n)^{hG}$, unreduced, resp. reduced, cohomology with twisted coefficients can be computed as follows:
\[ \pi_0 K(A,n)^{hG} = H^n(BG, A) \hspace{1cm} \pi_0 \mathrm{fib}\left(K(A,n)^{hG} \to K(A,n)\right) = \widetilde{H}^n(BG, A) 
\] } with coefficients twisted by the action $\pi_1(G) \stackrel{\pi_1(F)}{\to} \pi_1(X) \to \Aut(\pi_n(X))$ for $n=2,3$, respectively,  where the latter map is the canonical action of $\pi_1(X)$ on $\pi_n(X)$. 
\end{proposition}
\begin{proof}
Recall that $\Map(BG, X) \simeq X^{hG}$ is the space of homotopy fixed points for the trivial $G$-action on $X$. Similarly, $\Map_*(BG, X) = \fib(X^{hG} \to X)$ is the homotopy fiber of the forgetful map at the basepoint of $X$, i.e. the space of $G$-fixed point structures on the basepoint of $X$. In particular, the element $F \in \Map_*(BG,X)$ defines such a $G$-fixed point structure on the basepoint. The truncation map $X\to \tau_{\leq 2} X$ is (trivially) $G$-equivariant for the trivial action. Its fiber $K(\pi_3 X,3)$ therefore inherits a $G$-action from the $G$-fixed point structure on the basepoint of $\tau_{\leq 2} X$ corresponding to $F$. This lifts the Postnikov fiber sequence  $K(\pi_3 X,3) \to X \to \tau_{\leq 2} X$ to a fiber sequence in $G$-spaces. Taking $G$-fixed points, we therefore obtain a map of fiber sequences in pointed spaces
\[
\begin{tikzcd}
\ldots \arrow[r] & K(\pi_3 X, 3)^{hG}  \arrow[r]\arrow[d] &  (X)^{hG}\simeq \Map(BG, X)\arrow[r]\arrow[d]& (\tau_{\leq 2} X)^{hG} \simeq \Map(BG, \tau_{\leq 2} X)\arrow[d]\\ 
\ldots \arrow[r] & K(\pi_3 X, 3)  \arrow[r] & X \arrow[r]& \tau_{\leq 2}X\\ 
\end{tikzcd}.\]
Taking vertical fibers therefore results in a fiber sequence (with basepoints recorded):
\begin{equation}\label{eq:fiberseq}
\mathrm{fib}\left(K(\pi_3 X,3)^{hG} \to K(\pi_3 X, 3)\right) \to (\Map_*(BG, X), F) \to (\Map_*(BG, \tau_{\leq 2}X), F).
\end{equation}
Running the same argument for $\tau_{\leq 2}X \to \tau_{\leq 1} X$, we find a fiber sequence 
\begin{equation}\label{eq:fiberseq2}
\mathrm{fib}\left(K(\pi_2X, 2)^{hG} \to K(\pi_2 X,2) \right) \to \Map_*(BG, \tau_{\leq 2} X) \to \Map_*(BG, \tau_{\leq 1} X) = \Hom(G, \pi_1(X))
\end{equation}
where the last space is the set of group homomorphisms $G \to \pi_1X$. In particular, the first map is fully faithful (i.e. injective on $\pi_0$ and an isomorphism on higher homotopy groups) and we thus find $\pi_1 (\Map_*(BG, \tau_{\leq 2}X); F) = \widetilde{H}^1(BG, \pi_2 X)$ and $\pi_{i} = 0 $ for $i>1$.  

Plugging these into the long exact sequence of homotopy groups associated to the fiber sequence~\eqref{eq:fiberseq} therefore yields the desired long exact sequence.
\end{proof}

In our case, $X= B \BrPic(\C)$ is the $3$-groupoid with homotopy groups \cite[Prop 7.1]{etingof2010fusion}
    \begin{enumerate}[$($\rm i$)$]
    \item    $ \pi_0 X = *$
        \item $\pi_1X = \pi_0\BrPic(\C)$ is the group of equivalence classes of invertible bimodule categories over $\C$.
        \item $\pi_2 X = \pi_1\BrPic(\C)\cong \Inv(\Z(\C))$, the group of isomorphism classes of invertible objects in $\Z(\C)$.
        \item $\pi_3 X = \pi_2\BrPic(\C)\cong\mathbb \kk^\times$.
    \end{enumerate}
    Moreover, while the action of $\pi_1X$ on $\pi_2 X = \Inv(\Z(\C))$ can be non-trivial (and hence leads to twisted coefficients below), the action of $\pi_1 X$ on $\pi_3 X$ will always be trivial. 
    
    Thus, Proposition~\ref{prop:obstruction} immediately yields (using that $\widetilde{H}^k \to H^k$ is an isomorphism for $k\geq 2$ and $\pi_0 \BrPic(\C)$ acts trivially on $\pi_2 \BrPic(\C) = \kk^{\times}$)  :

\begin{corollary}\label{cor:LESconcrete}
    Let $G$ be a finite group, $\C$  a finite tensor category and ${\mathsf{F}}\colon\uuG \to \BrPic(\C)$ a monoidal $2$-functor. Then, there is a long exact sequence of abelian groups
    $$0 \to H^2(BG, \kk^{\times})\to \pi_1(\MonFun(\uuG,\BrPic(\C)); {\mathsf{F}}) \to \widetilde{H}^1(BG, \Inv(\Z(\C))) \to H^3(BG, \kk^{\times}).$$
Here,  $\widetilde{H}$ denotes reduced group cohomology with coefficients twisted by the action $G \to \pi_0 \BrPic(\C) \to \Aut(\Z(\C)^{\times})$.
\end{corollary}

Thus, following Corollary~\ref{cor:trivializationofelement}, Corollary~\ref{cor:LESconcrete} immediately yields:
    
\begin{corollary}
Compatible pivotal structures on a $G$-graded extension ${\mathsf{F}}$ of a pivotal tensor category $\C$ are obstructed by classes
\begin{align*}
O_1(F) & \in \mathrm{ker}\left(\widetilde{H}^1(BG,  \Inv(\Z(\C))) \to H^3(BG, \kk^\times)\right)\\
O_2(F) & \in  H^2(BG, \kk^{\times})
\end{align*}
where $\widetilde{H}$ denotes reduced group cohomology with coefficients twisted by the action 
\[
G\to \pi_0 \BrPic(\C) \to \Aut(\pi_1 \BrPic(\C)) = \Aut(\Inv(\Z(\C)))
\]
where the map $\widetilde{H}^1(BG, \Inv(\Z(\C))) \to H^3(BG, \kk^\times)$ is given by composing with the $3$-cocycle $\alpha \in H^3(B\Z(\C)^{\times}, \kk^{\times})$ classifying the monoidal structure of the groupoid of invertible objects and invertible morphisms in the Drinfeld center $\Z(\C)$. If both classes vanish, then compatible pivotal structures form a torsor over the group of group homomorphisms
\[
\widetilde{H}^1(BG, \kk^{\times}) = \Hom(G, \kk^{\times}).
\]
\end{corollary}
These classes unpack to the classes constructed in \S\ref{sec:algebraicobstruction}.

Similarly, by Proposition~\ref{prop:SphBrPic-Z2}, an extension of the spherical structure on a $G$-graded extension of a spherical tensor category $\C$ amounts to $B \underline{\Ztwo}$ fixed point structure on the classifying map $F\in \MonFun(\underline{G}, \BrPic(\C))$. By Example~\ref{ex: fixed point for BZ and BZ2 actions}, a $B\underline{\Ztwo}$-fixed point structure on an element $y$ of a $2$-groupoid $Y:=\MonFun(\underline{G}, \BrPic(\C)) $ in turn amounts to a trivialization of an element $\lambda \in \Omega_y Y$ which is compatible with the trivialization of $\lambda^2$ induced by the $B\underline{\Ztwo}$-action. Equivalently, this is a trivialization of an element in $\mathrm{fib}(2: \Omega_y Y \to \Omega_y Y)$. Thus, the corresponding class in $\pi_0 \mathrm{fib}(2:\Omega_y Y \to \Omega_y Y)$ serves as an obstruction.

\begin{proposition}\label{prop:sphericalobstruction}
    Let $X$ be a pointed connected $3$-groupoid,  $F\in \Map_*(BG, X)$ and suppose that $2: \pi_3 X \to \pi_3X$ is surjective.  Then, there is a long exact sequence of groups: 
$$0 \to \widetilde{H}^2(BG, (\pi_3 X)_2)\to \pi_0\left([\Omega_F\Map_*(BG,X)]_2\right) \to \widetilde{H}^1(BG, (\pi_2 X)_2).$$
Here, for an abelian group $A$, $A_2$ denotes the subgroup of elements of order $2$, and $[\Omega_F\Map_*(BG,X)]_2:= \fib(2: \Omega_F\Map_*(BG, X) \to \Omega_F\Map_*(BG, X)).$
\end{proposition}
\begin{proof}
For a pointed space $X$, we write $[\Omega X]_2:= \fib(2: \Omega X \to \Omega X)$, an expression which by naturality of $2: \Omega - \to \Omega-$ is functorial in the pointed space. Commuting limits, it follows that for an abelian group $A$ with an action by a group $G$, the space  $[\Omega (\fib(K(A,n)^{hG} \to K(A,n)))]_2$ has homotopy groups sitting in a long exact sequence 
\[
  \ldots \to \widetilde{H}^{n-2}(BG, A_2)  \to  \pi_1 \to \widetilde{H}^{n-3}(BG, A/2A) \to  \widetilde{H}^{n-1}(BG, A_2) \to \pi_0 \to \widetilde{H}^{n-2}(BG, A/2A). 
\]
In particular, if $n=2$, then the homotopy groups are $\pi_0 = \widetilde{H}^1(BG, A_2)$ and zero otherwise, and if $2:A \to A$ is surjective, then $\pi_k = \widetilde{H}^{n-1-k}(BG, A_2)$. 
Using this, and applying $[\Omega - ]_2$ to the fiber sequences~\eqref{eq:fiberseq} and~\eqref{eq:fiberseq2}, the statement follows as in the proof of Proposition~\ref{prop:obstruction}. 
\end{proof}

In our case $X= B \BrPic(\C)$ and $\pi_0 \BrPic(\C)$ acts trivially on $\pi_2 \BrPic(\C) = \kk^{\times}$. Moreover, since $\kk$ is algebraically closed,  $(-)^2: \kk^{\times} \to \kk^{\times}$ is surjective with kernel $(\kk^{\times})_2$ isomorphic to $\mathbb{Z}/2\mathbb{Z}$ if $\mathrm{char}(\kk) \neq 2$ and $0$ if $\mathrm{char}(\kk) =2$. Therefore, Proposition~\ref{prop:sphericalobstruction} immediately yields:

\begin{corollary}\label{cor:LESconcretespherical}
    Let $G$ be a finite group, $\C$  a finite tensor category and ${\mathsf{F}}\colon\uuG \to \BrPic(\C)$ a monoidal $2$-functor. Then, there is a long exact sequence of abelian groups
    $$0 \to H^2(BG, \mathbb{Z}/2\mathbb{Z})\to \pi_0([\Omega_F\MonFun(\uuG,\BrPic(\C))]_2) \to \widetilde{H}^1(BG, \Inv(\Z(\C))_2)$$
    if $\mathrm{char}(\kk) \neq 2$ and 
     $$0 \to 0\to \pi_0([\Omega_F\MonFun(\uuG,\BrPic(\C))]_2) \to \widetilde{H}^1(BG, \Inv(\Z(\C))_2)$$ if $\mathrm{char}(\kk)=2$.
Here,  $\widetilde{H}$ denotes reduced group cohomology with coefficients twisted by the action $G \to \pi_0 \BrPic(\C) \to \Aut(\Z(\C)^{\times})$.
\end{corollary}

Since the middle group is an obstruction group for extensions of the spherical structure, this yields:

\begin{corollary}
    Compatible spherical structures on a $G$-graded extension of a spherical (unimodular) finite tensor category $\C$ are obstructed by classes
    \begin{align*}
        O_1(F) & \in \widetilde{H}^1(BG,  \Inv(\Z(\C))_2) \\
        O_2(F) & \in \left\{ \begin{array}{lr}  H^2(BG,  \Ztwo) & \mathrm{char}(\kk) \neq 2\\ 0 & \mathrm{char}(\kk)= 2\end{array}\right. 
    \end{align*}
    where $\widetilde{H}$ denotes reduced group cohomology with coefficients twisted  as above and  $\Inv(\Z(\C))_2$ denotes the subgroup of $\Inv(\Z(\C))$ of elements of order dividing $2$.
    If both classes vanish, then compatible spherical structures form a torsor over the group of homomorphisms $G \to \Ztwo$ if $\mathrm{char}(\kk) \neq 2$ (resp. there is a unique one if $\mathrm{char}(\kk) = 2$). 
\end{corollary}


\subsection{Examples}
\begin{example}
Let us consider the above pivotal obstruction theory in the case of pointed fusion categories.
Let $C$ be a normal subgroup of a finite group $D$ and let $G=D/C$.
Then $\D=\Vect_D$ is a $G$-graded extension of $\C=\Vect_C$.
Pivotal structures in this case are simply characters, so extension of pivotal structures corresponds to the classical problem of extending a linear character from a subgroup to a group.
The group $G$ acts on characters of $C$ by conjugation, $\chi(-)\mapsto\chi(g^{-1}(-)g)$. By Remark~\ref{rem:nonpivotalizable more general}, the homogeneous component corresponding to the coset $xC,\, x\in D,$ is pivotalizable if and only if $\phi\times \phi$ vanishes on the stabilizer of $x$ in $C\times C^{\op}$. The latter is equal
to $\{(xcx^{-1},\, c^{-1}) \mid  c\in C \}$ and so  the first obstruction $O_1$ vanishes if and only if $\chi$ is $G$-invariant. And indeed, \[
     O_1 \in \widetilde{H}^1\big(G, \Hom(C, \kk^{\times})\big) \subseteq \widetilde{H}^1\big(G, \Inv(\Z(\Vect_C))\big) 
   \]
    unpacks to the function 
     \[
     G = D/C \ni d_0 C \mapsto \chi(d_0^{-1}(-)\, d_0 ) \in \Hom(C, \kk^{\times}).
     \]

The obstruction to lifting a character 
\(\chi \in H^1(C,\kk^\times)^G\) to a character of \(D\)
is determined by means of the transgression in the 
Lyndon–Hochschild–Serre five-term exact sequence (also known as the inflation-restriction exact sequence):
\[
0 \longrightarrow H^{1}\!\bigl(G,\kk^{\times}\bigr)
 \xrightarrow{\ \mathrm{inf}\ } H^{1}\!\bigl(D,\kk^{\times}\bigr)
 \xrightarrow{\ \mathrm{res}\ } H^{1}\!\bigl(C,\kk^{\times}\bigr)^{G}
 \xrightarrow{\ \mathrm{tr}\ } H^{2}\!\bigl(G,\kk^{\times}\bigr).
\]
Here, \(\mathrm{tr}(\chi)\) is represented by the cocycle
\[
(x,y) \longmapsto \chi\!\bigl(s(x)s(y)s(xy)^{-1}\bigr)\in \kk^\times,
\qquad x,y\in G,
\]
where \(s: G \to D\) is any set-theoretic section satisfying \(s(1)=1\).
This coincides with the obstruction defined in~\eqref{eqn: O2alg}; that is,
\[
O_2 = \mathrm{tr}(\chi) \in H^2(G,\kk^\times).
\]

Here are some explicit examples.
\begin{enumerate}
    \item[(i)] Let $C=\zZ/4\zZ$ with a character $\chi: C \to \mathbb{C}^\times$ given by $\chi(n)=i^n$
    and  let $D$ be the dihedral group of order $8$. Then $\chi$ is not conjugation invariant, so the 
    obstruction $O_1(\chi)$ is nontrivial.
    \item[(ii)] Let $C = \Ztwo$ with an injective character $\chi: \Ztwo \to \kk^{\times}$  and let $D$ be a non-Abelian central extension of ($\Ztwo \times \Ztwo)$ by $C$, i.e. either dihedral or quaternion group.
    Clearly, the first obstruction vanishes, because conjugation does nothing.
    Since $C \subset [D,D]$, any character of $D$ must vanish on $C$,
    so the obstruction $O_2(\chi)$ is non-trivial in this case.
    \item[(iii)] Let $C = \Ztwo$ with an injective character $\chi: \Ztwo \to \kk^{\times}$  and let $D$ be $\mathbb Z/4\mathbb Z$.
    Since $\chi$ extends to $D$ by the formula given in (i), we know that both obstructions must vanish.
    This extension to $D$ is not spherical though, and indeed there can be no spherical extension, because $O_2=\mathrm{tr}(\chi)$ is nontrivial in $H^2(\Ztwo,\Ztwo)$, despite being trivial in $H^2(\Ztwo,\mathbb C^\times)$.
\end{enumerate}
\end{example}

\begin{example}
    If $\Inv(\Z(\C))=1$ and $H^2(G,\kk^\times)=0$, then all pivotal obstructions vanish automatically, so any pivotal structure can be extended. On the other hand, if $\Inv(\Z( \C))_2 = 1 $ and $H^2(G, \mathbb{Z}/2\mathbb{Z})=0$ (resp. no additional condition in characteristic two), then any spherical structure can be extended. 

    For example, if $\Inv(\C)=1$ and $\C$ admits a nondegenerate braiding (such as the Fibonacci category), then it follows that $\Z(\C)\simeq\C\boxtimes\C^{rev}$, and therefore $\Inv(\Z(\C))\cong1$.
    If $\kk$ is algebraically closed and characteristic zero, then $H^2(G,\kk^\times)\cong H_2(G;\zZ)$, so any group with trivial Schur multiplier will satisfy the desired property.
    These Schur-trivial groups include all cyclic groups, and all groups for which all Sylow $p$-subgroups are Schur-trivial (e.g. $S_3$).
\end{example}

\begin{example}
Let $\C=\Vect_A$ for some finite abelian group $A$, and let $\D$ be the Tambara-Yamagami category $\C(A,\chi,\tau)$ determined by the nondegenerate symmetric bicharacter $\chi\colon A\times A\to\mathbb C^\times$, and $\tau=\pm|A|^{-1/2}$ (see \cite{tambarayamagami} for details).
By definition, $\D$ is a $\zZ/2\zZ$ extension of $\D_0=\C$ with the nontrivial homogeneous component $\D_1=\Vect$.
By Remark~\ref{rem:nonpivotalizable more general}, $\D_1$ is not pivotalizable, unless the pivotal structure on $\Vect_A$ is the trivial one. The obstruction $O_2$ vanishes since $H^2(\Ztwo,\,\kk^\times)=0$. So the pivotal structures on $\D$ are in bijection with 
$H^1(\Ztwo,\,\kk^\times)\cong\Hom(\Ztwo,\,\kk^\times)$, and both of these are spherical. 
\end{example}


\bibliographystyle{alpha}
\bibliography{references}

\end{document}

%% file: macros.tex
\makeatletter
\@namedef{subjclassname@2020}{%
  \textup{2020} Mathematics Subject Classification}
\makeatother

\allowdisplaybreaks

\theoremstyle{plain}
\newtheorem{theorem}{Theorem}[section]
\newtheorem{lemma}[theorem]{Lemma}

\newtheorem{proposition}[theorem]{Proposition}

\newtheorem{corollary}[theorem]{Corollary}

\theoremstyle{definition}
\newtheorem{definition}[theorem]{Definition}
\newtheorem{example}[theorem]{Example}

\theoremstyle{remark}
\newtheorem{remark}[theorem]{Remark}

\makeatother

\DeclareMathOperator{\End}{End}
\DeclareMathOperator{\Ind}{Ind}

\newcommand{\eqtriv}{\iota}

\newcommand{\C}{\mathcal{C}}

\newcommand{\D}{\mathcal{D}}
\newcommand{\E}{\mathcal{E}}
\newcommand{\Z}{\mathcal{Z}}

\newcommand{\cL}{\mathcal{L}}
\newcommand{\M}{\mathcal{M}}
\newcommand{\N}{\mathcal{N}}

\newcommand{\R}{\mathcal{R}}

\newcommand{\dS}{\mathbb{S}}

\newcommand{\bA}{\mathbf{A}}

\newcommand{\dN}{\mathds{N}}
\newcommand{\bS}{\mathbf{S}}

\newcommand{\fp}{\mathfrak{p}}
\newcommand{\fq}{\mathfrak{q}}

\newcommand{\forg}{\mathrm{forg}}

\newcommand{\zZ}{\mathbb{Z}}
\newcommand{\Ztwo}{\mathbb{Z}/2\mathbb{Z}}
\def\Map        {\mathrm{Map}}
\def\fib        {\mathrm{fib}}

\newcommand{\kk}{\Bbbk}

\newcommand{\bt}{\boxtimes}
\newcommand{\btC}{\bt_{\C}}
\newcommand{\btD}{\bt_{\D}}
\newcommand{\id}{\textnormal{\textsf{Id}}}
\newcommand{\Rep}{\textnormal{\textsf{Rep}}}

\newcommand{\Hom}{{\sf Hom}}
\newcommand{\unit}{\mathds{1}}

\newcommand{\Vect}{\mathsf{Vec}}
\newcommand{\uHom}{\underline{\sf Hom}}

\newcommand{\Rex}{{\sf Rex}}

\newcommand{\ra}{{\sf ra}}

\newcommand{\rra}{{\sf rra}}

\newcommand{\tl}{\,\triangleright\,}

\newcommand{\tr}{\,\triangleleft\,}

\newcommand{\op}{\textnormal{op}}

\newcommand{\piv}{\textnormal{piv}}
\newcommand{\sph}{\textnormal{sph}}

\def\HH            {{}^{\#\!\#\!\!}} 

\newcommand{\Fun}{\textnormal{Fun}}
\newcommand{\Aut}{{\bf Aut}}
\newcommand{\MonAut}{{\bf MonAut}}

\newcommand{\BrPic}{{\bf BrPic}}

\newcommand{\Ext}{{\bf Ext}}

\newcommand{\MonFun}{{\bf MonFun}}
\newcommand{\PivExt}{{\bf PivExt}}
\newcommand{\SphExt}{{\bf SphExt}}

\newcommand{\PivBrPic}{{\bf PivBrPic}}
\newcommand{\SphBrPic}{{\bf SphBrPic}}

\newcommand{\Inv}{{\mathrm{Inv}}}
\newcommand{\bG}{\bf G}

\newcommand{\uG}{\underline{G}}
\newcommand{\uuG}{\,\underline{\uG}\,}

\newcommand{\tfp}{\widetilde{\mathfrak{p}}}
\newcommand{\tfq}{\widetilde{\mathfrak{q}}}

\newcommand{\oC}{\overline{\mathcal{\C}}}
\newcommand{\oD}{\overline{\mathcal{\D}}}

\newcommand{\odS}{\overline{\dS}}

\definecolor{darkgreen}{RGB}{55,138,0}
\definecolor{burntorange}{RGB}{180,85,0}
\definecolor{navyblue}{RGB}{18,40,180}
\definecolor{cyan(process)}{rgb}{0.0, 0.6, 1.0}

%% file: main.bbl
\begin{thebibliography}{CGPMV23}

\bibitem[BW99]{barrett1999spherical}
John~W Barrett and Bruce~W Westbury.
\newblock Spherical categories.
\newblock {\em Advances in Mathematics}, 143(2):357--375, 1999.

\bibitem[CGPMV23]{costantino2023non}
Francesco Costantino, Nathan Geer, Bertrand Patureau-Mirand, and Alexis Virelizier.
\newblock Non compact (2+ 1)-tqfts from non-semisimple spherical categories.
\newblock {\em arXiv preprint arXiv:2302.04509}, 2023.

\bibitem[DGNO10]{drinfeld2010braided}
Vladimir Drinfeld, Shlomo Gelaki, Dmitri Nikshych, and Victor Ostrik.
\newblock On braided fusion categories {I}.
\newblock {\em Selecta Mathematica}, 16:1--119, 2010.

\bibitem[DN13]{davydov2013picard}
Alexei Davydov and Dmitri Nikshych.
\newblock {The Picard crossed module of a braided tensor category}.
\newblock {\em Algebra \& Number Theory}, 7(6):1365--1403, 2013.

\bibitem[DN21]{davydov2021braided}
Alexei Davydov and Dmitri Nikshych.
\newblock Braided {P}icard groups and graded extensions of braided tensor categories.
\newblock {\em Selecta Math. (N.S.)}, 27(4):Paper No. 65, 87, 2021.

\bibitem[DSPS18]{douglas2018dualizable}
Christopher Douglas, Christopher Schommer-Pries, and Noah Snyder.
\newblock Dualizable tensor categories.
\newblock {\em Memoirs of the American Mathematical Society}, 2018.

\bibitem[DSPS19]{douglas2019balanced}
Christopher~L Douglas, Christopher Schommer-Pries, and Noah Snyder.
\newblock The balanced tensor product of module categories.
\newblock {\em Kyoto Journal of Mathematics}, 59(1):167--179, 2019.

\bibitem[EGNO15]{etingof2015tensor}
Pavel Etingof, Shlomo Gelaki, Dmitri Nikshych, and Victor Ostrik.
\newblock {\em Tensor categories}, volume 205.
\newblock American Mathematical Soc., 2015.

\bibitem[ENO04]{etingof2004analogue}
Pavel Etingof, Dmitri Nikshych, and Viktor Ostrik.
\newblock {An analogue of Radford's $S^4$ formula for finite tensor categories}.
\newblock {\em International Mathematics Research Notices}, 2004(54):2915--2933, 2004.

\bibitem[ENO05]{etingof2005fusion}
Pavel Etingof, Dmitri Nikshych, and Viktor Ostrik.
\newblock On fusion categories.
\newblock {\em Annals of mathematics}, pages 581--642, 2005.

\bibitem[ENO10]{etingof2010fusion}
Pavel Etingof, Dmitri Nikshych, and Victor Ostrik.
\newblock Fusion categories and homotopy theory.
\newblock {\em Quantum topology}, 1(3):209--273, 2010.

\bibitem[EO04]{etingof2004tensor}
Pavel Etingof and Viktor Ostrik.
\newblock Finite tensor categories.
\newblock {\em Mosc. Math. J.}, 4(3):627--654, 782--783, 2004.

\bibitem[FGJS25]{spherical2022}
J{\"u}rgen Fuchs, C{\'e}sar Galindo, David Jaklitsch, and Christoph Schweigert.
\newblock {Spherical Morita contexts and relative Serre functors}.
\newblock {\em Kyoto Journal of Mathematics}, 65(3):537 -- 594, 2025.

\bibitem[FSS20]{fuchs2020eilenberg}
J{\"u}rgen Fuchs, Gregor Schaumann, and Christoph Schweigert.
\newblock Eilenberg-{W}atts calculus for finite categories and a bimodule {R}adford ${S}^4$ theorem.
\newblock {\em Transactions of the American Mathematical Society}, 373(1):1--40, 2020.

\bibitem[FY89]{freyd1989braided}
Peter~J Freyd and David~N Yetter.
\newblock Braided compact closed categories with applications to low dimensional topology.
\newblock {\em Advances in mathematics}, 77(2):156--182, 1989.

\bibitem[GJS22]{gjs2022}
Cesar Galindo, David Jaklitsch, and Christoph Schweigert.
\newblock {Equivariant Morita theory for graded tensor categories}.
\newblock {\em Bull. Belg. Math. Soc. Simon Stevin}, 29(2):145--171, 2022.

\bibitem[GM12]{galindo2012module}
C{\'e}sar Galindo and Mart{\'\i}n Mombelli.
\newblock Module categories over finite pointed tensor categories.
\newblock {\em Selecta Mathematica}, 18(2):357--389, 2012.

\bibitem[GN08]{gelaki2008nilpotent}
Shlomo Gelaki and Dmitri Nikshych.
\newblock Nilpotent fusion categories.
\newblock {\em Advances in Mathematics}, 217(3):1053--1071, 2008.

\bibitem[HSV17]{hessevalentino}
Jan Hesse, Christoph Schweigert, and Alessandro Valentino.
\newblock Frobenius algebras and homotopy fixed points of group actions on bicategories.
\newblock {\em Theory Appl. Categ.}, 32:Paper No. 18, 652--681, 2017.

\bibitem[JY25]{jakyad2025tensorfrobenius}
David Jaklitsch and Harshit Yadav.
\newblock $\otimes$-{F}robenius functors and exact module categories.
\newblock {\em arXiv preprint {$\mathtt{arXiv:2501.16978}$}}, 2025.

\bibitem[M{\"u}g03]{muger2003subfactors}
Michael M{\"u}ger.
\newblock {From subfactors to categories and topology I: Frobenius algebras in and Morita equivalence of tensor categories}.
\newblock {\em Journal of Pure and Applied Algebra}, 180(1-2):81--157, 2003.

\bibitem[NS07]{MR2381536}
Siu-Hung Ng and Peter Schauenburg.
\newblock Higher {F}robenius-{S}chur indicators for pivotal categories.
\newblock In {\em Hopf algebras and generalizations}, volume 441 of {\em Contemp. Math.}, pages 63--90. Amer. Math. Soc., Providence, RI, 2007.

\bibitem[RT90]{reshetikhin1990ribbon}
Nicolai~Yu Reshetikhin and Vladimir~G Turaev.
\newblock Ribbon graphs and their invaraints derived from quantum groups.
\newblock {\em Communications in Mathematical Physics}, 127(1):1--26, 1990.

\bibitem[Sch13]{schaumann2013traces}
Gregor Schaumann.
\newblock Traces on module categories over fusion categories.
\newblock {\em Journal of Algebra}, 379:382--425, 2013.

\bibitem[Sch15]{schaumann2015pivotal}
Gregor Schaumann.
\newblock Pivotal tricategories and a categorification of inner-product modules.
\newblock {\em Algebras and Representation Theory}, 18(6):1407--1479, 2015.

\bibitem[Shi15]{shimizu2015pivotal}
Kenichi Shimizu.
\newblock The pivotal cover and {F}robenius--{S}chur indicators.
\newblock {\em Journal of Algebra}, 428:357--402, 2015.

\bibitem[Shi23a]{shimizu2019relative}
Kenichi Shimizu.
\newblock Relative serre functor for comodule algebras.
\newblock {\em Journal of Algebra}, 634:237--305, 2023.

\bibitem[Shi23b]{shimizu2023ribbon}
Kenichi Shimizu.
\newblock {Ribbon structures of the Drinfeld center of a finite tensor category}.
\newblock {\em Kodai Mathematical Journal}, 46(1):75--114, 2023.

\bibitem[TY98]{tambarayamagami}
Daisuke Tambara and Shigeru Yamagami.
\newblock Tensor categories with fusion rules of self-duality for finite abelian groups.
\newblock {\em J. Algebra}, 209(2):692--707, 1998.

\bibitem[Yad23]{yadav2023unimodular}
Harshit Yadav.
\newblock On unimodular module categories.
\newblock {\em Advances in Mathematics}, 432:109264, 2023.

\end{thebibliography}
